\documentclass[12pt,reqno,a4paper]{amsart}            
\usepackage{a4}
\usepackage[english]{babel}
\usepackage[utf8]{inputenc}
\usepackage{hyperref}

\usepackage{tikz}
\usepackage{tikz-cd}
\usetikzlibrary{math}
\usetikzlibrary{matrix, arrows, backgrounds, cd, patterns}

\usepackage{xifthen} 

\usepackage[totalwidth=450pt, totalheight=735pt]{geometry}
\usepackage{color}
\usepackage{bm}

\usepackage{scrlayer-scrpage}
\usepackage{scrtime}     
\usepackage{ifthen}       

\usepackage{graphicx}                

\usepackage{amsmath, amssymb}
\usepackage{amsthm}

\usepackage{accents}     
\usepackage{mathtools}  
\usepackage{bbm}         

\usepackage{mathrsfs}    
\renewcommand \mathcal \mathscr

\numberwithin{equation}{section}

\newcounter{myenumi}
\newenvironment{myenumerate}[1]{
\begin{list}{\indent(\themyenumi) }
  {\renewcommand{\themyenumi}{#1{myenumi}}
    \usecounter{myenumi}
    \setlength{\topsep}{0em}
    \setlength{\itemsep}{0em}
    \setlength{\leftmargin}{0em}
    \setlength{\labelwidth}{0em}
    \setlength{\labelsep}{0em}}
  }
  {
  \end{list}
  }

\newcommand{\itemref}[1]{\eqref{#1}}

\newcommand{\myfont}{\sffamily}

\newtheoremstyle{mythmstyle}
  {\topsep}
  {\topsep}
  {\itshape}
  {}
  {\bfseries \myfont}
  {.}
  {.5em}
  {}

\newtheoremstyle{mydefstyle}
  {\topsep}
  {\topsep}
  {\normalfont}
  {}
  {\bfseries \myfont}
  {.}
  {.5em}
  {}

\swapnumbers                    

\theoremstyle{mythmstyle}       

\newtheorem{theorem}{Theorem}[section]

\newtheorem{proposition}[theorem]{Proposition}
\newtheorem{lemma}[theorem]{Lemma}
\newtheorem{corollary}[theorem]{Corollary}

\newcounter{intro}

\theoremstyle{mydefstyle}        

\newtheorem{definition}[theorem]{Definition}
\newtheorem{assumption}[theorem]{Assumption}
\newtheorem{example}[theorem]{Example}
\newtheorem{remark}[theorem]{Remark}
\newtheorem{remarks}[theorem]{Remarks}
\newtheorem*{remark*}{Remark}

\expandafter\let\expandafter\oldproof\csname\string\proof\endcsname
\let\oldendproof\endproof
\renewenvironment{proof}[1][\bfseries\myfont\proofname]{%
  \oldproof[\bfseries \myfont #1]%
}{\oldendproof}

\makeatletter
\renewcommand\section{\@startsection{section}{1}%
  \z@{.7\linespacing\@plus\linespacing}{.5\linespacing}%
  {\Large\myfont\bfseries}}
\renewcommand\subsection{\@startsection{subsection}{2}%
  \z@{-.5\linespacing\@plus-.7\linespacing}{.5\linespacing}%
  {\large\myfont\bfseries}}
\renewcommand\subsubsection{\@startsection{subsubsection}{3}%
  \z@{.5\linespacing\@plus.7\linespacing}{-.5em}%
  {\myfont\bfseries}}
\renewenvironment{abstract}{%
  \ifx\maketitle\relax
    \ClassWarning{\@classname}{Abstract should precede
      \protect\maketitle\space in AMS document classes; reported}%
  \fi
  \global\setbox\abstractbox=\vtop \bgroup
    \normalfont\Small
    \list{}{\labelwidth\z@
      \leftmargin3pc \rightmargin\leftmargin
      \listparindent\normalparindent \itemindent\z@
      \parsep\z@ \@plus\p@
      
    }%
    \item[\hskip\labelsep
      \myfont\bfseries
    \abstractname.]%
}{%
  \endlist\egroup
  \ifx\@setabstract\relax \@setabstracta \fi
}
\renewcommand\contentsnamefont{\myfont\bfseries}
\renewcommand\@starttoc[2]{\begingroup
  \setTrue{#1}%
  \par\removelastskip\vskip\z@skip
  \@startsection{}\@M\z@{\linespacing\@plus\linespacing}%
    {.5\linespacing}{
      \contentsnamefont}{#2}%
  \ifx\contentsname#2%
  \else \addcontentsline{toc}{section}{#2}\fi
  \makeatletter
  \@input{\jobname.#1}%
  \if@filesw
    \@xp\newwrite\csname tf@#1\endcsname
    \immediate\@xp\openout\csname tf@#1\endcsname \jobname.#1\relax
  \fi
  \global\@nobreakfalse \endgroup
  \addvspace{32\p@\@plus14\p@}%
  \let\tableofcontents\relax
}
\renewcommand\@settitle{\begin{center}%
  \baselineskip14\p@\relax
    \LARGE
    \bfseries
    \myfont
  \@title
  \end{center}%
}
\renewcommand\@setauthors{%
  \begingroup
  \def\thanks{\protect\thanks@warning}%
  \trivlist
  \centering\footnotesize \@topsep30\p@\relax
  \advance\@topsep by -\baselineskip
  \item\relax
  \author@andify\authors
  \def\\{\protect\linebreak}%
  \large
  \myfont\bfseries\authors
  \ifx\@empty\contribs
  \else
    ,\penalty-3 \space \@setcontribs
    \@closetoccontribs
  \fi
  \endtrivlist
  \normalfont\myfont\@setaddresses
  \endgroup
}
\renewcommand\@setaddresses{\par
  \nobreak \begingroup
\footnotesize
  \def\author##1{\nobreak\addvspace\bigskipamount}%
  \def\\{\unskip, \ignorespaces}%
  \interlinepenalty\@M
  \def\address##1##2{\begingroup
    \par\addvspace\bigskipamount\indent
    \@ifnotempty{##1}{(\ignorespaces##1\unskip) }%
    {
      \ignorespaces##2}\par\endgroup}%
  \def\curraddr##1##2{\begingroup
    \@ifnotempty{##2}{\nobreak\indent\curraddrname
      \@ifnotempty{##1}{, \ignorespaces##1\unskip}\/:\space
      ##2\par}\endgroup}%
  \def\email##1##2{\begingroup
    \@ifnotempty{##2}{\nobreak\indent\emailaddrname
      \@ifnotempty{##1}{, \ignorespaces##1\unskip}\/:\space
      \ttfamily##2\par}\endgroup}%
  \def\urladdr##1##2{\begingroup
    \def~{\char`\~}%
    \@ifnotempty{##2}{\nobreak\indent\urladdrname
      \@ifnotempty{##1}{, \ignorespaces##1\unskip}\/:\space
      \ttfamily##2\par}\endgroup}%
  \addresses
  \endgroup
}
\renewcommand\enddoc@text{\ifx\@empty\@translators \else\@settranslators\fi
}
\renewcommand\@secnumfont{\myfont\bfseries} 
\makeatother

\newcommand{\Sec}[1]{Section~\ref{sec:#1}}

\newcommand{\App}[1]{Appendix~\ref{app:#1}}

\newcommand{\Subsec}[1]{Subsection~\ref{ssec:#1}}

\newcommand{\Fig}[1]{Figure~\ref{fig:#1}}

\newcommand{\Figs}[2]{Figures~\ref{fig:#1} and~\ref{fig:#2}}
\newcommand{\FigS}[2]{Figures~\ref{fig:#1}--\ref{fig:#2}}

\newcommand{\Thm}[1]{Theorem~\ref{thm:#1}}

\newcommand{\Thms}[2]{Theorems~\ref{thm:#1} and~\ref{thm:#2}}

\newcommand{\Lem}[1]{Lemma~\ref{lem:#1}}

\newcommand{\LemS}[2]{Lemmata~\ref{lem:#1}--\ref{lem:#2}}
\newcommand{\Lemenum}[2]{Lemma~\ref{lem:#1}~(\ref{#2})}

\newcommand{\Cor}[1]{Corollary~\ref{cor:#1}}

\newcommand{\Cors}[2]{Corollaries~\ref{cor:#1} and~\ref{cor:#2}}

\newcommand{\Prp}[1]{Proposition~\ref{prp:#1}}
\newcommand{\Prps}[2]{Propositions~\ref{prp:#1} and~\ref{prp:#2}}

\newcommand{\Rem}[1]{Remark~\ref{rem:#1}}

\newcommand{\Remenum}[2]{Remark~\ref{rem:#1}~(\ref{#2})}

\newcommand{\Def}[1]{Definition~\ref{def:#1}}

\newcommand{\Ass}[1]{Assumption~\ref{ass:#1}}
\newcommand{\Asss}[2]{Assumptions~\ref{ass:#1} and~\ref{ass:#2}}

\newcommand{\Assenum}[2]{Assumption~\ref{ass:#1}~(\ref{#2})}

\newcommand{\abs}[2][{}]{\lvert{#2}\rvert_{{#1}}}    
\newcommand{\abssqr}[2][{}]{\lvert{#2}\rvert^2_{#1}} 
\newcommand{\bigabs}[2][{}]{\bigl\lvert{#2}\bigr\rvert_{#1}}     
\newcommand{\bigabssqr}[2][{}]{\bigl\lvert{#2}\bigr\rvert^2_{#1}}
\newcommand{\Bigabs}[2][{}]{\Bigl\lvert{#2}\Bigr\rvert_{#1}}     
\newcommand{\Bigabssqr}[2][{}]{\Bigl\lvert{#2}\Bigr\rvert^2_{#1}}

\newcommand{\normsymb}{\|}
\newcommand{\bignormsymb}[1]{#1\|}

\newcommand{\norm}[2][{}]{\normsymb{#2}\normsymb_{{#1}}}    
\newcommand{\normsqr}[2][{}]{\normsymb{#2}\normsymb^2_{#1}} 
\newcommand{\bignorm}[2][{}]{\bignormsymb{\bigl}{#2}\bignormsymb{\bigr}_{#1}}

\newcommand{\iprod}[3][{}]{\langle{#2},{#3}\rangle_{#1}}  
\newcommand{\bigiprod}[3][{}]{\bigl\langle{#2},{#3}\bigr\rangle_{#1}}

\newcommand{\set}[2]{\{ \, #1 \, | \, #2 \, \} }      
\newcommand{\bigset}[2]{\bigl\{ \, #1 \, \bigl|\bigr. \, #2 \, \bigr\} }
\newcommand{\Bigset}[2]{\Bigl\{ \, #1 \, \Bigl|\Bigr. \, #2 \, \Bigr\} }

\DeclareMathOperator*{\bigdcup}{\mathaccent\cdot{\bigcup}}
\DeclareMathOperator*{\dcup}   {\mathaccent\cdot\cup}

\newcommand{\map}[3]{ #1 \colon #2 \longrightarrow #3}    
\newcommand{\embmap}[3]{ #1 \colon #2 \hookrightarrow #3} 

\newcommand{\bd}  {\partial}          
\newcommand{\clo}[2][]{\overline{{#2}}^{#1}} 
\newcommand{\intr}[1]{\ring{{#1}}}    

\newcommand{\restr}[1]{{\restriction}_{#1}} 

\newcommand{\dd}    {\, \mathrm d}    

\DeclareMathOperator{\dom}    {dom}

\DeclareMathOperator{\id}     {id}   

\DeclareMathOperator{\Ric}    {Ric}

\DeclareMathOperator{\supp}   {supp}
\DeclareMathOperator{\vol}    {vol}

\newcommand{\specsymb} {\sigma} 

\newcommand{\spec}[2][{}]   {\specsymb_{\mathrm{#1}}(#2)}

\newcommand{\eps}{\varepsilon} 
\renewcommand{\phi}{\varphi}   
\renewcommand{\rho}{\varrho}   

\DeclareMathOperator{\myRe} {Re}
\renewcommand{\Re}     {\myRe}

\newcommand{\conj}[1]{\overline {#1}}

\newcommand{\R}{\mathbb{R}} 
\newcommand{\C}{\mathbb{C}} 
\newcommand{\N}{\mathbb{N}} 
\newcommand{\Z}{\mathbb{Z}} 
\newcommand{\Sphere}{\mathbb{S}} 

\newcommand{\1}{\mathbbm 1}                    

\newcommand{\e}{\mathrm e}  

\DeclareMathSymbol{\widetildesym}{\mathord}{largesymbols}{"65}

\newcommand{\wt}{\widetilde}           
\newcommand {\qf}[1]{\mathfrak{#1}}    

\newcommand{\HS}{\mathcal H}           
\newcommand{\HSaux}{\mathcal G}        

\newcommand{\Sobsymb} {\mathsf H} 
\newcommand{\Sobnsymb} {\ring{\mathsf H}} 

\newcommand{\Contsymb} {\mathsf C}     
\newcommand{\Lsymb}    {\mathsf L}     

\newcommand{\Sobspace}[1][1]{\Sobsymb^{#1}}
\newcommand{\Sobnspace}[1][1]{\Sobnsymb^{#1}}

\newcommand{\Contspace}[1][{}]{\Contsymb^{#1}}     
\newcommand{\Lpspace}[1][p]    {\Lsymb_{#1}}     
\newcommand{\Lsqrspace}    {\Lpspace[2]}     

\newcommand{\BdOpsymb} {\mathsf L}       
\newcommand{\BdOp}[2][{}]{\BdOpsymb_{#1}({#2})}

\newcommand{\Ci} [2][{}]{\Contspace [\infty]_{#1} ({#2})}
\newcommand{\Cci}[1]{\Ci[\mathrm c]{#1}}
\newcommand{\Cont}[2][{}]{\Contspace[#1]({#2})}

\newcommand{\Lp}[2][p]{\Lpspace [#1]({#2})} 

\newcommand{\Lsqr}[2][{}]{\Lsqrspace^{#1}({#2})} 
\newcommand{\Sob}[2][1]{\Sobspace [#1]({#2})}         
\newcommand{\Sobn}[2][1]{\Sobnspace [#1]({#2})}  

\newcommand{\Sobx}[3][1]{\Sobspace [#1]_{{#2}}({#3})} 

\newcommand{\Dir}{{\mathrm D}}              
\newcommand{\laplacian}[1]{\Delta_{{#1}}}
\newcommand{\Err}{\mathrm O}

\newcommand {\loc}{\mathrm{loc}}

\newcommand{\spacetext}[2]{\hspace*{#1}\text{#2}\hspace*{#1}}
\newcommand{\quadtext}[1]{\spacetext{1em}{#1}}
\newcommand{\qquadtext}[1]{\spacetext{2em}{#1}} 

\newcommand{\myparagraph}[1]{\noindent\textbf{\myfont{#1}}}

\newcommand{\eucl}{\mathrm{eucl}}
\newcommand{\can}{\mathrm{can}}  

\newcommand{\Cellreg}{C_{\mathrm{ell.reg}}}
\newcommand{\cellreg}{c_{\mathrm{ell.reg}}}
\newcommand{\Cext}{C_{\mathrm{ext}}}
\newcommand{\Cyl}{C}                     
\newcommand{\Mu}{\mathrm M}              
\newcommand{\itr}{\mathrm{int}}         
\newcommand{\exr}{\mathrm{ext}}         

\newcommand{\normder}{\mathrm d_{\mathrm n}}
\newcommand{\deltaBall}[1]{\delta_{\mathrm{ball},{#1}}}
\newcommand{\deltaHarm}[1]{\delta_{\mathrm{harm},{#1}}}
\newcommand{\deltaHandle}[1]{\delta_{\mathrm{handle},{#1}}}
\newcommand{\deltaAntisym}[1]{\delta_{\mathrm{antisym},{#1}}}

\newcommand{\OptNonConc}{C} %
\newcommand{\OptNonConcEucl}{C^{\mathrm{eucl}}} 
\newcommand{\OptSobTr}{C'} %
\newcommand{\OptSobTrEucl}{C^{\prime \mathrm{eucl}}} %
\newcommand{\Cnbhd}{C_{\mathrm{nbhd}}} %

\begin{document}

\title[Manifolds with many wormholes]%
{Manifolds with many small wormholes: norm resolvent and spectral
  convergence}

\author{Colette Ann\'e}%
\address{Laboratoire de Math\'ematiques Jean
  Leray, CNRS -- Nantes Universit\'e , Facult\'e des Sciences, BP 92208,
  44322 Nantes, France}
\email{colette.anne@univ-nantes.fr}

\author{Olaf Post}
\address{Fachbereich 4 -- Mathematik,
  Universit\"at Trier,
  54286 Trier, Germany}
\email{olaf.post@uni-trier.de}
\date{\today, \thistime, \emph{File:} \texttt{\jobname.tex}}

\begin{abstract}
  We present results concerning the norm convergence of resolvents for
  wild perturbations of the Laplace-Beltrami operator.  This article
  is a continuation of our analysis on wildly perturbed manifolds
  presented in~\cite{anne-post:21}.  We study here manifolds with an
  increasing number of small (i.e., short and thin) handles added.
  The handles can also be seen as \emph{wormholes}, as they connect
  different parts being originally far away.  We consider two
  situations: if the small handles are distributed too sparse the
  limit operator is the unperturbed one on the initial manifold, the
  handles are \emph{fading}.  On the other hand, if the small handles
  are dense in certain regions the limit operator is the
  Laplace-Beltrami operator acting on functions which are identical on
  the two parts joined by the handles, the handles hence produce
  \emph{adhesion}.  Our results also apply to non-compact manifolds.
  Our work is based on a norm convergence result for operators acting
  in varying Hilbert spaces described in the book~\cite{post:12} by
  the second author.
\end{abstract}


\maketitle

\section{Introduction}
\label{sec:intro}
In this article, we continue our study of norm convergence of the
resolvents of Laplacians on manifolds with \emph{wild perturbations}
initiated in~\cite{anne-post:21}.  \emph{Wild perturbations} refers
here to increase the complexity of topology, a terminology already
used in~\cite{rauch-taylor:75}.  In the context of homogenisation,
some authors call such spaces \emph{manifolds with complicated} or
\emph{complex microstructure}
(cf.~\cite{boutet-de-monvel-khruslov:95,boutet-de-monvel-khruslov:97,%
  bck:97,boutet-de-monvel-khruslov:98,%
  khrabustovskyi:08,khrabustovskyi-stephan:08}).

In~\cite{anne-post:21} we have studied the convergence of the
Laplace-Beltrami operator on manifolds (not necessarily compact) under
the removal of tiny obstacles, such as many small balls.  Here, we
study in the same spirit the perturbation of adding many thin and
short handles.  As in~\cite{anne-post:21} we look at two situations:
\emph{fading}, which means that in the limit one does not see the
handles, and \emph{adhering} where the handles or wormholes change the
limit object: in the adhering case, two isometric parts of the space
joint by many thin handles with a length tending to zero are
identified.  For the effect of \emph{long} thin handles, i.e., handles
shrinking to a one-dimensional interval of \emph{positive} length, we
refer e.g.\ to~\cite{anne:87,anne-colbois:95,khrabustovskyi:13}.

\subsection{Main results and examples}

We consider an $m$-dimensional (for convenience) complete Riemannian
manifold $(X,g)$ of \emph{bounded geometry} with $m \ge 2$, remove
many small balls $B_\eps=\bigcup_{p \in I_\eps} B_\eps(p)$ of radius
$\eps$ and attach many thin handles
$\Cyl_\eps = \bigcup_{p \in I_\eps^-} \Cyl_\eps(p)$ to
$X_\eps=X \setminus B_\eps$.  The resulting manifold is called
$(M_\eps,g_\eps)$; the procedure is described in detail in
\Subsec{many.small.handles}, see also \Fig{mfd-handles}.  In the
following, we briefly describe the parameters of the attached handles
$\Cyl_\eps(p)\cong [0,\ell_\eps] \times \eps \Sphere^{m-1}$, how to
attach them and the location of the points $p \in I_\eps \subset X$:
\begin{itemize}
\item the \emph{(handle) radius} $\eps$;

\item the \emph{(handle) length} $\ell_\eps$;
\item the \emph{centres of balls} $p \in I_\eps \subset X$ where the
  handles are attached to the original manifold $X$: we assume that
  $I_\eps=I_\eps^-\dcup I_\eps^+$ (disjoint union) and that there is a
  bijection $\map {\bar \cdot}{I_\eps^-}{I_\eps^+}$; the handles are
  indexed by points $p \in I_\eps^-$; and the handles are attached at
  the boundary of $X \setminus B_\eps(p)$ and
  $X \setminus B_\eps(\bar p)$;
\item the balls are \emph{$\eps$-separated} (i.e.\ $d(p,q)\ge 2\eps$
  for $p,q \in I_\eps$ and $p \ne q$) and \emph{N-uniformly
    $\eta_\eps$-covered}, i.e., there is $N \in \N$ (independent of
  $\eps$) such that each point in the union of the family
  $(B_{\eta_\eps}(p))_{p \in I_\eps}$ is contained in maximal $N$
  balls, see~\eqref{eq:uni.loc.bdd}.  We call $\eta_\eps$ the
  \emph{uniform cover distance}.
\end{itemize}
As in the case of manifolds with obstacles in our previous
article~\cite{anne-post:21}, we describe two opposite situations:
\begin{itemize}
\item \myparagraph{Fading handles:} %
  here the effect of the handles or wormholes is not seen, i.e., the
  limit operator is the Laplacian on the original manifold, provided
  \begin{equation*}
    \eta_\eps=\eps^\alpha, \quad
    \ell_\eps=\eps^\lambda
    \quadtext{with}
    \begin{cases}
      0 \le \alpha \le \frac{m-2}m \text{ and } 0<\lambda < 1 \quad\text{or}\\
      \frac{m-2}m \le \alpha < 1 \text{ and } 0<\lambda < (m-1)-m \alpha.
    \end{cases}
  \end{equation*}
  see \Thm{handles1} and \Cor{handles1}.  If we want faster shrinking
  handle lengths, i.e., with $\lambda\ge 1$, we need a stricter
  condition on $\alpha$, namely $\alpha<1/3$ if $m=3$ and $\alpha<1/2$
  if $m \ge 4$, see \Thm{handles0} and \Cor{handles0}.  The case $m=2$
  is special, see the precise statements in the mentioned results.

\item \myparagraph{Adhering handles:} %
  in \Thm{handles3} and \Cor{handles3}, we assume that the points
  $p \in I_\eps^\pm $ are becoming denser in two open subsets
  $\Omega^\pm$ of $X$.  Moreover, we assume that $\Omega^-$ is
  isometric with $\Omega^+$ and that (for simplicity here)
  $\ell_\eps=\eps$.  Then, in the limit, the handles or wormholes
  ``\emph{adhere}'' the two sets $\Omega^-$ and $\Omega^+$ provided
  \begin{equation*}
    \frac{\eta_\eps^m}{\wt \eps \eps^{m-2}} \to 0,
    \quad
    \frac\eps{\eta_\eps}\to 0,
  \end{equation*}
  and some other technical assumptions, see \Subsec{adhering.handles}
  for details and \Fig{identified-mfd} for a visualisation.  Here,
  $\wt \eps$ can be chosen $\wt\eps =\Err(\eps^\alpha)$ in the case of
  flat identified parts (\Lem{flat.example}), or $\wt \eps=1$ if all
  parts are identified ($\Omega^+\dcup \Omega^-=X$, see
  \Cor{handles4}); in particular, we need
  \begin{align*}
    \frac{m-2}{m-1} < \alpha < 1
    \quadtext{resp.}
    \frac{m-2}m < \alpha < 1;
  \end{align*}
  in \Figs{par-range}{par-range-simple}, the above parameter region
  (for $\lambda=1$) is given by the line segment between $B$ and
  $\wt E$ resp.\ $B$ and $E$, see also these figures for the precise
  parameter region.
\end{itemize}
In all three theorems and its corollaries, we show convergence of the
energy forms leading to a generalised norm resolvent convergence of
the associated Laplacians: Since the perturbation changes the spaces
on which the operators act, we define a ``distance'' $\delta_\eps$
between two energy forms $\qf d_\eps$ and $\qf d_0$, acting in
different Hilbert spaces, called \emph{$\delta_\eps$-quasi-unitary
  equivalence}.  The generalised abstract theory (see \App{main.tool})
assures that if $\delta_\eps \to 0$, then a \emph{generalised norm
  resolvent convergence} holds for the associated Laplace-like
operators on varying spaces (\Prp{quasi-uni}).  This powerful tool and
many consequences (like convergence of eigenvalues, eigenfunctions,
functions of the operators such as spectral projections, the heat
operator, see e.g.~\Thm{spectrum}) was first introduced by the second
author in~\cite[Appendix]{post:06} and is explained in detail in the
monograph~\cite{post:12}.  For a related concept of Weidmann
(see~\cite[Sec.~9.3]{weidmann:00}) and its comparison with the one
used here we refer to~\cite{post-zimmer:22}
or~\cite{post-zimmer:pre24} and references therein.

Let us stress that we do not need a compactness assumption on the
space or on the resolvents as in some of the previous works; note that
these works mostly show some sort of \emph{strong} resolvent
convergence (cf.\ \Subsec{previous}).  Moreover, our approach allows
precise estimates on the convergence speed implying the same
convergence speed for the spectra of the resolvents, see
\Thm{spectrum}.

\subsection{Previous works}
\label{ssec:previous}

Handles added to a manifold are sometimes also called \emph{wormholes}
according to John Wheeler, cf.~\cite{wheeler:62}: one might travel in
space-time via a ``shortcut'' from one part of the space-time to another
one like a worm in an apple.

To the best of our knowledge, the study of adding handles to a
manifold was initiated by Chavel and Feldman~\cite{chavel-feldman:81}
who considered the ``fading'' case (in our notation), i.e., they gave
a sufficient condition under which one has convergence of the
eigenvalues of a manifold with a thin and short handle attached
towards the eigenvalues of the original compact manifold.  In
particular, Chavel and Feldman concluded that one cannot hear
asymptotically the topology of the manifold (the eigenvalues converge,
but the topology is different).  A sufficient condition for the
eigenvalue convergence given in~\cite{chavel-feldman:81} is a uniform
lower bound on an isoperimetric constant of the (compact) manifold
with handles.  Under this condition, the Dirichlet spectrum of the
handle has lower bound of order $\eps^{-2}$ and volume of order
$\eps^m$.  In our notation this means that $\lambda \ge 1$, i.e., the
handle length has to shrink at least of order $\eps$.  Our analysis
shows that the uniform lower bound on the isoperimetric constant is
not necessary as we still have fading handles (not seen in the
spectrum) for $0<\lambda<1$ and $\alpha=0$.

In subsequent works, the first author of this paper studied
in~\cite{anne:87,anne-colbois:95} thin handles shrinking to a
one-dimensional interval of positive length, which contribute to the
limit spectrum via its Dirichlet spectrum.  In~\cite{anne-colbois:95},
the authors also considered the spectra of Laplacians on differential
forms.

Maybe one of the first papers dealing with spaces of increasing
topological complexity and convergence of related Laplacians is the
work of Rauch and Taylor~\cite{rauch-taylor:75}, where Euclidean
domains with (many) small obstacles are considered and strong
resolvent convergence of Dirichlet or Neumann Laplacians is shown.
Manifolds with many thin handles have been treated also in
\cite{boutet-de-monvel-khruslov:95,boutet-de-monvel-khruslov:97,%
  notarantonio:98,boutet-de-monvel-khruslov:98,%
  dmgm:01,khrabustovskyi:08,khrabustovskyi-stephan:08,khrabustovskyi:09}.
In particular, a similar problem to ours is treated in
~\cite{boutet-de-monvel-khruslov:95,boutet-de-monvel-khruslov:97},
where the authors consider manifolds connected by an increasing number
of small and short handles (among related spaces).  They show that in
the critical case and below (in our notation, $\eta_\eps=\eps^\alpha$
with $\alpha \le (m-2)/m$ if $m \ge 3$ and $\eps$-homothetic handles,
i.e., $\lambda=1$), solutions of the heat equations converge (in a
strong sense) to a homogenised version with certain parts of the
spaces identified.  In~\cite{boutet-de-monvel-khruslov:98} it is shown
that harmonic vector fields converge towards a homogenised solution.

The question with many handles has also been studied via
$\Gamma$-convergence in the preprints~\cite{notarantonio:98,dmgm:01}.
As in our case, an increasing sequence of handles is attached to a
compact manifold (in~\cite{notarantonio:98} it is a sphere,
in~\cite{dmgm:01} one connects two isometric manifolds via thin
handles).  The situation resembles again our adhering case without any
ambient (unidentified) parts.  Note that $\Gamma$-convergence is
somehow equivalent with \emph{strong} resolvent convergence; while our
results imply \emph{norm} resolvent convergence and estimates on the
convergence speed.

Khrabustovskyi~\cite{khrabustovskyi:13} treated a situation similar to
ours (but only convergence of spectra is shown): he attaches an
increasing number of small handles to two copies of a given manifold.
His situation resembles our adhering case, but note that our model
also allows parts of the manifold which are not identified.
\Thm{handles4} and \Cor{handles4} describe the situation closest to
the one of~\cite[Thm.~2.6]{khrabustovskyi:13}; and we are able to
cover a parameter range $(\alpha,\lambda)$ closest to the optimal one
of~\cite[Thm.~2.6]{khrabustovskyi:13}.  In particular for dimensions
$m \in \{2,3,4\}$ and handle lengths shrinking as $\eps^\lambda$ with
$0<\lambda <1$, we generalise Khrabustovskyi's work to generalised
norm resolvent convergence with estimates on the convergence speed.
Here, our result is stronger in the sense that we allow non-compact
manifolds and that our results imply (generalised) norm resolvent
convergence instead of just convergence of eigenvalues.  For a
detailed comparison of Khrabustovskyi's and our results, we refer to
\Rem{andrii}.  We will treat operator norm convergence for the
interesting homogenisation case (boundary between fading and adhering
case) and also for other interesting models analysed e.g.\
in~\cite{boutet-de-monvel-khruslov:97} in a forthcoming publication.

(Generalised) norm resolvent convergence for homogenisation problems
using methods as in this paper was first shown (to the best of our
knowledge) in~\cite{khrabustovskyi-post:18}, see also the references
therein and~\cite{dcr:18} for an alternative approach.  Note that in
most of the works using $\Gamma$-convergence or related concepts, one
has to assume that the underlying manifold is \emph{compact}.  We do
not need this assumption in our case here.

We have announced \Thm{handles1} and \Cor{handles1} in a weaker
version in~\cite[Thm.~5.1]{anne-post:20}: namely we stated
\Cor{handles1} only for $0<\lambda<1-2\alpha$.  Similarly,
\Thm{handles3} is announced in\cite[Thm~.6.1]{anne-post:20} again
under weaker assumptions.  The reason for the better estimates here is
a more detailed analysis of an optimal Sobolev trace estimate, see
\Rem{opt.const}.

\subsection{Structure of the article}

In \Subsec{form.laplacian} we collect some basic facts on manifolds
and energy forms.  In \Subsec{non-concentr} we recall the
non-concentrating property (\Def{non-concentr}) already used
in~\cite{anne-post:21}.  Moreover, we define in
\Subsec{many.small.handles} the space $M_\eps$ obtained from a
complete manifold $X$ by adding many handles of radius $\eps$ and
length $\ell_\eps$, and its natural energy form $\qf d_\eps$ and
Laplacian.  In \Sec{thm} we state the main results on fading handles
(\Thms{handles0}{handles1}) and on adhering handles (\Thm{handles3})
together with examples.  In \Sec{est.handles} we present some
estimates of the harmonic extension on the handles.  In \Sec{proofs}
we prove the main results in the fading case, and in
\Sec{gluing.two.parts} we prove the main result in the adhering case.
In \App{eucl.balls}, we provide some estimates on Euclidean balls and
calculate the asymptotic expansion of the optimal constant in a
Sobolev trace estimate in terms of two radii, see~\Prp{copt.asymp};
improving a previous estimate.  The improvement is crucial when
comparing our results with the ones of~\cite{khrabustovskyi:13}.  In
\App{mfds.bdd.geo} we recall the notion of bounded geometry
(\Def{bdd.geo}) and we compare norms on small balls in complete
Riemannian manifolds of bounded geometry with the Euclidean case.
Finally, \App{main.tool} contains the general concept of norm
convergence of operators acting in varying Hilbert spaces.

\subsection*{Acknowledgements}
We would like to thank Andrii Khrabustovskyi for helpful comments
concerning this article.  OP would like to thank the F\'ed\'eration
Recherche Math\'ematiques des Pays de Loire for the hospitality at the
\emph{Universit\'e de Nantes}
. CA would like to thank the \emph{University of Trier} for the
hospitality.

%
\section{Setting of the problem}
\label{sec:set}
%

\subsection{Energy forms, Laplacians and harmonic radius on a
  Riemannian manifold with bounded geometry}
\label{ssec:form.laplacian}

Let $(X,g)$ be a complete Riemannian manifold of dimension $m \ge 2$.
Under suitable assumptions, our results also hold for manifolds with
boundary or non-complete manifolds.  We sometimes indicate how
manifolds with boundary can be treated (e.g.\ in \Subsec{examples}).
The Riemannian measure induced by $g$ is denoted by $\dd g$. The
Riemannian measure allows us to define the Hilbert space $\Lsqr {X,g}$
of square-integrable (equivalence classes of) functions with norm
given by
\begin{equation*}
  \normsqr[\Lsqr {X,g}] u \coloneqq \int_X\abssqr u \dd g.
\end{equation*}
The \emph{energy form} associated with $(X,g)$ is defined by
\begin{equation*}
  \qf d_{(X,g)}(u) \coloneqq \int_X \abssqr[g]{d u} \dd g
\end{equation*}
for $u$ in the Sobolev space $\Sob {X,g}$, which can be defined as the
completion of smooth functions with compact support, under the
so-called \emph{energy norm} given by
\begin{equation*}
  \normsqr[\Sob {X,g}] u
  \coloneqq \int_X\bigl(\abssqr{u} + \abssqr[g]{d u} \bigr) \dd g.
\end{equation*}
Here, $du$ is a section into the cotangent bundle $T^*M$ and
$\abs[g]\cdot$ the Euclidean norm induced by $g$ on it.  Note that by
definition, $\qf d_{(X,g)}$ is a closed quadratic form with
$\dom \qf d_{(X,g)}=\Sob{X,g}$.  The \emph{Laplacian} $\Delta_{(X,g)}$
(in our notation $\Delta_{(X,g)}\ge 0$) associated with $(X,g)$ is the
operator associated with the energy form $\qf d_{(X,g)}$.  The
Laplacian is a self-adjoint, non-negative and unbounded operator and
hence introduces a scale of Hilbert spaces
$\HS^k \coloneqq\dom ((\laplacian {(X,g)}+1)^{k/2})$ with norms
$u \mapsto \norm[\Lsqr{X,g}]{(\laplacian {(X,g)}+1)^{k/2} u}$ for
$k \ge 0$.

We review briefly some facts already cited
in~\cite[Sec.~3]{anne-post:21} on manifolds of \emph{bounded geometry}
in \App{mfds.bdd.geo}.

\subsection{The non-concentrating property}
\label{ssec:non-concentr}

A main tool of our analysis is the property called
``non-concentrating'' in~\cite[Subsec..~3.3]{anne-post:21}.
\begin{definition}[non-concentrating property]
  \label{def:non-concentr}
  Let $(X,g)$ be a Riemannian manifold, $A \subset B \subset X$ and
  $\delta>0$.  We say that $(A,B)$ is
  $\delta$-\emph{non-concentrating} (of order $1$) if
  \begin{equation*}
    \norm[\Lsqr{A,g}] f
    \le \delta \norm[\Sob{B,g}] f
  \end{equation*}
  for all $f \in \Sob{B,g}$.
\end{definition}
Note that if $\wt B \supset B$, $\wt A \subset A$ and if $(A,B)$ is
$\delta$-non-concentrating, then also $(\wt A,\wt B)$ is
$\delta$-non-concentrating.

Typically, $A=A_\eps$ is an open subset of the manifold $X$ and
$\delta=\delta_\eps \to 0$ as $\eps \to 0$; the name
``non-concentrating'' comes from the fact that if $f=f_\eps$ is a
normalised eigenfunction with eigenvalue $\lambda_\eps$ bounded in
$\eps$, then $f_\eps$ cannot concentrate on $A_\eps$ as $\eps \to 0$.
The definition of ``$\delta$-non-concentrating'' allows us to quantify
how much a function $f \in \Sob B$ is not concentrated in
$A\subset B$. In~\cite[Rem.~3.8]{anne-post:21} further explanations on
this concept are given.  A related result can be found
in~\cite[Lem.~4.9]{marchenko-khruslov:06}.

Once we have the non-concentrating property, we can immediately
conclude a similar estimate for the derivatives:
\begin{proposition}[{see~\cite[Prp.~3.9]{anne-post:21}}]
  \label{prp:non-concentr2}
  Assume that $(A,B)$ is $\delta$-non-concentrating, then $(A,B)$ is
  $\delta$-non-concentrating of order $2$, i.e.,
  \begin{equation*}
    \norm[\Lsqr{A,g}] {df}
    \le \delta \norm[{\Sob[2]{B,g}}] f
  \end{equation*}
  for all $f \in \Sob[2]{B,g}$.
\end{proposition}
\begin{remark}[how to get back to graph norms]
  \label{rem:trick.sob2}
  The $\Sobspace[2]$-norm in \Prp{non-concentr2} can be applied to
  balls and summed to a global $\Sobspace[2]$-norm on a Riemannian
  manifold $(X,g)$.  \Prp{ell.reg} then allows us to pass from this
  $\Sobspace[2]$-norm to the graph norm of the corresponding Laplacian
  on $(X,g)$. We apply this argument only for $\eps$-independent
  spaces in order to avoid complications in tracing parameters in
  $\Cellreg=\Cellreg(X,g)$.
\end{remark}
We now want to show the non-concentrating property for a union of many
balls, cf.\ \Prp{0}.  In order to compare estimates on Euclidean balls
with ones on geodesic balls, we need a uniform lower bound on the
harmonic radius $r_0>0$, see \Prp{eucl.metric} and the text before.
For simplicity, we assume that $r_0 \le 1$.
\begin{definition}[$\eps$-separation and $N$-uniform $\eta$-cover]
  \label{def:sep-cov-numb}
  Let $r_0,\eta,\eps>0$ such that $\eps < \eta < r_0$.
  \begin{enumerate}
  \item A set $I \subset X$ is called \emph{$\eps$-separated} if for
    any two distinct points $p,q\in I$ one has $d(p,q)\geq 2\eps$.
  \item We say that $I$ has \emph{$\eta$-cover number
      $N \in \N^*:= \N\setminus\{0\}$} if
    \begin{equation}
      \label{eq:uni.loc.bdd}
      \bigabs {\bigset{p \in I}{x \in B_\eta(p)}}
      \le N \qquad\text{for all $x \in X$.}
    \end{equation}
    We say that $I$ is \emph{uniformly $\eta$-covered} or that $I$ has
    an \emph{uniform $\eta$-cover} if there is $N \in \N^*$ such that
    $I$ has $\eta$-cover number $N$.
  \item For $\eps \in (0,r_0)$ we denote
    \begin{equation*}
      B_\eps(I)=\bigcup_{p\in I} B_\eps(p)
    \end{equation*}
    the union of all open balls with centre in $I$.
  \end{enumerate}
  If $I$ is $\eps$-separated and $N$-uniformly $\eta$-covered, then we
  call $\eps$ the \emph{radius} and $\eta$ the \emph{$N$-uniform} (or
  simply) \emph{uniform cover distance} of $B_\eps:=B_\eps(I)$.
\end{definition}
Note that the union $B_\eps(I)$ is disjoint if $I$ is
$\eps$-separated.  The term ``uniform cover distance'' refers to how
close points in $I$ can be in order to still have control over the
cover number of the $\eta$-ball cover $B_\eta(I)$.

As an example for the non-concentrating property we show it for the
union of many balls.  The corresponding estimate in the Euclidean
space is provided in \Cor{ball.est}; as we have a better optimal
constant than in previous publications, we include a proof in the
appendix.
\begin{proposition}[non-concentrating property for union of balls]
  \label{prp:0}
  Let $(X,g)$ be a complete Riemannian manifold of bounded geometry
  (cf.\ \Def{bdd.geo}) and harmonic radius $r_0>0$.  Let
  $\eta \in (0,r_0)$ and $\eps \in (0,\eta)$.  Assume that $I$ is
  $\eps$-separated and has $\eta$-cover number $N \in \N^*$, then for
  all $f \in \Sob {B_\eta(I),g}$ we have
  \begin{align}
    \label{eq:prp.0}
    \norm[\Lsqr{B_\eps(I),g}] f
    &\le \OptNonConc_m(\eps,\eta)
      \norm[\Sob{B_\eta(I),g}] f,
  \end{align}
  where
  $\OptNonConc_m(\eps,\eta)= N^{1/2}K^{(m+1)/2}
  \OptNonConcEucl_m(\eps,\eta)$ with a positive constant $K$ depending
  only on the geometry of $(X,g)$ and where
  $\OptNonConcEucl_m(\eps,\eta)$ is given in~\eqref{eq:c.opt.eucl}.
\end{proposition}
\begin{proof}
  The estimates for Euclidean balls follow from~\eqref{eq:c.opt.eucl}
  and~\eqref{eq:c.opt'}.  Moreover, from \Cor{eucl.metric2} we obtain
  the extra factors $K^{m/4} K^{(m+2)/4}=K^{(m+1)/2}$ from the
  deviation of the manifold to the Euclidean case, cf.\
  \Cor{eucl.metric2}.  Passing from the sum over the individual
  squared norms over balls $B_\eta(p)$ to the union $B_\eta(I)$ gives
  the extra factor $N^{1/2}$ in the inequality.
\end{proof}
Estimate~\eqref{eq:prp.0} can be restated by saying that
$(B_\eps(I),B_\eta(I))$ is $\delta$-non-con\-cen\-tra\-ting (of order
$1$) with $\delta=\OptNonConc_m(\eps,\eta)$.  For the asymptotic
expansion of the optimal constant $\OptNonConcEucl_m(\eps,\eta)$ we
refer to \Cor{ball.est} in \App{eucl.balls}.

\begin{remark}[on the new constant]
  \label{rem:opt.const}
  With some efforts, we show in \App{eucl.balls} that
  \begin{align*}
    \OptNonConc_m(\eps,\eta)
    =\Err\Bigl(\Bigl(\frac{\eps^m}{\eta^m} +  \eps^2[-\log \eps]_2
    \Bigr)^{1/2}\Bigr),
  \end{align*}
  where $[\dots]_2$ appears only if $m=2$, and the order in $\eps$ and
  $\eta$ is optimal.  We previously estimated
  \begin{equation}
    \label{eq:bd.ball.est'}
    \OptNonConc_m(\eps,\eta)
    = \Err\Bigl(\Bigl(\frac {\eps^2}{\eta^2}\Bigl[-\log \frac \eps \eta \Bigr]_2
    \Bigr)^{1/2}\Bigr),
  \end{equation}
  in~\cite[Lem.~3.10]{anne-post:21}.  With our new optimal estimate of
  $\OptNonConc_m(\eps,\eta)$, we come much closer to similar results
  on manifolds with handles already proven
  by~\cite{khrabustovskyi:13}, see \Rem{andrii}.  If we insist on an
  estimate of the form
  $\OptNonConc_m(\eps,\eta)\le c_m(\eps/\eta)^\beta$ (for $m\ge 3$) as
  in our previous publications, e.g.~\cite[Lem.~3.10]{anne-post:21},
  then we must have $\beta=1$.  Note that a simple scaling argument
  leads only to the estimate~\eqref{eq:bd.ball.est'}, so we really
  need the more advanced arguments such as in \App{eucl.balls} for the
  better constant.
\end{remark}

\subsection{Manifolds with handles and their Laplacians}
\label{ssec:many.small.handles}
%

\subsubsection*{Definition of the handles}
\label{ssec:def.handles}

Let $(X,g)$ be a complete connected Riemannian manifold of dimension
$m\geq 2$ of \emph{bounded geometry} (cf.\ \Def{bdd.geo}).  In
particular, the harmonic radius $r_0>0$ is then positive (cf.\
\Prp{eucl.metric}).  For each $\eps \in (0,r_0)$, let
$\eta_\eps \in (0,r_0)$ such that $0 < \eps< \eta_\eps$.  We specify
the dependence of $\eta_\eps$ on $\eps$ later.  Moreover, let $I_\eps$
be an $\eps$-separated subset of $X$, i.e., if $p,q \in I_\eps$ with
$p \ne q$, then $d(p,q) \ge 2\eps$.  Let
\begin{equation*}
  B_\eps=\bigdcup_{p\in I_\eps}B_\eps(p)
  \qquadtext{and}
  X_\eps=X \setminus B_\eps.
\end{equation*}
Additionally, we assume that $I_\eps$ splits into two disjoint subsets
\begin{equation*}
  I_\eps = I^-_\eps \dcup I^+_\eps
\end{equation*}
such that there is a bijective map
$\map {\bar \cdot}{I_\eps^-}{I_\eps^+}$.  We call such a set
$I_\eps=I^-_\eps \dcup I^+_\eps$ a \emph{split set}.

\begin{figure}[h]
  \centering
  \scalebox{0.9}{\includegraphics{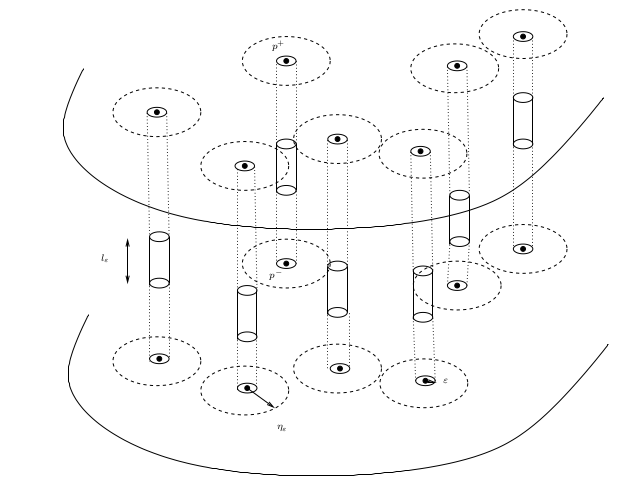}}

  \caption{The manifold $M_\eps$ obtained from $X$ (here the top and
    lower flat region) by removing many small balls $B_\eps(p)$ of
    radius $\eps$ by attaching handles $\Cyl_\eps(p)$.  The small
    dotted lines mean identification.}
    \label{fig:mfd-handles}
\end{figure}

If $(N,h)$ is a Riemannian manifold, we use the notation $\eps N$ for
some $\eps>0$ to denote the scaled Riemannian manifold $(N,\eps^2h)$.

For each $p \in I^-_\eps$ we specify a \emph{handle} or
\emph{cylinder} of radius $\eps \in (0,r_0)$ and length $\ell_\eps>0$
(again, we specify the dependency of $\ell_\eps$ on $\eps$ later),
namely we set
\begin{equation*}
  \Cyl_\eps(p)\coloneqq [0,\ell_\eps] \times \eps \Sphere \times\{p\},
\end{equation*}
i.e., the underlying space is
$[0,\ell_\eps]\times \Sphere \times \{p\}$ with metric
$g_\eps=\dd s^2 + \eps^2 g_\Sphere$.  Here, $g_\Sphere$ is the
standard metric on the $(m-1)$-dimensional sphere
$\Sphere \coloneqq \Sphere^{m-1}$.

We denote by
\begin{equation*}
  \bd^- \Cyl_\eps(p)
  \coloneqq\{0\}\times \Sphere\times \{p\}
  \qquadtext{resp.}
  \bd^+ \Cyl_\eps(p)
  \coloneqq\{\ell_\eps\}\times \Sphere\times \{p\}
\end{equation*}
the lower resp.\ upper boundary of $\Cyl_\eps(p)$ (see
\Fig{mfd-handles}).  Moreover, let
\begin{equation*}
  \Cyl_\eps
  = \bigdcup_{p \in I^-_\eps} \Cyl_\eps(p)
\end{equation*}
be the disjoint union of the isometric handles labelled by
$p \in I^-_\eps$.

We glue $\Cyl_\eps$ to $X_\eps=X\setminus B_\eps$ by identifying
$\bd^-\Cyl_\eps(p)$ with $\bd B_\eps(p)$ and $\bd^+\Cyl_\eps(p)$ with
$\bd B_\eps(\bar p)$ for each handle labelled by $p \in I^-_\eps$.  We
denote the resulting space by
\begin{equation*}
  M_\eps = (X_\eps \cup \Cyl_\eps)/ {\sim},
\end{equation*}
where $\sim$ denotes the identification.

\begin{remarks}
  \label{rem:ident.handle}
  \indent
  \begin{enumerate}
  \item
    \label{ident.handle.a}
    Note that there is some freedom how we identify the spheres of
    $\bd^\pm \Cyl_\eps(p)$ with the spheres $\bd B_\eps(p)$ and
    $\bd B_\eps(\bar p)$.  We fix these identifications, and it will
    play no further role in our analysis.

  \item
    \label{ident.handle.b}
    To make this identification isometric, we would have to suppose
    that the metric on $X$ is Euclidean near the points of $I_\eps$.
    This is a rather strict constraint as we allow the index set
    $I_\eps$ to vary as $\eps \to 0$.  Moreover, the manifold $M_\eps$
    is not smooth where the handles are glued, it is only a
    topological manifold.

    We show in \Prp{handle.smooth} that we can modify the metric $g$
    on $X_\eps$ into $g'_\eps$ in such a way that $(X_\eps,g_\eps')$
    is Euclidean on $B_{2\eps}$.  In particular, the corresponding
    forms are $\Err(\eps^a)$-quasi-unitary equivalent, where
    $a \in (0,1)$ is the local H\"older exponent of the metric, see
    \Prp{eucl.metric}.  In addition, we use this flattened manifold
    $(X_\eps, \wt g_\eps)$ and define a smooth manifold
    $(\wt M_\eps,\wt g_\eps)$ in \Prp{handle.smooth} such that the
    form $\qf d_\eps$ and the natural energy form on
    $(\wt M_\eps,\wt g_\eps)$ are $\Err(\eps^a)$-quasi-unitarily
    equivalent.
  \end{enumerate}
\end{remarks}
It will be more convenient to work with the $\eps$-independent space
$\Lsqr {\Cyl,g_\can}$, where
\begin{equation*}
  \Cyl = \bigdcup_{p \in I^-_\eps} \Cyl(p),
  \quad
  \Cyl(p)\coloneqq \Cyl_1 \times \{p\},
  \quadtext{and}
  \Cyl_1\coloneqq[0,1]\times \Sphere
\end{equation*}
using the unitary map
\begin{equation*}
  \Lsqr{\Cyl_\eps,g_\eps} \to \Lsqr {\Cyl,g_\can},
  \quad
  (h_{p})_{p\in I^-_\eps} \mapsto (\eps^{(m-1)/2}\ell_\eps \wt
  h_{p})_{p\in I^-_\eps}
  \quadtext{with}
  \wt h_p(s,\theta)=h_p(s \ell_\eps,\theta).
\end{equation*}
Here $g_\can=\dd s^2+g_{\Sphere}$ is the canonical metric on each
cylinder $[0,1]\times \Sphere$.  With this unitary map, we also have
\begin{equation*}
  \Lsqr{M_\eps,g_\eps}
  \cong \Lsqr{X_\eps,g} \oplus \Lsqr {\Cyl,g_\can}
  \cong \Lsqr{X_\eps,g} \oplus  (\Lsqr{\Cyl_1,g_\can})^{I^-_\eps}.
\end{equation*}

\begin{definition}[canonical energy form on manifold with handles]
  \label{def:qeps}
  The canonical energy form $\qf d_\eps$ of the manifold $M_\eps$ with
  handles is defined on the domain
  \begin{subequations}
    \begin{multline}
      \label{eq:qeps.a}
      \dom {\qf d}_\eps
      =\Bigset{U=(u,h)\in \Sob{X_\eps,g} \times
        \Sob{\Cyl,g_\can}}
      {\\
        \forall p \in I^-_\eps \colon \;
        h_p(0,\cdot) = \sqrt{\eps^{m-1}\ell_\eps} \cdot u_p (\eps,\cdot),
        \quad
        h_p(1,\cdot) = \sqrt{\eps^{m-1}\ell_\eps} \cdot u_{\bar p}(\eps,\cdot)},
    \end{multline}
    where $u_p(r,\theta)$ is the value of $u$ on $B_\eps(p)$ in
    spherical coordinates $(r,\theta)$ around $p \in I_\eps$.
    Moreover, we set
    \begin{equation*}
      \qf d_\eps(U)
      \coloneqq \int_{X_\eps}\abssqr{du} \dd g + \qf d_{\Cyl_\eps}(h),
    \end{equation*}
    where
    \begin{equation}
      \label{eq:qeps.c}
      \qf d_{\Cyl_\eps}(h)=\sum_{p \in I^-_\eps}\qf d_{\Cyl_\eps(p)}(h_p)
      \quadtext{and}
      \qf d_{\Cyl_\eps(p)}(h_p)=
      \int_{\Cyl_1} \Bigl( \frac1{\ell_\eps^2} \abssqr{\partial_1 h_p}
      +\frac1{\eps^2}\abssqr{d_{\Sphere} h_p}\Bigr) \dd g_\can.
    \end{equation}
    Here, $\partial_1$ is the derivative with respect to the first
    (longitudinal) variable and $d_{\Sphere}$ the exterior derivative
    with respect to the second variable $(s,\theta)\in \Cyl_1$.
  \end{subequations}
\end{definition}
It is easy to see that the energy form ${\qf d}_\eps$ is the natural
form on the manifold $M_\eps$, and that it is a closed form.  The
associated non-negative and self-adjoint operator $\Delta_\eps$ is
called the \emph{energy operator} or \emph{Laplacian} of $M_\eps$.
Its domain can be described explicitly as a collection of functions in
$\Sob[2]{X_\eps,g}$ resp.\ $\Sob[2]{\Cyl,g_\can}$ satisfying the
\emph{gluing} conditions of~\eqref{eq:qeps.a} together with conditions
involving the normal derivative on the different common boundaries,
see for instance~\cite{anne:87}.

As already mentioned, our definition of the space $M_\eps$ and its
canonical energy form $\qf d_\eps$ is close to the natural energy form
$\wt {\qf d}_\eps$ on a \emph{smooth} manifold $\wt M_\eps$ in the
sense of quasi-unitary equivalence (see \App{main.tool} for the
definition of this concept):
\begin{proposition}[modification to a smooth manifold]
  \label{prp:handle.smooth}
  Let $M_\eps$ and $\qf d_\eps$ as above. Then there is a smooth
  manifold $\wt M_\eps$ such that its natural energy form
  $\wt{\qf d}_\eps$ is $\wt \delta_\eps$-quasi-unitarily equivalent
  with $\qf d_\eps$ with $\wt \delta_\eps=\Err(\eps^a)$ as
  $\eps \to 0$, where $a \in (0,1)$ appears as local H\"older exponent
  in~\eqref{eq:eucl.metric.b}.
\end{proposition}
\begin{proof}
  We proceed in two steps.  First, we change the metric $g$ on each ball $B_{\eta_\eps}(p)$
  into
  \begin{equation*}
    g'_{p,\eps} \coloneqq\chi_{p,\eps} g_{\eucl,p} + (1-\chi_{p,\eps}) g,
  \end{equation*}
  where $\chi_{p,\eps}(x)\coloneqq\chi(d(p,x)/\eps)$ for
  $x \in B_{\eta_\eps}(p)$ and where $\map \chi {[0,\infty)}{[0,1]}$
  is a smooth function with $\chi(s)=1$ for $s \in [0,2]$ and
  $\chi(s)=0$ for $s \ge 3$.  Moreover, $g_{\eucl,p}$ is the Euclidean
  metric on $B_{r_0}(p)$ introduced in~\eqref{eq:eucl.met}.  Note that
  balls around $p$ are the same with respect to both metrics $g$ and
  $g'_\eps$, and that $(B_{2\eps}(p),g'_{p,\eps})$ is flat.

  Let $A'_{p,\eps}$ be the endomorphism measuring the deviation of
  $g_{p,\eps}'$ from $g$ on $B_{\eta_\eps}(p)$ (called \emph{relative
    distortion} of $g'_\eps$ from $g$ in~\cite[Sec.~5.2]{post:12}).
  In particular, we have
  \begin{align*}
    g(A'_{p,\eps} \xi,\xi)
    \coloneqq g_{p,\eps}'(\xi,\xi)
    &=\chi_{p,\eps} g_{\eucl,p}(\xi,\xi)+(1-\chi_{\eps,p})g(\xi,\xi)\\
    &=g\Bigl(\bigl(\chi_{p,\eps} A_{\eucl,p}+(1-\chi_{p,\eps})\id\bigr)\xi,\xi\Bigr)\\
    &=g\Bigl(\chi_{p,\eps}\bigl(A_{\eucl,p} - \id) + \id \bigr)\xi,\xi\Bigr),
  \end{align*}
  hence
  \begin{equation*}
    A'_{p,\eps} - \id = \chi_{p,\eps}(A_{\eucl,p} - \id)
  \end{equation*}
  on $B_{\eta_\eps}(p)$.  Moreover, we have
  \begin{equation*}
    \norm[\Cont{B_{3\eps}(p)}]{A'_{p,\eps} - \id}
    \le \norm[\Cont{B_{3\eps}(p)}]{A_{\eucl,p}- \id}
    \le m k 3^a \eps^a
  \end{equation*}
  uniformly in $p \in I_\eps$ using \Cor{eucl.metric}.  If we choose now
  $\eps \in (0,\eps_0]$ with $\eps_0=\min\{(2m k 3^a)^{-1/a},r_0/3\}$,
  then the above norm of the difference $A'_{p,\eps} -\id$ is smaller
  than or equal to $1/2$.

  As $g'_{p,\eps}$ equals $g$ outside $B_{\eta_\eps}(p)$, we can
  define a metric $g'_\eps$ on $X$ equal to $g'_{p,\eps}$ on
  $B_{\eta_\eps}(p)$ for each $p \in I_\eps$.  Note that $(X,g'_\eps)$
  is flat on $B_{2\eps}$.  We denote its relative distortion from $g$
  by $A'_\eps$, and on $B_{\eta_\eps}(p)$ it equals $A'_{p,\eps}$
  given above.  Now~\cite[Thm.~5.2.6]{post:12} applies, namely the
  canonical energy forms on $(X_\eps,g)$ and $(X_\eps,g'_\eps)$ are
  $\delta_\eps'=\Err(\eps^a)$-quasi-unitarily equivalent with
  \begin{align*}
    \delta_\eps'
    = 4 \norm[\Cont{X}]{A'_\eps - \id}
    \le 4m k 3^a \eps^a=\Err(\eps^a).
  \end{align*}

  In a second step, we smoothen the metric: Note that the Riemannian
  manifold $(M_\eps,g'_\eps)$ constructed as $(M_\eps, g_\eps)$, but
  with $(X_\eps,g)$ replaced by $(X_\eps, g'_\eps)$ is non-smooth at
  the gluing of the identification of the handles: in polar
  coordinates on $B_{2\eps}(p) \setminus B_\eps(p)$ (we drop now the
  dependence of $p$ in the notation), the metric is
  $g_\eps'=\dd s^2 + r_\eps(s)^2 g_{\Sphere}$ with $r_\eps(s)=s$ for
  $s \in [\eps,2\eps]$ and $r_\eps(s)=\eps$ for $s \in [0,\eps]$.  We
  change the non-smooth function $r_\eps$ into a smooth function
  $\wt r_\eps$ such that
  $\e^{-\eps} \le \wt r_\eps(s)/r_\eps(s) \le \e^\eps$ for all $s$.
  We call the resulting metric $\wt g_\eps$ which is of the form
  $\dd s^2+\wt r_\eps(s)^2 g_{\Sphere}$ on $X_\eps$ and we have
  defined a Riemannian manifold $(M_\eps,\wt g_\eps)$.  Again
  by~\cite[Thm.~5.2.6]{post:12}, the two canonical energy forms on
  $(M_\eps, g_\eps')$ and $(M_\eps,\wt g_\eps)$ are
  $\delta_\eps''$-quasi-unitarily equivalent with $\delta_\eps''$ of
  order $\eps$ provided $\eps \in (0,1/2]$.  Using the transitivity of
  quasi-unitary equivalence we obtain that $\qf d_\eps$ and
  $\wt {\qf d}_\eps$ are
  $14(\delta_\eps'+\delta_\eps'')$-quasi-unitary equivalent
  (cf.~\cite[Prp.~1.6]{post-simmer:19}, correcting the erroneous proof
  of Prp.~4.4.16
  of~\cite{post:12}.
\end{proof}

%
\section{Main Results}
\label{sec:thm}
%

We consider different settings (i.e., assumptions on the parameters
$\ell_\eps$ (the handle length), $\eta_\eps$ (the uniform cover
distance of the points $I_\eps^\pm$) and the set of points
$I^\pm_\eps$, where the handles are attached.  Recall from
\Subsec{def.handles} that $M_\eps=X_\eps \cup \Cyl_\eps/{\sim}$
denotes the manifold with many small handles
$\Cyl_\eps=\bigdcup_{p \in I_\eps^-} \Cyl_\eps(p)$ of length
$\ell_\eps>0$ and transversal radius $\eps>0$ attached to $X_\eps$,
where $X_\eps$ is obtained from $X$ by removing disjoint balls
$B_\eps=\bigdcup_{p \in I_\eps} B_\eps(p)$ of radius $\eps>0$.  Recall
that $r_0$ denotes a lower bound on the harmonic radius, cf.\
\Prp{eucl.metric}.

In each situation, we prove that the quadratic form $\qf d_\eps$ is
\emph{$\delta_\eps$-quasi-unitary equivalent} to a limit quadratic
form, with $\delta_\eps \to 0$ as $\eps \to 0$.  The concept of
quasi-unitary equivalence is presented in \App{main.tool}.  It assures
that the resolvents of the corresponding Laplacians converges, if
intertwined with some ``good'' bounded operators
$\map{J_\eps}{\HS_0}{\HS_\eps}$ which permit the transplantation of
functions of the limit space $\HS_0$ ($\Lsqr{X,g}$ or a subspace of
it) on the manifold with handles $\HS_\eps=\Lsqr{M_\eps,g_\eps}$,
namely we obtain
\begin{equation*}
  \bignorm{J_\eps(\Delta_0+1)^{-1}-(\Delta_\eps+1)^{-1}J_\eps}
  \leq 7\delta_\eps\to 0
\end{equation*}
(\Prp{quasi-uni}).  As a consequence, we also have convergence of
spectra as in \Thm{spectrum}.

We have first two results of convergence where in the limit, the
handles are no longer seen.  We call such handles \emph{fading}:

\subsection{Fading handles}
\label{ssec:fading.handles}

In our first main theorem, we just ``ignore'' the handles by the use
of a cut-off function (in the same way as the Dirichlet problem for
many small holes, see~\cite[Sec.~5]{anne-post:21}).
In this subsection, we assume that there is
\begin{itemize}
\item $\eta_\eps \in (0,r_0)$ with $\eps<\eta_\eps$ (later the
  uniform cover distance)
\item an $\eps$-separated split set $I_\eps=I^-_\eps \dcup I^+_\eps$
  with uniform $\eta_\eps$-cover (where we join the handles) and
\item $\ell_\eps>0$ (later the handle length).
\end{itemize}
Here, the handles are too \emph{sparse}, i.e., they are uniformly
$\eta_\eps$-covered and $\eta_\eps$ tends to $0$ rather slowly (slower
than $\zeta_\eps$ defined in the next theorem).

\begin{theorem}[fading handles I]
  \label{thm:handles0}
  Let $X$ be a complete Riemannian manifold of dimension $m\geq 2$
  with bounded geometry.  Assume that
  \begin{enumerate}
  \item
    \label{handles0.b}
    \myparagraph{Length shrinking:}
    $\ell_\eps \to 0$ as $\eps \to 0$;
  \item
    \label{handles0.d}
    $\omega_\eps:=\dfrac{\zeta_\eps}{\eta_\eps}\to 0$ as $\eps \to 0$, where
    \begin{equation*}
      \zeta_\eps=
      \begin{cases}
        \eps^{\alpha_m} &(m\geq 3),\\
        \abs{\log\eps}^{-\alpha_2}&(m=2),
      \end{cases}
      \qquadtext{and}\alpha_m
      \begin{cases}
        = 1/2 &(m\geq 5),\\
        \in (0,1/2) &(m=4),\\
        = 1/3 &(m=3),\\
        =1/2  &(m=2);
      \end{cases}
    \end{equation*}
  \end{enumerate}
  Then the energy form $\qf d_\eps$ on the manifold $M_\eps$ with
  small handles of length $\ell_\eps$ and radius $\eps$ and the energy
  form $\qf d_0$ on the original manifold $X$ are $\delta_\eps$-quasi
  unitarily equivalent of order $2$ with partially isometric
  identification operators, where $\delta_\eps \to 0$ as $\eps \to 0$.
\end{theorem}

Let us now specify the dependence of the parameters on $\eps$:
\begin{corollary}[fading handles I]
  \label{cor:handles0}
  Within the setting of \Thm{handles0}, assume that
  \begin{enumerate}
  \item the handle length fulfils $\ell_\eps=\eps^\lambda$ for some
    $\lambda>0$;
  \item the uniform cover distance is $\eta_\eps=\eps^\alpha r_0$ for
    $m \ge 3$ resp.\ $\eta_\eps=\abs{\log \eps}^{-\alpha} r_0$ for
    $m=2$ and $\alpha \in [0,\alpha_m)$.
  \end{enumerate}
  Then the energy form $\qf d_\eps$ on the manifold $M_\eps$ with
  small handles of length $\ell_\eps$ and radius $\eps$ and the energy
  form $\qf d_0$ on the original manifold $X$ are $\delta_\eps$-quasi
  unitarily equivalent of order $2$ with partially isometric
  identification operators, where
  \begin{align*}
    \delta_\eps
    =\Err\Bigl(\eps^{\min\bigl\{\lambda,\frac{\alpha_m-\alpha}{4(1-\alpha_m)},
    \frac{\alpha_m+\alpha}2\bigr\}}\Bigr)
    \quadtext{resp.}
    \delta_\eps
    =\Err\Bigr(\abs{\log\eps})^{-\frac{1/2+\alpha}2}
    \bigabs{\log \abs{\log\eps}}^{\frac12}
    \Bigr)
  \end{align*}
  for $m \ge 3$ resp.\ $m=2$.
\end{corollary}
\begin{remark}[fading handles I: meaning of the conditions]
  \label{rem:handles0}
  \indent
  \begin{enumerate}
  \item
    \label{rem.handles0.a}
    The condition $\ell_\eps \to 0$ is rather natural, as if
    $\ell_\eps \to \ell_0>0$ then in the limit we obtain for each
    handle an interval $[0,\ell_0]$ together with the Dirichlet
    Laplacian on it (see
    e.g.~\cite{anne:87,anne-colbois:95,khrabustovskyi:13}).  In the
    parameter range of \Fig{alpha-lambda0}, this condition gives the
    horizontal $\lambda$-axis as bound of the allowed parameter
    region.
  \item Condition~\itemref{handles0.d} implies that the ratio
    \emph{radius per uniform cover distance} $\eps/\eta_\eps$ tends to
    $0$:
    \begin{align*}
      \frac \eps{\eta_\eps}
      = \frac \eps{\zeta_\eps} \cdot \omega_\eps
      =\begin{cases}
        \eps^{1-\alpha_m} \omega_\eps \to 0,& m \ge 3,\\
        \eps \abs{\log \eps}^{1/2} \omega_\eps \to 0,& m=2.
      \end{cases}
    \end{align*}
    Moreover, this implies that
    \begin{align*}
      \deltaBall \eps
      = \OptNonConc_m(\eps,\eta_\eps)
      = \Err\Bigl( \Bigl(\frac\eps{\eta_\eps}\Bigr)^{m/2}
      +\eps[-\log \eps]_2^{1/2}\Bigr)
    \end{align*}
    In addition, if $\ell_\eps$ is bounded in $\eps$ (as a consequence
    of~\itemref{handles0.b}), then we conclude that
    $\deltaHarm \eps \to 0$: we have
    \begin{align}
      \nonumber
      \deltaHarm \eps
      &= \Err\Bigl(\Bigl(\Bigl(\frac{\ell_\eps}\eps+1\Bigr)
        \Bigl(\frac{\eps^m}{\eta_\eps^m} + \eps^2 [-\log \eps]_2 \Bigr)
        \Bigr)^{1/2}\Bigr),\\
      \label{eq:handles0.deltaharm}
      &= \begin{cases}
        \Err(\eps^{(m-1-m \alpha_m)/2} \omega_\eps^{m/2}
        + \eps) ,& m \ge 3,\\
        \eps^{1/2} \abs{\log \eps}^{1/2}
        \omega_\eps,& m=2.
      \end{cases}
    \end{align}
    as $m-1-m \alpha_m \ge 1$ for $m \ge 3$.  We need
    $\deltaHarm \eps \to 0$ in \Cor{harm}.  In the spirit of
    \Def{non-concentr} the estimate in \Cor{harm} means that the
    handles are \emph{$\deltaHarm \eps$-non-concentrating of order $1$
      for harmonic functions}.

    Together with~\itemref{handles0.b}, we conclude that
    $\deltaHandle \eps \to 0$, see \Lem{dech2}.  The estimate in
    \Lem{dech2} means that the handles are
    \emph{$\deltaHandle \eps$-non-concentrating of order $1$ for all
      functions}.

  \item The precise estimate $\delta_\eps$ on the convergence speed is
    given in~\eqref{eq:handles0.error}.
  \item
    \label{rem.handles0.d}
    In our analysis the value $\alpha_m=1/2$ is needed for the
    same reason as in~\cite[Rem.~5.8]{anne-post:21}: we want to remain
    in Sobolev spaces of order $2$ and not higher, see
    \Rem{why.order.2}.  The optimal value would be
    $\wt \alpha_m=(m-2)/m$, and could possibly be reached with our
    methods by allowing a resolvent power on the right hand side such
    as $\norm{(J_\eps R-R_\eps J_\eps)R^ {k-2}}$ for some $k >2$
    depending on $m$.  From this weaker estimate, one can still
    conclude spectral convergence and other estimates; for details
    see~\cite{post:06,post:12}.
  \end{enumerate}
\end{remark}

We expect better results (at least for slowly shrinking handle length
$\ell_\eps$) by comparing the functions on the perturbed domain with
the harmonic extension on the handles of the functions on the
unperturbed one.  For a detailed discussion see \Rem{handles01} and
\Fig{alpha-lambda}.  The estimates on the harmonic extension of a
given function on $X_\eps$ onto the handles are given in
\Sec{est.handles}.
\begin{theorem}[fading handles II]
  \label{thm:handles1}
  Let $X$ be a complete Riemannian manifold of dimension $m\geq 2$
  with bounded geometry.  Assume that
  \begin{enumerate}
  \item
    \label{handles1.b}
    \myparagraph{Length shrinking:}
    $\ell_\eps \to 0$ as $\eps \to 0$,
  \item
    \label{handles1.c}
    \myparagraph{Non-concentrating on the handles (of order $2$ for
      harmonic functions):} $\deltaHarm \eps' \to 0$, i.e.,
    \begin{align*}
      \frac1{\eps \ell_\eps}
      \Bigl(\frac{\eps^m}{\eta_\eps^m} + \eps^2 [-\log \eps]_2 \Bigr)
      \to 0
    \end{align*}
    as $\eps \to 0$.
  \end{enumerate}
  Then the energy form $\qf d_\eps$ on the manifold $M_\eps$ with
  small handles of length $\ell_\eps$ and radius $\eps$ and the energy
  form $\qf d_0$ on the original manifold $X$ are $\delta_\eps$-quasi
  unitarily equivalent of order $2$ with partially isometric
  identification operators, where $\delta_\eps \to 0$ as $\eps \to 0$.
\end{theorem}

\begin{corollary}[fading handles II]
  \label{cor:handles1}
  Within the setting of \Thm{handles1}, we assume that
  $\eta_\eps=\eps^\alpha r_0$ if $m \ge 2$ and that the handle length
  fulfils $\ell_\eps=\eps^\lambda$.  If
  \begin{equation*}
    \mathrm{(i)\ } %
    \lambda>0
     \quadtext{and}
    \mathrm{(ii)\ } %
    \begin{cases}
      0 \le \alpha \le \frac{m-2}m \text{ and } \lambda < 1 \quad\text{or}\\
      \frac{m-2}m \le \alpha < 1 \text{ and } \lambda < (m-1)-m \alpha,
    \end{cases}
  \end{equation*}
  then the energy form $\qf d_\eps$ on the manifold $M_\eps$ with
  handles of length $\ell_\eps$ and radius $\eps$ and the energy form
  $\qf d_0$ on the original manifold $X$ are $\delta_\eps$-quasi
  unitarily equivalent of order $2$ with partially isometric
  identification operators, where
  $\delta_\eps=\Err(\eps^{\min\{\lambda,(-\lambda + (m-1)-m
    \alpha))/2, (1-\lambda)/2\}})$.
\end{corollary}

\begin{remark}[fading handles II: meaning of the conditions]
  \label{rem:handles1}
  \indent
  \begin{enumerate}
  \item The condition $\ell_\eps \to 0$ is needed as usual, see
    \Remenum{handles0}{rem.handles0.a}.

  \item Condition~\itemref{handles1.c} is equivalent with
    $\deltaHarm \eps' \to 0$; needed in \Cor{harm2}.  Actually,
    $\deltaHarm \eps'$ is of the same order as the square root of the
    expression in~\itemref{handles1.c}.  In the spirit of
    \Def{non-concentr} the estimate in \Cor{harm2} means that the
    handles are \emph{$\deltaHarm \eps'$-non-concentrating of order
      $2$ for harmonic functions}.  This condition implies also that
    $\deltaHarm \eps \to 0$, see \Rem{harm}.  Moreover, in the
    parameter range of \Fig{alpha-lambda1}, this condition determines
    the parameter range inside the polygon $CDD'E$.

  \item Conditions~\itemref{handles1.b}--\itemref{handles1.c} imply
    again that the ratio radius per uniform cover distance
    $\eps/\eta_\eps$ converges to $0$:
    \begin{align*}
      \frac \eps{\eta_\eps}
      = \Bigl(\frac 1{\eps \ell_\eps}
      \cdot \frac{\eps^m}{\eta_\eps^m}\Bigr)^{1/m}
      (\eps \ell_\eps))^{1/m}
      \to 0
    \end{align*}
    as $\eps \to 0$ using~\itemref{handles1.b}--\itemref{handles1.c}.

  \item The precise estimate $\delta_\eps$ on the convergence speed is
    given in~\eqref{eq:err.handles1}.
  \end{enumerate}
\end{remark}
\newcommand{\DrawParRangeFadingThree}
{
  \begin{tikzpicture}[scale=1.8, inner sep=2pt,
    vertex/.style={circle,draw=black!50,%
      fill=black!50, 
      inner sep=0pt,
      minimum size=1.5mm}
    ]
    \newcommand{\alphaM}{1/3}
    \newcommand{\m}{3}
    \node at (0.5,2.0) {$m=3$};
    \node at (0.5,1.7) {$\alpha_3=1/3$};
    \node at (2/3,-0.2) {$\frac13$};
    \node at ({2*\alphaM*\alphaM/(2-\alphaM)},-0.2) {$\frac1{15}$};

    \node(Cp) at ({2*\alphaM},0) {};
    \node(D) at (0,0) {};
    \node(Dm) at (0,{\alphaM/2}) {};
    \node(Dp) at (0,1) {};
    \node(Ep) at (2*\alphaM,1) {};
    \node(E) at ({2-4/\m},1) {};
    \node(F) at ({2-4/\m},1.3) {};
    \node(H) at ({2*\alphaM*\alphaM/(2-\alphaM)},%
    {\alphaM/(2-\alphaM)}) {};

    \filldraw[fill=gray!80!white, draw=gray!80!]
    (0,0) -- (0,\alphaM/2)
    -- ({2*\alphaM*\alphaM/(2-\alphaM)},{\alphaM/(2-\alphaM)}) 
    -- (2*\alphaM,0) -- (0,0);
    \filldraw[fill=gray!40!white, draw=gray!40!]
    (0,\alphaM/2) -- (0,1.5) --
    ({2*\alphaM*\alphaM/(2-\alphaM)},1.5)
    -- ({2*\alphaM*\alphaM/(2-\alphaM)},%
    {\alphaM/(2-\alphaM)}) 
    -- (0,\alphaM/2);
    \filldraw[fill=gray!20!white, draw=gray!20!]
    ({2*\alphaM*\alphaM/(2-\alphaM)},%
    {\alphaM/(2-\alphaM)}) 
    -- ({2*\alphaM*\alphaM/(2-\alphaM)},1.5)
    -- (2*\alphaM,1.5)
    -- (2*\alphaM,0)
    ({2*\alphaM*\alphaM/(2-\alphaM)},%
    {\alphaM/(2-\alphaM)}); 

    \draw[very thin,color=gray,step=2/\alphaM] (0,-0.05) grid (1.0,0.05);
    \draw[very thin,color=gray,step=2] (2/15,-0.05) --(2/15,+0.05);
    \draw[->] (0,0) -- (1.0,0) node[right] {$\alpha$};
    \draw[very thin,color=gray,step=0.5/\alphaM] (-0.05,0) grid (0.05,1.5);
    \draw[very thin,color=gray,step=1] (-0.1,0) grid (0.1,1);
    \draw[->] (0,0) -- (0,1.8) node[above] {$\lambda$};

    \node at (Cp) [vertex,label=above:$C'$] {};
    \node at (D) [vertex,label=left:$D$] {};
    \node at (Dm) [vertex,label=left:$D^-$] {};
    \node at (Dp) [vertex,label=left:$D'$] {};
    \draw[dotted] (0,1) -- (-0.4,1) node[left] {$1$};
    \node at (Ep) [vertex,label=right:$E'$] {};
    \node at (H) [vertex,label=right:$H$] {};
  \end{tikzpicture}
}
\newcommand{\DrawParRangeFadingFour}
{
  \begin{tikzpicture}[scale=1.8, inner sep=2pt,
    vertex/.style={circle,draw=black!50,%
      fill=black!50, 
      inner sep=0pt,
      minimum size=1.5mm}
    ]
    \newcommand{\alphaM}{0.5} 
    \newcommand{\m}{4}
    \node(Cp) at ({2*\alphaM},0) {};
    \node(D) at (0,0) {};
    \node(Dm) at (0,{\alphaM/2}) {};
    \node(Dp) at (0,1) {};
    \node(Ep) at (2*\alphaM,1) {};
    \node(E) at ({2-4/\m},1) {};
    \node(F) at ({2-4/\m},1.3) {};
    \node(H) at ({2*\alphaM*\alphaM/(2-\alphaM)},%
    {\alphaM/(2-\alphaM)}) {};

    \node at (0.5,2.0) {$m=4$};
    \node at (0.5,1.7) {$\alpha_4 \nearrow 1/2$};
    \node at ({2*\alphaM*\alphaM/(2-\alphaM)},-0.2) {$\alpha_4^*$};
    \node at (2*\alphaM,-0.2) {$\alpha_4$};

    \filldraw[fill=gray!80!white, draw=gray!80!]
    (0,0) -- (0,\alphaM/2)
    -- ({2*\alphaM*\alphaM/(2-\alphaM)},{\alphaM/(2-\alphaM)}) 
    -- (2*\alphaM,0) -- (0,0);
    \filldraw[fill=gray!40!white, draw=gray!40!]
    (0,\alphaM/2) -- (0,1.5) --
    ({2*\alphaM*\alphaM/(2-\alphaM)},1.5)
    -- ({2*\alphaM*\alphaM/(2-\alphaM)},%
    {\alphaM/(2-\alphaM)}) 
    -- (0,\alphaM/2);
    \filldraw[fill=gray!20!white, draw=gray!20!]
    ({2*\alphaM*\alphaM/(2-\alphaM)},%
    {\alphaM/(2-\alphaM)}) 
    -- ({2*\alphaM*\alphaM/(2-\alphaM)},1.5)
    -- (2*\alphaM,1.5)
    -- (2*\alphaM,0)
    ({2*\alphaM*\alphaM/(2-\alphaM)},%
    {\alphaM/(2-\alphaM)}); 

    \draw[very thin,color=gray,step=2/\alphaM] (0,-0.05) grid (1.0,0.05);
    \draw[very thin,color=gray,step=2] (2/6,-0.05) --(2/6,+0.05);
    \draw[very thin,color=gray,step=2] ({2*\alphaM*\alphaM/(2-\alphaM)},-0.05)
    --({2*\alphaM*\alphaM/(2-\alphaM)},+0.05);
    \draw[->] (0,0) -- (1.3,0) node[right] {$\alpha$};
    \draw[very thin,color=gray,step=\alphaM/2] (-0.05,0) grid (0.05,1.5);
    \draw[very thin,color=gray,step=1] (-0.1,0) grid (0.1,1);
    \draw[->] (0,0) -- (0,1.8) node[above] {$\lambda$};
    \draw[dotted] (0,1) -- (-0.4,1) node[left] {$1$};
    \node at (Cp) [vertex,label=above:$C'$] {};
    \node at (D) [vertex,label=left:$D$] {};
    \node at (Dm) [vertex,label=left:$D^-$] {};
    \node at (Dp) [vertex,label=left:$D'$] {};
    \node at (Ep) [vertex,label=right:$E'$] {};
    \node at (H) [vertex,label=right:$H$] {};
  \end{tikzpicture}
}

\newcommand{\DrawParRangeFadingHigher}
{
  \begin{tikzpicture}[scale=1.8, inner sep=2pt,
    vertex/.style={circle,draw=black!50,%
      fill=black!50, 
      inner sep=0pt,
      minimum size=1.5mm}
    ]
    \newcommand{\alphaM}{1/2} 

    \newcommand{\m}{5}
    \node(Cp) at ({2*\alphaM},0) {};
    \node(D) at (0,0) {};
    \node(Dm) at (0,{\alphaM/2}) {};
    \node(Dp) at (0,1) {};
    \node(Ep) at (2*\alphaM,1) {};
    \node(E) at ({2-4/\m},1) {};
    \node(F) at ({2-4/\m},1.3) {};
    \node(H) at ({2*\alphaM*\alphaM/(2-\alphaM)},%
    {\alphaM/(2-\alphaM)}) {};

    \node at (0.5,2.0) {$m\ge 5$};
    \node at (0.5,1.7) {$\alpha_m=1/2$};
    \node at (1,-0.2) {$\frac12$};
    \node at (2/6,-0.2) {$\frac16$};

    \filldraw[fill=gray!80!white, draw=gray!80!]
    (0,0) -- (0,\alphaM/2)
    -- ({2*\alphaM*\alphaM/(2-\alphaM)},{\alphaM/(2-\alphaM)}) 
    -- (2*\alphaM,0) -- (0,0);
    \filldraw[fill=gray!40!white, draw=gray!40!]
    (0,\alphaM/2) -- (0,1.5) --
    ({2*\alphaM*\alphaM/(2-\alphaM)},1.5)
    -- ({2*\alphaM*\alphaM/(2-\alphaM)},%
    {\alphaM/(2-\alphaM)}) 
    -- (0,\alphaM/2);
    \filldraw[fill=gray!20!white, draw=gray!20!]
    ({2*\alphaM*\alphaM/(2-\alphaM)},%
    {\alphaM/(2-\alphaM)}) 
    -- ({2*\alphaM*\alphaM/(2-\alphaM)},1.5)
    -- (2*\alphaM,1.5)
    -- (2*\alphaM,0)
    ({2*\alphaM*\alphaM/(2-\alphaM)},%
    {\alphaM/(2-\alphaM)}); 

    \draw[very thin,color=gray,step=2/\alphaM] (0,-0.05) grid (1.0,0.05);
    \draw[very thin,color=gray,step=2] (2/6,-0.05) --(2/6,+0.05);
    \draw[->] (0,0) -- (1.3,0) node[right] {$\alpha$};
    \draw[very thin,color=gray,step=0.5/\alphaM] (-0.05,0) grid (0.05,1.5);
    \draw[very thin,color=gray,step=1] (-0.1,0) grid (0.1,1);
    \draw[->] (0,0) -- (0,1.8) node[above] {$\lambda$};

    \node at (Cp) [vertex,label=above:$C'$] {};
    \node at (D) [vertex,label=left:$D$] {};
    \node at (Dm) [vertex,label=left:$D^-$] {};
    \node at (Dp) [vertex,label=left:$D'$] {};
    \draw[dotted] (0,1) -- (-0.4,1) node[left] {$1$};
    \node at (Ep) [vertex,label=right:$E'$] {};
    \node at (H) [vertex,label=right:$H$] {};
\end{tikzpicture}
}
\begin{figure}[h]
  \centering
  \DrawParRangeFadingThree
  \DrawParRangeFadingFour
  \DrawParRangeFadingHigher

  \caption{The range of the parameters $\alpha$
    ($\eta_\eps=\eps^\alpha r_0$, the larger the denser the handles
    are) and $\lambda$ ($\ell_\eps=\eps^\lambda$, the larger the
    shorter the handles are) in which \Cor{handles0} is valid
    ($m \ge 3$).  For the letters describing points in the
    $(\alpha,\lambda)$-plane, see~\eqref{eq:alpha.lambda.points}).
    The estimate on the convergence speed is of order
    $\Err(\eps^{(\alpha_m-\alpha)/(2(1-\alpha_m))})$ in the (infinite)
    rectangle $C'G(\alpha_m^*,\infty)(\alpha_m,\infty)$ (lightest
    grey), of order $\Err(\eps^{(\alpha_m+\alpha)/2)})$ in the
    (infinite) rectangle $D^-(0,\infty)(\alpha_m^*,\infty)G$ (middle
    grey) resp.\ of order $\Err(\eps^\lambda)$ in the triangle
    $C'DD^-G$ (dark grey).}
  \label{fig:alpha-lambda0}
\end{figure}

\newcommand{\DrawParRangeFadingAlt}[1] 
{
  \begin{tikzpicture}[scale=1.8, inner sep=2pt,
    vertex/.style={circle,draw=black!50,%
      fill=black!50, 
      inner sep=0pt,
      minimum size=1.5mm}
    ]

    \foreach \m in {#1}
    {
      \node(C) at ({1-1/\m},0) {};
      \node(Ch) at ({1-2/\m},1/3) {};
      \node(D) at (0,0) {};
      \node(Dh) at (0,1/3) {};
      \node(Dp) at (0,1) {};
      \node(E) at ({1-2/\m},1) {};

      \filldraw[fill=gray!80!white, draw=gray!80!]
      (0,0) -- (0,1/3) -- (1-2/\m,1/3) -- (1-1/\m,0) -- (0,0);

      \filldraw[fill=gray!60!white, draw=gray!60!]
      (0,1/3) -- (0,1) -- (1-2/\m,1) -- (1-2/\m,1/3) -- (0,1/3);
      \filldraw[fill=gray!40!, draw=gray!40!]
      (1-2/\m,1) -- (1-1/\m,0) -- (1-2/\m,1/3) -- (1-2/\m,1);

      \draw[very thin,color=gray,step=1/\m] (0,-0.05) grid (1.0,0.05);
      \draw[->] (0,0) -- (1.1,0) node[right] {$\alpha$};
      \node at (1,-0.25) {$1$};

      \node at (D) [vertex,label=left:$D$] {};
      \node at (Dp) [vertex,label=left:$D'$] {};
      \ifthenelse{\m=2}
      {
        \node at (Ch) [vertex,label=left:${\hat C=\hat D}$] {};
        \node at (E) [vertex,label=right:$E$] {};
      }
      {
        \node at (C) [vertex,label=above:$C$] {};
        \node at (Ch) [vertex,label=above:$\hat C$] {};
        \node at (D) [vertex,label=left:$D$] {};
        \node at (Dh) [vertex,label=left:$\hat D$] {};
        \node at (E) [vertex,label=left:$E$] {};
      }
      \ifthenelse{\m=3}
      {}
      {}
      \ifthenelse{\m=4}
      {
        \node at (3/4,-0.25) {$\frac34$};
        \node at (1/2,-0.25) {$\frac12$};
      }{}
      \ifthenelse{\m=5} 
      {
        \node at (0.4,1.5) {$m\ge5$};
        \draw[dotted,color=gray] (1-1/\m,0.05) -- (1-1/\m,1.1);
        \node at (1-1/\m, 1.1) {$\frac{m-1}m$};
        \node at (1-2/\m,-0.25) {$\frac{m-2}m$};
      }
      {
        \node at (0.5,1.5) {$m=\m$};
        \draw[dotted] (0,1) -- (0.8,1) node[right] {$1$};
      }

      \draw[very thin,color=gray,step=1/3] (-0.05,0) grid (0.05,1.2);
      \draw[very thin,color=gray,step=1] (-0.1,0) grid (0.1,1);
      \draw[->] (0,0) -- (0,1.3) node[above] {$\lambda$};
    }
  \end{tikzpicture}
}

\begin{figure}[h]
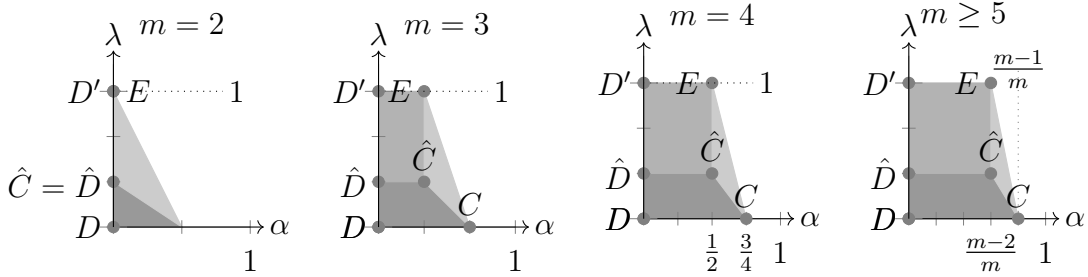

  \centering
  \DrawParRangeFadingAlt{2}
  \DrawParRangeFadingAlt{3}
  \DrawParRangeFadingAlt{4}
  \DrawParRangeFadingAlt{5}

  \caption{The range of the parameters $\alpha$
    ($\eta_\eps=\eps^\alpha r_0$ in which \Cor{handles1} is valid.
    The estimate on the convergence speed is of order
    $\Err(\eps^{(1-\lambda)/2})$ in the polygon $C'DD^-G$ (horizontal
    lines); of order $\Err(\eps^{(-\lambda+(m-1)-m\alpha)/2})$ in the
    polygon (diagonal lines) resp.\ of order $\Err(\eps^\lambda)$
    (dark grey).}
  \label{fig:alpha-lambda1}
\end{figure}

\newcommand{\DrawParRangeFadingCompare}[1]%
{
  \begin{tikzpicture}[scale=2, inner sep=2pt,
    vertex/.style={circle,draw=black!50,%
      fill=black!50, 
      inner sep=0pt,
      minimum size=1.5mm}
    ]
    \newcommand{\Myinfty}{2.0}
    \newcommand{\alphaScale}{1.5}
    \foreach \m in {#1}
    {
      \ifthenelse{\m=3}
      {\newcommand{\alphaM}{1/3}}
      {\ifthenelse{\m=4}
        {\newcommand{\alphaM}{1/2}}
        {\newcommand{\alphaM}{1/2}}
      }
      \coordinate(C) at ({\alphaScale*(1-1/\m)},0);
      \coordinate(Cp) at ({\alphaScale*\alphaM},0);
      \coordinate(D) at (0,0); 
      \coordinate(Dm) at (0,{\alphaM/2});
      \coordinate(Dp) at (0,1);
      \coordinate(Dpl) at (0,{1-\alphaM});
      \coordinate(E) at ({\alphaScale*(1-2/\m)},1);
      \coordinate(Ep) at ({\alphaScale*\alphaM},1);
      \coordinate(F) at ({\alphaScale*(1-2/\m)},{\Myinfty});
      \coordinate(H) at ({\alphaScale*\alphaM*\alphaM/(2-\alphaM)},%
      {\alphaM/(2-\alphaM)});
      \coordinate(Hp) at ({\alphaScale*\alphaM*\alphaM/(2-\alphaM)},%
      {(2-3*\alphaM)/(2-\alphaM)});

      \filldraw[fill=gray!60!, draw=gray!60!]
      (D)--(Cp)--(H)--(Dm)--(D);

      \filldraw[pattern=horizontal lines, pattern color=gray!80!, draw=gray!80!]
      (Dm)--(H)--(Cp)--(Ep)--(Hp)--(Dpl)--(Dm);

      \filldraw[pattern=horizontal lines, pattern color=gray!30!, draw=gray!30!]
      (Cp)--(C)--(E)--(Ep) --(Cp);

      \filldraw[pattern=vertical lines, pattern color=gray!80!, draw=gray!80!]
      (Dpl)--(Hp)--(Ep)--(\alphaScale*\alphaM,\Myinfty)--(0,\Myinfty)--(Dpl);

      \filldraw[pattern=vertical lines, pattern color=gray!30!, draw=gray!30!]
      (Dp)--(Ep)--(\alphaScale*\alphaM,\Myinfty)--(0,\Myinfty)--(Dp);

      \ifthenelse{\m=3}
      {
        \node at (0.6,\Myinfty+0.2) {$m=3$};
        \node at (Hp) [vertex,label=right:$H^+$] {};
      }
      {}
      \ifthenelse{\m=4}
      {
        \node at (0.6,\Myinfty+0.2) {$m=4$};
      }{}
      \ifthenelse{\m>4}
      {
        \node at (0.6,\Myinfty+0.2) {$m \ge 5$};
        \filldraw[pattern=dots, pattern color=gray!40!, draw=gray!40!]
        (E)--(Ep)--(\alphaScale*\alphaM,\Myinfty)
        --({\alphaScale*(1-2/\m)},\Myinfty)
        --(E);
        \node at (E) [vertex,label=right:$E$] {};
        \node at (Ep) [vertex,label=left:$E'$] {};
      }
      {
        \node at (E) [vertex,label=right:$E$] {};
      }
      \node at (C) [vertex,label=below:$C$] {};
      \node at (Cp) [vertex,label=below:$C'$] {};
      \node at (D) [vertex,label=left:$D$] {};
      \node at (Dp) [vertex,label=left:$D'$] {};
      \node at (Dm) [vertex,label=left:$D^-$] {};
      \node at (Dpl) [vertex,label=left:$D^+$] {};
      \node at (H) [vertex,label=right:$H$] {};

      \draw[very thin,color=gray,step=\alphaScale/\m] (0,-0.05) grid (\alphaScale*1.0,0.05);
      \draw[->] (0,0) -- (\alphaScale*1.1,0) node[right] {$\alpha$};
      \node[label=below:$1$] at (\alphaScale*1,-0) {};
      \draw[very thin,color=gray,step=1/3] (-0.05,0) grid (0.05,1.2);
      \draw[very thin,color=gray,step=1] (-0.1,0) grid (0.1,1);
      \draw[->] (0,0) -- (0,\Myinfty) node[above] {$\lambda$};
    }
  \end{tikzpicture}
}

\begin{figure}[h]
  \centering
  \DrawParRangeFadingCompare{3}\quad
  \DrawParRangeFadingCompare{4}\quad
  \DrawParRangeFadingCompare{6}
  \caption{Comparison of the two fading results: horizontal light
    lines: only \Cor{handles1} applies; vertical light lines: only
    \Cor{handles0} applies; horizontal dark lines: \Cor{handles1} has
    better estimates than \Cor{handles0}; vertical dark lines:
    \Cor{handles0} has better estimates than \Cor{handles1}; grey
    area: both results have the same order on the convergence speed.
    Dotted area (only for $m\ge 5$): none of our results apply.}
  \label{fig:alpha-lambda}
\end{figure}

\begin{remarks}[comparision of both results and further remarks on the
  fading case]
  \label{rem:handles01}
  We now discuss the two fading results \Cors{handles0}{handles1} and
  introduce some points in the $(\alpha,\lambda)$-plane.  Here,
  $\alpha$ enters as exponent in the uniform cover distance
  ($\eta_\eps=\eps^\alpha r_0$, for \Cor{handles0} only if $m \ge 3$);
  the larger $\alpha$ is, the denser the handles are.  Moreover,
  $\lambda$ is the exponent in the handle length
  ($\ell_\eps=\eps^\lambda$);, the larger the shorter the handles are.

  We first introduce some points needed to compare
  \Cors{handles0}{handles1}, see also
  \FigS{alpha-lambda0}{alpha-lambda} (we use similar letters as
  in~\cite{khrabustovskyi:13}):
  \begin{align}
    \nonumber
    C&=\Bigl(\frac{m-2}m,0\Bigr),
    &C'&=(\alpha_m,0),
    &\hat C&=\Bigl(\frac{m-2}m,\frac13\Bigr),\\
    \nonumber
    D&=(0,0),
    &D'&=(0,1),
    &D^-&=\Bigl(0,\frac{\alpha_m}2\Bigr),\\
    \hat D&=(0,1/3),
    &E&=\Bigl(\frac{m-2}m,1\Bigr),
    &H&=\Bigl(\alpha_m^*,\frac{\alpha_m}{2-\alpha_m}\Bigr),
    \label{eq:alpha.lambda.points}
  \end{align}
  where
  \begin{equation}
    \label{eq:alpha.star}
    \alpha_m^*
    =\frac{\alpha_m^2}{2-\alpha_m},
  \end{equation}
  i.e., $\alpha_3^*=1/15$, $\alpha_4^* \in (0,1/6)$ and $\alpha_m=1/6$
  for $m\ge 5$; we refer to the letters also later on.

  Our two results complement each other:
  \begin{enumerate}
  \item \myparagraph{Very fast length-shrinking handles; the case only
      covered by \Cor{handles0}.} %
    In any dimension $m \ge 2$, \Cor{handles0} applies for handles
    which length $\ell_\eps=\eps^\lambda$ shrinks as fast as wanted,
    i.e., for any $\lambda>0$.  This is due to the fact that the
    cut-off function method used in \Thm{handles0} does not see the
    handles at all.  The case when only \Cor{handles0} applies is the
    vertically lined infinite polygon given in \Fig{alpha-lambda}.  On
    the other hand, the range of $\alpha$ is restricted by $\alpha_m$,
    and this constant is not optimal for $m \ge 5$ for reasons
    explained in \Remenum{handles0}{handles0.d}; this is the dotted
    infinite rectangle in \Fig{alpha-lambda}.

  \item \myparagraph{The case only in \Cor{handles1}.}  In any
    dimension $m \ge 2$, \Cor{handles1} applies in the parameter range
    for the triangle $((m-2)/m,0)-((m-2)/m,1)-((m-1)/m,0)$.  This
    range never appears in \Cor{handles0}.  Also, for $m\ge 5$, we
    allow $\alpha \in [\alpha_m,(m-2)/m)$, a range not covered by
    \Cor{handles0}.  Nevertheless, we have to restrict to slowly
    decaying handle lengths, namely to $\lambda < 1$; this restriction
    comes from $\deltaHarm \eps ' \to 0$, and the bad term is the
    energy of the transversally constant part of a harmonic extension,
    the $1/\ell_\eps^2$-term in~\eqref{eq:est.harm.c}, entering in
    $\deltaHarm \eps^\bullet$ in \Prp{harm2}, and finally also in
    $\deltaHarm \eps '$.

  \item \myparagraph{Homothetically shrinking handles.} %
    In particular, the case $\lambda \ge 1$ is allowed only
    \Cor{handles0}.  The case when $\ell_\eps=\eps$ (i.e., the handles
    shrink homothetically $\Cyl_\eps=\eps\Cyl_1$); in other words, the
    length $\ell_\eps=\eps$ shrinks as fast as the radius $\eps$) is
    covered only by \Cor{handles0}, but there is a gap for the values
    $1/2 \le \alpha < (m-2)/m$ from dimension $m \ge 5$ not covered
    be neither \Cor{handles0} nor \Cor{handles1}.

  \item \myparagraph{Two-dimension case.} %
    if $m=2$ the scale of the uniform cover distance
    $\eta_{\hat \eps}$ used in \Cor{handles0} (with $\hat \eps$
    instead of $\eps$) is exponentially slower than $\eta_\eps$ used
    in \Cor{handles1} (with $\eps$): if
    $\eta_{\hat \eps}=(1/\abs{\log \hat \eps})^\alpha$ of
    \Cor{handles0} for $\hat \eps<1$ equals $\eta_\eps=\eps^\alpha$ of
    \Cor{handles1}, then $-1/\log\hat \eps=\eps$, i.e.,
    $\hat \eps=\e^{-1/\eps}$, hence \Cor{handles1} covers also cases
    when the uniform cover distance decays much faster.
  \end{enumerate}
\end{remarks}

We present some examples in \Subsec{examples}.

\subsection{Adhering handles}
\label{ssec:adhering.handles}

The third theorem presents a situation where we obtain
\emph{adhering handles}.  In this case the handles are so dense that
they glue or identify two parts of the original manifold, gluing them
together at the limit.

We first fix the assumptions on the parameters; recall that $r_0$
denotes a lower bound on the harmonic radius, cf.\ \Prp{eucl.metric}:

\begin{assumption}[relation of handle length, radius and uniform cover
  distance]
  \label{ass:handle.len}
  We assume that for each $\eps \in (0,r_0)$ the handle length
  $\ell_\eps>0$ and the uniform cover distance $\eta_\eps>0$ fulfil
  $0<\eps < \eta_\eps \le r_0$ and
   \begin{enumerate}
  \item
    \label{handles3.a}
    \myparagraph{Length shrinking:}
    $\ell_\eps \to 0$,
  \item
    \label{handles3.b}
    \myparagraph{Non-concentrating on the handles (of order $1$ for
      harmonic functions):} $\deltaHarm \eps \to 0$, i.e.,
    \begin{align*}
      \Bigl(\frac{\ell_\eps}\eps +1\Bigr)
      \Bigl(\frac{\eps^m}{\eta_\eps^m} + \eps^2 [-\log \eps]_2 \Bigr)
      \to 0
    \end{align*}
    as $\eps \to 0$,

  \item
    \label{handles3.c}
    \myparagraph{Non-concentrating of anti-symmetric functions:}
    $\deltaAntisym \eps \to 0$, i.e.,
    \begin{align*}
      \frac{\eta_\eps^m}{\eps^{m-2}} \cdot
      \Bigl(\dfrac{\ell_\eps}\eps + \Bigr[\log \frac {\eta_\eps}\eps\Bigr]_2
      \Bigr) \to 0
    \end{align*}
    as $\eps \to 0$,

    \item
    \label{handles3.d}
    \myparagraph{Non-concentrating on the handles (of order $2$ for
      harmonic functions without constant transversal part):}
    $\deltaHarm \eps^\perp \to 0$, i.e.,
    \begin{align*}
      \Bigl(\frac \eps {\ell_\eps} +1\Bigr)
      \Bigl(\frac{\eps^m}{\eta_\eps^m} + \eps^2[-\log \eps]_2 \Bigr) \to 0
    \end{align*}
    as $\eps \to 0$.
  \end{enumerate}
\end{assumption}

We give a concrete setting in \Cor{handles3} again with
$\eta_\eps=\eps^\alpha r_0$ and $\ell_\eps=\eps^\lambda$ for the
parameters $\alpha$ and $\lambda$.

\begin{remark}[adhering handles: meaning of the conditions]
  \label{rem:handles3}
  \indent
  \begin{enumerate}
  \item The condition $\ell_\eps \to 0$ is needed as usual, see
    \Remenum{handles0}{rem.handles0.a}.

  \item Condition~\eqref{handles3.b} is equivalent with
    $\deltaHarm \eps \to 0$.  Actually, $\deltaHarm \eps$ is of the
    same order as the square root of the expression
    in~\itemref{handles3.b}.

  \item Condition~\eqref{handles3.c} is equivalent with
    $\deltaAntisym \eps \to 0$ needed in \Lem{ua}.  Actually,
    $\deltaAntisym \eps$ is of the same order as the square root of
    the expression in~\itemref{handles3.c}.  In the spirit of
    \Def{non-concentr} the estimate in \Lem{ua} means that the handles
    are \emph{$\deltaAntisym \eps$-non-concentrating of order $1$ for
      functions anti-symmetric nearby the handles} (see
    \Subsec{more.tools} for a precise definition).

  \item The last Condition~\eqref{handles3.d} is equivalent with
    $\deltaHarm \eps^\perp \to 0$ needed in \Prp{harm2} (see also
    \Fig{alpha-lambda-new}).  Again, $\deltaHarm \eps^\perp$ is of the
    same order as the square root of the expression
    in~\itemref{handles3.d}.  In the spirit of \Def{non-concentr} the
    second estimate in \Prp{harm2} means that the handles are
    \emph{$\deltaHarm \eps^\perp$-non-concentrating of order $2$ for
      harmonic functions without constant transversal part}.
  \end{enumerate}
\end{remark}

We now need conditions on the distributions of the points, where the
handles are attached (see also \Fig{identified-mfd}):
\begin{figure}[h]
  \centering
    \includegraphics{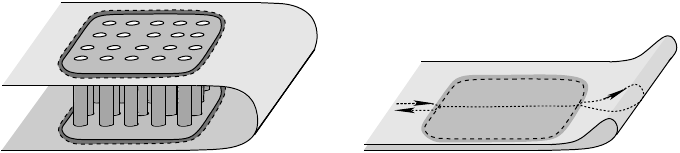}%

  \caption{\emph{Left: }The manifold $M_\eps$ with handles of radius
    $\eps$ and length $\ell_\eps$; medium grey: $\Omega^+$ (top) and
    $\Omega^-$ (bottom), together with its $\wt \eps$-neighbourhood
    $\Omega^+_{\wt \eps}$ and $\Omega^+_{\wt \eps}$ in medium and dark
    grey, respectively.  \newline %
    \emph{Right:} The space after identification: one can think of the
    limit space as a piece of paper where the parts $\Omega^+$ (top)
    and $\Omega^-$ (bottom) are glued together.  The dotted line in
    the right figure is used to illustrate the effect of the
    symmetrisation operator $S f$ in \Fig{symmetrisation} for the
    proof of \Thm{handles3}.}
  \label{fig:identified-mfd}
\end{figure}
\begin{assumption}[isometric subsets and uniform cover]
  \label{ass:isometric}
  \indent
  \begin{enumerate}
  \item
    \label{isometric.a}
    We assume that $\Omega^-$ and $\Omega^+$ are open, isometric
    and that $\Omega_{r_1}^-$ and $\Omega_{r_1}^+$ are disjoint and
    still isometric, where
    \begin{equation*}
      \Omega_{r_1}^\pm
      := \set{x \in X}{d(x,\Omega^\pm)< r_1}
    \end{equation*}
    denotes the $r_1$-neighbourhood, i.e.\ we assume that there exist
    an $r_1 \in (0,r_0)$ and an isometry
    \begin{equation*}
      \map \phi {\Omega^-_{r_1}}{\Omega^+_{r_1}}
      \qquadtext{such that}
      \phi(\Omega^-)=\Omega^+.
    \end{equation*}
    We set $\Omega=\Omega^- \dcup \Omega^+$ and
    $\Omega_{\wt \eps}=\Omega^-_{\wt \eps} \dcup \Omega^+_{\wt \eps}$
    for $\wt \eps \in (0,r_1)$.
  \item
    \label{isometric.b}
    We assume that there is $N \in \N^*$ and for each
    $\eps \in (0,r_0)$ there is $\eta_\eps \in (2\eps,r_0)$ and an
    $\eps$-separated split set $I_\eps=I^-_\eps \dcup I^+_\eps$ with
    $I^\pm_\eps\subset \Omega^\pm$ such that $I^+_\eps=\phi(I^-_\eps)$
    and that $I_\eps$ has $N$-uniform $\eta_\eps$-cover
    (cf.~\eqref{eq:uni.loc.bdd}, again the dependence of $\eta_\eps$
    on $\eps$ will be specified later).
  \item
    \label{isometric.c}
    Finally, we assume that there is a function
    $\map {\wt \cdot}{(0,r_1)}{(0,1]}$, $\eps \mapsto \wt \eps$, with
    \begin{align*}
      \wt \eps \to 0
      \quadtext{and}
      \frac{\deltaAntisym \eps}{\sqrt{\wt \eps}} \to 0
    \end{align*}
    as $\eps \to 0$, i.e.,
    \begin{equation*}
      \wt \eps \to 0
      \quadtext{and}
      \Bigl(\frac{\eta_\eps^m}{{\wt \eps}\eps^{m-2}} \cdot
      \Bigl(\dfrac{\ell_\eps}\eps + \Bigr[\log \dfrac {\eta_\eps}\eps\Bigr]_2
      \Bigr)\Bigr)^{1/2} \to 0
      \qquadtext{as} \eps \to 0,
    \end{equation*}
    such that
    \begin{equation*}
      \clo \Omega_{\wt \eps} \subset \bigcup_{p \in I_\eps} B_{\eta_\eps}(p),
    \end{equation*}
    i.e., $(B_{\eta_\eps}(p))_{p \in I_\eps}$, is a uniform
    locally finite cover of $\clo \Omega_{\wt \eps}$.
  \end{enumerate}
\end{assumption}
The estimate on the convergence speed $\deltaAntisym \eps$ is defined
in \Lem{ua}.  A necessary condition for such a covering to exist is
that $\wt \eps/\eta_\eps$ is bounded in $\eps>0$.  If
$\eta_\eps=\eps^\alpha r_0$ and $\wt \eps =\eps^\beta r_1$ for some
$r_1 \in (0,1)$, then we necessarily have $\beta \ge \alpha$.  The
case $\beta=\alpha$ can actually be achieved in the situation of
\Lem{flat.example} below.

Summarising, we have
\begin{equation*}
  B_\eps
  \coloneqq \bigdcup_{p \in I_\eps}B_\eps(p)
  \subset \Omega
  \subset \Omega_{\wt \eps}
  \subset B_{\eta_\eps}
  \coloneqq \bigcup_{p \in I_\eps}B_{\eta_\eps}(p)
  \quadtext{and}
  \frac{\wt \eps}{\eta_\eps} \text{ bounded,}
\end{equation*}
see \Fig{covering}.  We give examples satisfying \Ass{isometric} in
\Subsec{examples}

We suppose that $\bd\Omega=\bd\Omega^- \dcup \bd \Omega^+$ admits a
uniform tubular $r_1$-neighbourhood in the sense of~\cite[Assumption
6.9]{anne-post:21}:
\begin{assumption}[uniform tubular neighbourhood]
  \label{ass:unif.tub.nbhd}
  \indent
  \begin{enumerate}
  \item We assume that $Y=\bd \Omega$ is a smooth\footnote{One can
      relax the smoothness assumption to the case of Lipschitz
      manifolds. To keep this presentation simple, we will not go into
      details here.}  submanifold with embedding $\embmap \iota Y X$
    and induced metric $h:=\iota^* g$.

  \item We assume that \emph{$Y$ admits a uniform  tubular
      ($r_1$-)neighbourhood}, i.e., $Y$ has a global outwards normal
    unit vector field $\vec N$ 
    and that there is $r_1>0$ such that
    \begin{equation}
      \label{eq:expnormal}
      \map{\exp}{Y\times (-r_1,r_1)} X,
      \qquad
      (y,t) \mapsto \exp_y(t\vec N(y))
    \end{equation}
    is a diffeomorphism.
  \end{enumerate}
\end{assumption}
The previous assumption is fulfilled if $\bd \Omega$ is compact.

\begin{figure}[h]
  \centering
  \begin{picture}(0,0)%
    \includegraphics[scale=0.9]{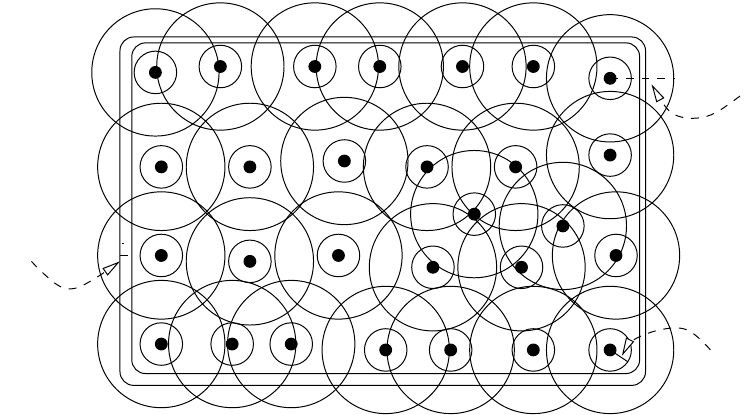}%
  \end{picture}%
  \setlength{\unitlength}{3729sp}%
  \begin{picture}(5745,3146)(1021,-4964)
    \put(6701,-2536){$\eta_\eps$}%
    \put(1136,-3706){$\wt\eps$}%
    \put(6516,-4471){$\eps$}%
  \end{picture}%
  \caption{The covering of $\Omega^\pm_{\wt \eps}$ with
    $\eta_\eps$-cover number $N=3$.}
  \label{fig:covering}
\end{figure}
Let now, as before, $\qf d_\eps$ be the quadratic form of the manifold
$M_\eps$ with handles of radius $\eps>0$ and length $\ell_\eps>0$
joining $\bd B_\eps(p)$ and $\bd B_\eps(\bar p)$ with $\bar p=\phi(p)$
for each $p \in I_\eps^-$ (see \Def{qeps}).

The limit quadratic form $\qf d_0$ is defined via
\begin{align}
  \label{eq:limit.solid}
  \HS_0=\set{f\in \Lsqr{X,g}}{\Gamma f=0} \quadtext{and}
  \HS_0^1 = \set{f\in \Sob{X,g}}{\Gamma f=0},
\end{align}
where
\begin{equation}
  \label{eq:def.gamma}
  \Gamma f := (f - f \circ \phi) \restr{\Omega^-}
\end{equation}
and $\qf d_0(f)=\normsqr[\Lsqr {X,g}]{df}$ for $f \in \HS_0^1$.  Note that
\begin{align*}
  \map \Gamma{\Lsqr{X,g}}{\Lsqr{\Omega^-,g}}
  \qquadtext{and}
  \map \Gamma{\Sob{X,g}}{\Sob{\Omega^-,g}}
\end{align*}
are both bounded maps, hence $\qf d_0$ is a closed quadratic form.
We call $\qf d_0$ the \emph{identifying} energy form.

\begin{remark}[the Laplacian on the limit manifold]
  \label{rem:lapl.kirchhoff}
  (A unitary equivalent version of) the associated operator $\Delta_0$
  acts on
  $f=(f_\itr,f_\exr) \in \Sob[2]{\Omega^-} \oplus \Sob[2]{X \setminus
    \intr \Omega}$ with $f_\itr := f \restr{\clo {\Omega^-}}$ and
  $f_\exr := f \restr{X \setminus \Omega}$ as
  $\Delta_0 f = (2\Delta f_\itr, \Delta f_\exr)$ with conditions
  \begin{equation*}
    f_\itr = f_\exr^- = f_\exr^+
    \qquadtext{and}
    2\partial_\itr f_\itr
    = \partial_\exr^- f_\exr^- + \partial_\exr^+ f_\exr^+,
  \end{equation*}
  where $f_\exr^-:=f_\exr \restr {\bd \Omega^-}$ and
  $f_\exr^+ := (f_\exr \circ \phi) \restr {\bd \Omega^-}$ and where
  $\partial_\itr$ is the (outwards) normal derivative on $\Omega^-$
  and $\partial_\exr^-$ (resp.\ $\partial_\exr^+$) the (inwards)
  normal derivative on $X \setminus \Omega$ near $\bd \Omega^-$
  (resp.\ near $\bd \Omega^+$ taken back to $\bd \Omega^-$ with
  $\phi$, i.e.,
  $\partial_\exr^+ f_\exr^+:= \partial_\exr^- (f_\exr \circ \phi)$).
  This condition resembles a so-called weighted ``Kirchhoff''
  condition on a metric graph; here in a higher dimensional version.
\end{remark}

The next main result is now as follows:
\begin{theorem}[adhering handles identifying parts of the manifold]
  \label{thm:handles3}
  Let $X$ be a complete Riemannian manifold of dimension $m\geq 2$
  with bounded geometry and harmonic radius $r_0>0$.  Assume that the
  parameters $\eps$, $\eta_\eps$ and $\ell_\eps$ fulfil
  \Ass{handle.len}.  Moreover, assume that
  \Asss{isometric}{unif.tub.nbhd} on the isometric subsets
  $\Omega^\pm$ and a uniform tubular neighbourhood and cover exist.
  Then the energy form $\qf d_\eps$ defined in \Def{qeps} is
  $\delta_\eps$-quasi-unitary equivalent of order $2$ with the
  identifying energy form $\qf d_0$ defined above with partially
  isometric identification operators, where $\delta_\eps \to 0$ as
  $\eps \to 0$ is given in~\eqref{eq:err.handles3}.
\end{theorem}

\newcommand{\DrawParRangeKhr}[1]%
{
  \begin{tikzpicture}[scale=1.5, inner sep=2pt,
    vertex/.style={circle,draw=black!50,%
      fill=black!50, 
      inner sep=0pt,
      minimum size=1.5mm}
    ]
    \foreach \m in {#1}
    { 
      \draw[very thin,color=gray,step=1/\m] (-0.05,0) grid (0.05,3.5);
      \draw[very thin,color=gray,step=1] (-0.1,0) grid (0.1,3);
      \draw[very thin,color=gray,step=1/\m] (0,-0.05) grid (1,0.05);
      \draw[very thin,color=gray,step=1] (0,-0.1) grid (1,0.1);
      \draw[->] (0,0) -- (1.3,0) node[right] {$\alpha$};
      \draw[->] (0,0) -- (0,3.8) node[above] {$\lambda$};
      \node at (-0.2,2) {$2$};
      \node at (1.1,0.2) {$1$};

      \node at (0.7,3.8) {$m=\m$};

      \node(A) at (1,3.5) {};
      \node(B) at (1,1) {};
      \node(C) at ({1-1/\m},0) {};
      \node(wtC) at ({\m/(\m+1)},{(\m-1)/(\m+1)})  {};
      \node(D) at (0,0) {};
      \node(wtD) at (0,4) {};   
      \node(E) at ({1-2/\m},1) {};
      \node(wtE) at ({(\m-2)/(\m-1)},1) {};
      \node(F) at ({1-2/\m},3.5) {};
      \node(wtF) at ({(\m-2)/(\m-1)},{(\m+1)/(\m-1)}) {};

      \filldraw[fill=gray!40!white, draw=gray!40!]
      (1,3.5) -- (1,1) -- (1-1/\m,0) %
      -- (1-2/\m,1) -- (1-2/\m,3.5);

      \filldraw[fill=gray!20!white, draw=gray!20!]
      (0,0) -- (0,3.5) -- (1-2/\m,3.5) %
      -- (1-2/\m,1) -- (1-1/\m,0.01) -- (0,0.01);

      \filldraw[fill=gray!60!white, draw=gray!60!]
      (wtF) -- (1,1) -- ({\m/(\m+1)},{(\m-1)/(\m+1)}) %
      -- ({(\m-2)/(\m-1)},1) -- ({(\m-2)/(\m-1)},{(\m+1)/(\m-1)});               ;

      \node at (A) [vertex,label=right:$A$] {};
      \node at (B) [vertex,label=right:$B$] {};
      \node at (C) [vertex,label=below:$C$] {};
      \node at (wtC) [vertex,label=right:$\wt C$] {};
      \node at (D) [vertex,label=left:$D$] {};
      \node at (F) [vertex,label=left:$F$] {};
      \node at (wtF) [vertex,label=left:$\wt F$] {};%

      \ifthenelse{\m=2}%
      {
        \node at (wtE) [vertex,label=left:${E=\wt E}$] {};

        \draw[thick,dotted]%
       (D) -- (E) -- (C) -- (D);
        \node at (0.7,4.2) {\Cor{handles3}};
      }
      {
        \node at (E) [vertex,label=left:$E$] {};
        \node at (wtE) [vertex,label={[label distance=0.2mm]267:${\wt E}$}] {};
        \ifthenelse{\m=3}%
        {
          \draw[thick,dotted]%
          (0,3.5) -- (D) -- (0.5,0) -- (0.33,0.33) -- (0.33,3.5);
        }
        {
          \ifthenelse{\m=4}%
          {
            \draw[thick,dotted]%
            (0,3.5) -- (D) -- (0.5,0) -- (0.5,3.5);
          }
          {
            \draw[thick,dotted]%
            (0,3.5) -- (D) -- (0.5,0) -- (0.5,3.5);
          }
        }
      }

      \draw[thick, dashed](wtC) -- (wtE) -- (wtF) -- (B) -- (wtC);
    }
  \end{tikzpicture}
}

\newcommand{\DrawParRangeAdhering}[1] 
{
  \begin{tikzpicture}[scale=2, inner sep=2pt,
    vertex/.style={circle,draw=black!50,%
      fill=black!50, 
      inner sep=0pt,
      minimum size=1.5mm}
    ]

    \newcommand{\Myinfty}{3.2}
    \newcommand{\alphaScale}{1}
    \foreach \m in {#1}
    {
      \coordinate(A) at (\alphaScale,\Myinfty);
      \coordinate(B) at (\alphaScale,1);
      \coordinate(C) at ({\alphaScale*(1-1/\m)},0);
      \coordinate(wtC) at ({\alphaScale*(2*(\m-1))/(2*\m-1)},{(\m-1)/(2*\m-1)}) {};
      \coordinate(D) at (0,0); 
      \coordinate(Dp) at (0,1);
      \coordinate(E) at ({\alphaScale*(1-2/\m)},1);
      \coordinate(wtE) at ({\alphaScale*(\m-2)/(\m-1)},1);
      \coordinate(F) at ({\alphaScale*(1-2/\m)},{\Myinfty});
      \coordinate(Fp) at ({\alphaScale*(1-2/\m)},3);
      \coordinate(wtF) at ({\alphaScale*(\m-2)/(\m-1)},{(2*\m-1)/(\m-1)}) {};

      \filldraw[fill=gray!30!, draw=gray!30!]
      (C)--(B)--(A)--(F)--(E)--(C);

      \filldraw[fill=gray!60!, draw=gray!60!]
      (wtC)--(B)--(wtF)--(wtE)--(wtC);

      \draw[dotted,color=gray] ({(\m-2)/(\m-1)},0.00) -- ({(\m-2)/(\m-1)},1.1);

      \node at (A) [vertex,label=right:$A$,yshift=9.5] {};
      \draw[dotted](A)--({\alphaScale},{\Myinfty+0.2});
      \node at (C) [vertex,label=above:$C$] {};
      \node at (D) [vertex,label=left:$D$] {};
      \node at (F) [vertex,label=left:$F$,yshift=9.5] {};
      \draw[dotted](F)--({\alphaScale*(1-2/\m)},{\Myinfty+0.2});

      \node at (B) [vertex,label=right:$B$] {};
      \node at (wtC) [vertex,label=right:$\wt C$] {};
      \node at (wtF) [vertex,label=left:$\wt F$] {};

      \ifthenelse{\m=2}
      {
        \node at (0.6,2.8) {$m=\m$};
        \node at (E) [vertex,label=right:${E=\wt E}$] {};
      }
      {
        \node at (wtE) [vertex,label=above:$\wt E$] {};
        \node at (E) [vertex,label=left:$E$] {};
      }
      \ifthenelse{\m=3}
      {
        \node at (0.6,2.8) {$m=\m$};
        \node at ({(\m-2)/(\m-1)},-0.25) {$\frac12$};
      }
      {}
      \ifthenelse{\m=4}
      {
        \node at (0.6,2.8) {$m=\m$};
        \node at ({(\m-2)/(\m-1)},-0.25) {$\frac23$};
      }{}
      \ifthenelse{\m=5} 
      {
        \node at (0.6,2.8) {$m\ge5$};
        \node at ({(\m-2)/(\m-1)}, -0.25) {$\frac{m-2}{m-1}$};
      }
      {
      }

      \draw[very thin,color=gray,step=1/\m] (0,-0.05) grid (1.0,0.05);
      \draw[->] (0,0) -- (1.1,0) node[right] {$\alpha$};
      \node at (1,-0.25) {$1$};

      \draw[very thin,color=gray,step=1/3] (-0.05,0) grid (0.05,3.5);
      \draw[very thin,color=gray,step=1] (-0.1,0) grid (0.1,1);
      \draw[->] (0,0) -- (0,3.5) node[above] {$\lambda$};
    }
  \end{tikzpicture}
}

\begin{figure}[h]
  \centering
  \DrawParRangeAdhering{2}\quad
  \DrawParRangeAdhering{3}\quad
  \DrawParRangeAdhering{4}\quad
  \DrawParRangeAdhering{5}\quad

  \caption{The range inside $B\wt F\wt E\wt C$ is the range covered by
    our adhering case \Cor{handles3} (dark grey).  The parameter range
    covered by~\cite[Thm.~2.6]{khrabustovskyi:13} is given by the
    (unbounded) polygon $ABCEF$ (lighter grey), see \Rem{andrii}.}
\label{fig:par-range}
\end{figure}
\begin{corollary}[adhering handles]
  \label{cor:handles3}
  Within the setting of \Thm{handles3}, we assume that $\bd \Omega$ is
  compact, that the uniform cover distance is
  $\eta_\eps=\eps^\alpha r_0$, that the handle length fulfils
  $\ell_\eps=\eps^\lambda$ and that \Ass{isometric} holds with
  $\wt \eps=\eps^\alpha r_1$ for some $r_1\in (0,1)$.  If
  \begin{enumerate}
  \item[(0)] $0 \le \alpha < 1$,
  \item
    $\lambda>0$,
  \item
    $\lambda>m\alpha-(m-1)$,
  \item
    if $\lambda \ge 1$ then $\alpha \ge \frac{m-2}{m-1}$,
    if $\lambda \le 1$ then $\lambda>(m-1)(1-\alpha)$,

  \item
    $\lambda<(m+1)-m\alpha$.
  \end{enumerate}
  Then the energy form $\qf d_\eps$ on the manifold $M_\eps$ with
  small handles is $\delta_\eps$-quasi-unitary equivalent of order $2$
  with the identifying energy form $\qf d_0$ defined above with
  partially isometric identification operators.
\end{corollary}
\begin{remarks}[adhering handles: meaning of the conditions II]
  We comment now on the conditions of \Cor{handles3}:
  \begin{enumerate}
  \item (0) and~\itemref{handles3.a} are used as usual to ensure that
    the handle length shrinks to $0$ and that $\deltaBall \eps \to 0$,
    see \Remenum{handles0}{rem.handles0.a}.

  \item Moreover, from (0) and~\itemref{handles3.b} we conclude that
    $\deltaHarm \eps \to 0$, needed for
    \Assenum{handle.len}{handles3.b}.  This condition gives the line
    segment $B\wt C$ in \Fig{par-range}.

  \item Moreover, \itemref{handles3.c} and $\alpha<1$ ensure that
    $\deltaAntisym \eps \to 0$, this is needed for
    \Assenum{handle.len}{handles3.c}, and also
    $\deltaAntisym \eps/\eps^\alpha \to 0$, needed for
    \Assenum{isometric}{isometric.c}.  The condition with
    $\lambda \ge 1$ gives the line segment $\wt E \wt F$ in
    \Fig{par-range}, the other the line segment $\wt E \wt C$.

  \item Finally, \itemref{handles3.d}, \itemref{handles3.c} and
    $\alpha<1$ ensure that $\deltaHarm \eps^\perp \to 0$ needed for
    \Assenum{handle.len}{handles3.d}.  This gives the line segment
    $B \wt F$ in \Fig{par-range}.
  \end{enumerate}
\end{remarks}
\subsection{Examples and further remarks}
\label{ssec:examples}

\subsubsection*{The fading case: graph-like manifolds}
\begin{example}[The case $\alpha=0$: Graph-like manifolds, spectral
  gaps and eigenvalues in gaps]
  Let $(V,E)$ be an infinite graph of bounded degree and let $M_v$ be
  an isometric copy of a compact Riemannian manifold $(M,g)$ for each
  vertex $v \in V$.  Let $X=\bigdcup_{v \in V} M_v$.  For each
  $v \in V$ let $I_v \subset M_v$ be an $\eta$-separated set with
  $\deg v$ elements (as the degree $\deg v$ is uniformly bounded,
  $\eta>0$ can be chosen to be independent of $v \in V$).  Let
  $I\coloneqq \bigdcup_{v \in V} I_v$.  As each edge $e \in E$ is
  adjacent with two vertices, we have a natural splitting
  $I=I^- \dcup I^+$, and $I^-$ can be identified with $E$.  Let now
  $M_\eps$ be the manifold with handles obtained from $X$ by gluing
  handles of radius $\eps>0$ and length $\ell_\eps=\eps^\lambda$ for
  some $\lambda \in (0,1)$ as above.  Here, the set of points
  $I=I_\eps$ is independent of $\eps$, hence we can choose
  $\eta_\eps=r_0$.

  The resulting manifold can be seen as a graph-like manifold obtained
  from the graph $(V,E)$ by replacing each vertex by a copy $M_v$ of
  $M$ and each edge joining the vertices $v$ and $w$ by a thin short
  handle from $M_v$ to $M_w$.  Moreover, \Cor{handles0} applies for
  all $\lambda>0$.  The estimate on the convergence speed is better
  using \Cor{handles1} for $\lambda \in (0,1)$ ($m=2$) and
  $\lambda \in (\alpha_m/2,1-\alpha_m)$ ($m \ge 3$), see
  \Fig{alpha-lambda} with $\alpha=0$.

  \begin{enumerate}
  \item \myparagraph{Spectral gaps:} The spectrum of
    $X$ is the one of
    $M$ but each eigenvalue now has infinite multiplicity.  Moreover,
    the spectrum of $M_\eps$ converges to the spectrum of
    $X$ in any compact interval
    $[0,\Lambda]$ (cf.\ \Thm{spectrum}).  In particular, we have
    constructed a Riemannian manifold $M_\eps$ such that for each $k
    \in \N^*$ there is $\eps_k$ ($\eps_k \to 0$ as $k \to
    \infty$) such that $\spec {M_{\eps_k}}$ has at least
    $k$ spectral gaps.

  \item \myparagraph{Eigenvalues in gaps:} If we change $M_v$ at one
    vertex and put there another manifold $N$ of the same dimension
    (call the resulting graph-like manifold $\wt M_\eps$) the
    essential spectrum of $M_\eps$ and $\wt M_\eps$ agree.  This is
    due to the fact that the perturbation appears only on a compact
    set (see e.g.~\cite[Thm.~4.1]{post:03b}).  If $M$ and $N$ have
    different spectra then in the limit, there must be discrete
    spectrum, i.e., eigenvalues of finite multiplicity outside the
    essential spectrum.
\end{enumerate}
  In particular, our results extends the
  results~\cite{post:03a,post:03b,lledo-post:08} for (perturbed)
  periodic manifolds to general graph-like manifolds.
\end{example}

\subsubsection*{The adhering case: two identical copies}
Let us first treat the situation, when $X=\Omega^-\dcup \Omega^+$
consists only of two isometric copies $\Omega^+$ and $\Omega^-$ of a
given manifold.  If $\Omega^\pm$ has boundary, one has to take care of
the position of the points $I_\eps^\pm$ in order to avoid to be too
close to the boundary, see e.g.\ \eqref{eq:def.ieps.flat} for concrete
example.

In this simpler situation, we can drop some of the assumptions of
\Thm{handles3} and obtain better results.  In particular, this
situation is the one (with $\Omega^+$ and $\Omega^-$ being subsets of
$\R^m$) considered in~\cite{khrabustovskyi:13}.  As there are no
points outside $\Omega$, we do not need \Ass{unif.tub.nbhd}.
Moreover, from \Ass{isometric} we only need~\itemref{isometric.b}
(as~\itemref{isometric.a} is automatically fulfilled: $\Omega^+$ and
$\Omega^-$ are already isometric by assumption).

But most important, we can drop~\itemref{isometric.c} and just have to
assume that $\deltaAntisym \eps \to 0$.  The result on adhering
handles now can be formulated stronger (formally, by setting
$\wt \eps=0$ in parts of the proof of~\Thm{handles3}).  In particular,
we do not need the assumption
$\deltaAntisym \eps/\sqrt{\wt \eps} \to 0$:

\begin{theorem}[adhering handles identifying two copies]
  \label{thm:handles4}
  Let $X=\Omega^-\dcup \Omega^+$ consist of two isometric copies of a
  given manifold. 
  Assume that the parameters $\eps$, $\eta_\eps$ and $\ell_\eps$
  fulfil \Ass{handle.len}.  Moreover, assume
  that~\Assenum{isometric}{isometric.b} is fulfilled, i.e., that
  $I^\pm_\eps \subset \Omega^\pm$ has an $N$-uniform $\eta_\eps$-cover
  (N independent on $\eps$).  Then the energy form $\qf d_\eps$
  defined in \Def{qeps} is $\delta_\eps$-quasi-unitary equivalent of
  order $2$ with the identifying energy form $\qf d_0$ defined above
  with partially isometric identification operators, where
  $\delta_\eps = \Err(\deltaAntisym \eps) + \Err(\deltaBall \eps) +
  \Err(\deltaHandle \eps) + \Err(\deltaHarm \eps^\perp) \to 0$ as
  $\eps \to 0$.
\end{theorem}
\begin{proof}
  The main proof is as in the proof of \Thm{handles3} given in
  \Subsec{proof.main3}.  We only indicate the differences: as
  $\Omega_{\wt \eps} \setminus \Omega=\emptyset$ here, we only have
  $(2\Cnbhd' \deltaAntisym \eps)^2$ times the norm terms
  in~\eqref{eq:handles3.est5}.  In particular, the first term
  in~\eqref{eq:err.handles3} with $\wt \eps$ can be replaced by
  $\Err(\deltaAntisym \eps)$.
\end{proof}
In terms of the concrete parameter dependence, we have:
\begin{corollary}[adhering handles: two identical copies]
  \label{cor:handles4}
  Let $X=\Omega^-\dcup \Omega^+$ consist of two isometric copies of a
  given manifold.  Moreover, we assume that the uniform cover distance
  is $\eta_\eps=\eps^\alpha r_0$, and that the handle length fulfils
  $\ell_\eps=\eps^\lambda$ .  If
  \begin{enumerate}
  \item[(0)] $0 \le \alpha < 1$,
  \item
    $\lambda>0$,
  \item
    $\lambda>m\alpha-(m-1)$ (this gives the line segment $BC$ in
    \Fig{par-range-simple}),
  \item
    if $\lambda \ge 1$ then $\alpha \ge \frac{m-2}m$ (line segment
    $E ((m-2)/(m-1),\infty)$ in \Fig{par-range-simple}), if
    $\lambda \le 1$ then $\lambda> m-1-m\alpha)$ (line segment
    $EC$ in \Fig{par-range-simple}),
  \item
    $\lambda<(m+1)-m\alpha$ (this gives the line segment $\wt
    FB$ in \Fig{par-range-simple}).
  \end{enumerate}
  Then the energy form $\qf d_\eps$ on the manifold $M_\eps$ with
  small handles is $\delta_\eps$-quasi-unitary equivalent of order $2$
  with the identifying energy form $\qf d_0$ defined above with
  partially isometric identification operators.
\end{corollary}

\subsubsection*{Adhering case: flat identified parts --- concrete split sets}
Let $X$ be a Riemannian manifold (or a subset of $\R^m$ with Dirichlet
or Neumann boundary condition at $\bd X$).  Assume that $\Omega^\pm$
are two isometric (with isometry $\phi$), flat, open and disjoint
subsets of $X$, i.e., without loss of generality we consider
$\Omega^\pm$ as subsets of $\R^m$.  We assume that $\bd \Omega^+$ and
$\bd \Omega^-$ have distance at least $2r_1>0$ and both admit a
uniform exterior tubular $r_1$-neighbourhood.  Here, we need to
strengthen the assumption and assume that we also have a uniform
\emph{interior} tubular $r_1$-neighbourhood, i.e., the map $\exp$
in~\eqref{eq:expnormal} is a diffeomorphism for $t \in (-r_1,r_1)$.

We now define the $\eps$-separated split set and $\wt \eps$ as
follows:
\begin{equation}
  \label{eq:def.ieps.flat}
  I_\eps^-
  := \set{p \in  a \eta_\eps \Z^m}{B_\eps(p) \subset \Omega^-}
  \qquadtext{and}
  \wt \eps = \eta_\eps/3;
\end{equation}
and we set $I_\eps^+ := \phi(I_\eps^-)$.  Later, we fix
$a=1/(2\sqrt m)$.  We then have:
\begin{lemma}
  \label{lem:flat.example}
  If $\eps/\eta_\eps \to 0$ as $\eps \to 0$ and if $\bd \Omega$ admits
  a uniform exterior and interior uniform tubular $r_1$-neighbourhood
  with $\eta_\eps/3 \le r_1$ then $\wt \eps = \eta_\eps/3$ and
  $I_\eps$ as above fulfil \Ass{isometric}.
\end{lemma}
\begin{proof}
  We have to show that $I_\eps$ has $\eta_\eps$-cover number $N$ and
  that
  $\Omega_{\wt \eps} \subset \bigcup_{p \in I_\eps} B_{\eta_\eps(p)}$.

  To start with the first assertion, we have to bound the cardinality of
  \begin{align*}
    \set{p \in I_\eps}{x \in B_{\eta_\eps}(p)}.
  \end{align*}
  for any $x \in \Omega$; it suffices to show this for $x \in \R^m$.
  As the maximal distance of a point $x$ to the grid
  $a\eta_\eps \Z^m$ is $(\sqrt m/2)a \eta_\eps$ it is enough to
  bound the cardinality of the set
  \begin{align*}
    \bigset{p \in I_\eps}{ \abs{q-p} \le (1+a \sqrt m/2)\eta_\eps}
  \end{align*}
  for some $q \in a\eta_\eps \Z^m$; without loss of generality again
  consider $q=0$.  But the latter number is bounded by the points in a
  hypercube with side length $(2+a\sqrt m+1)$, hence the cover number
  $N=(2+a\sqrt m+1)^m=(7/2)^m$ will work.

  For the second assertion, let $x \in \Omega_{\wt \eps}$, then there
  is $x_0 \in \Omega$ such that $d(x,x_0) <\wt \eps=\eta_\eps/3$.  As
  $\bd \Omega$ has a uniform \emph{interior} tubular
  $r_1$-neighbourhood and $\eta_\eps/3 \le r_1$, we can assume that
  there is $x_1 \in \Omega$ such that $x_0 \in B_{\eta_\eps/3}(x_1)$
  and $B_{\eta_\eps/3}(x_1) \subset \Omega$.  Now
  $B_{\eta_\eps/3}(x_1)$ contains at least one point $p \in I_\eps$ if
  $\eps$ is small enough (note that $\eps/\eta_\eps \to 0$, and that
  $1/3> a\sqrt m/2=1/4$).  Altogether we have
  \begin{equation*}
    d(p,x) \le \abs{p-x_1} + \abs{x_1-x_0} + d(x_0,x_1)
    < \frac {\eta_\eps}3 + \frac {\eta_\eps}3 +\frac {\eta_\eps}3
    =\eta_\eps. \qedhere
  \end{equation*}
\end{proof}

\subsubsection*{Comparing the results with Khrabustovskyi's results}

\newcommand{\DrawParRangeAdheringSimple}[1] 
{
  \begin{tikzpicture}[scale=2, inner sep=2pt,
    vertex/.style={circle,draw=black!50,%
      fill=black!50, 
      inner sep=0pt,
      minimum size=1.5mm}
    ]

    \newcommand{\Myinfty}{3.2}
    \newcommand{\alphaScale}{1}
    \foreach \m in {#1}
    {
      \coordinate(A) at (\alphaScale,\Myinfty);
      \coordinate(B) at (\alphaScale,1);
      \coordinate(C) at ({\alphaScale*(1-1/\m)},0);
      \coordinate(wtC) at ({\alphaScale*(2*(\m-1))/(2*\m-1)},{(\m-1)/(2*\m-1)}) {};
      \coordinate(D) at (0,0); 
      \coordinate(Dp) at (0,1);
      \coordinate(E) at ({\alphaScale*(1-2/\m)},1);
      \coordinate(wtE) at ({\alphaScale*(\m-2)/(\m-1)},1);
      \coordinate(F) at ({\alphaScale*(1-2/\m)},{\Myinfty});
      \coordinate(Fp) at ({\alphaScale*(1-2/\m)},3);
      \coordinate(wtF) at ({\alphaScale*(\m-2)/(\m-1)},{(2*\m-1)/(\m-1)}) {};

      \filldraw[fill=gray!30!, draw=gray!30!]
      (C)--(B)--(A)--(F)--(E)--(C);

      \filldraw[fill=gray!60!, draw=gray!60!]
      (C)--(B)--(Fp)--(E)--(C);

      \draw[dotted,color=gray] ({(\m-2)/(\m-1)},0.00) -- ({(\m-2)/(\m-1)},1.1);

      \node at (A) [vertex,label=right:$A$,yshift=9.5] {};
      \draw[dotted](A)--({\alphaScale},{\Myinfty+0.2});
      \node at (C) [vertex,label=below:$C$] {};
      \node at (D) [vertex,label=left:$D$] {};
      \node at (E) [vertex,label=left:$E$] {};
      \node at (F) [vertex,label=left:$F$,yshift=9.5] {};
      \draw[dotted](F)--({\alphaScale*(1-2/\m)},{\Myinfty+0.2});
      \node at (Fp) [vertex,label=left:$F'$] {};

      \node at (B) [vertex,label=right:$B$] {};

      \ifthenelse{\m=2}
      {
        \node at (0.6,2.8) {$m=\m$};
      }
      {
      }
      \ifthenelse{\m=3}
      {
        \node at (0.6,2.8) {$m=\m$};
      }
      {}
      \ifthenelse{\m=4}
      {
        \node at (0.6,2.8) {$m=\m$};
      }{}
      \ifthenelse{\m=5} 
      {
        \node at (0.6,2.8) {$m\ge5$};
      }
      {
      }

      \draw[very thin,color=gray,step=1/\m] (0,-0.05) grid (1.0,0.05);
      \draw[->] (0,0) -- (1.1,0) node[right] {$\alpha$};
      \node at (1,-0.25) {$1$};

      \draw[very thin,color=gray,step=1/3] (-0.05,0) grid (0.05,3.5);
      \draw[very thin,color=gray,step=1] (-0.1,0) grid (0.1,1);
      \draw[->] (0,0) -- (0,3.5) node[above] {$\lambda$};
    }
  \end{tikzpicture}
}
\newcommand{\DrawParRangeKhrSimple}[1]%
{
  \begin{tikzpicture}[scale=1.5, inner sep=2pt,
    vertex/.style={circle,draw=black!50,%
      fill=black!50, 
      inner sep=0pt,
      minimum size=1.5mm}
    ]
    \foreach \m in {#1}
    { 
      \draw[very thin,color=gray,step=1/\m] (-0.05,0) grid (0.05,3.5);
      \draw[very thin,color=gray,step=1] (-0.1,0) grid (0.1,3);
      \draw[very thin,color=gray,step=1/\m] (0,-0.05) grid (1,0.05);
      \draw[very thin,color=gray,step=1] (0,-0.1) grid (1,0.1);
      \draw[->] (0,0) -- (1.3,0) node[right] {$\alpha$};
      \draw[->] (0,0) -- (0,3.8) node[above] {$\lambda$};
      \node at (-0.2,2) {$2$};
      \node at (1.1,0.2) {$1$};

      \node at (0.7,3.8) {$m=\m$};

      \node(A) at (1,3.5) {};
      \node(B) at (1,1) {};
      \node(C) at (1-1/\m,0) {};
      \node(wtC) at ({\m/(\m+2)},{(\m-2)/(\m+1)})  {};
      \node(D) at (0,0) {};
      \node(wtD) at (0,3.5) {};   
      \node(E) at (1-2/\m,1) {};
      \node(F) at (1-2/\m,3.5) {};
      \node(wtF) at (1-2/\m,1+4/\m) {};
      \node(G) at (1-2/\m,0) {};
      \node(wtG) at (1/2,0) {};

      \filldraw[fill=gray!40!white, draw=gray!40!]
      (1,3.5) -- (1,1) -- (1-1/\m,0) %
      -- (1-2/\m,1) -- (1-2/\m,3.5);

      \filldraw[fill=gray!20!white, draw=gray!20!]
      (0,0) -- (0,3.5) -- (1-2/\m,3.5) %
      -- (1-2/\m,1) -- (1-1/\m,0.01) -- (0,0.01);

      \filldraw[fill=gray!60!white, draw=gray!60!]
      (1-2/\m,1+4/\m) -- (1,1) -- ({\m/(\m+2)},{(\m-2)/(\m+1)}) %
      -- (1-2/\m,1) -- (1-2/\m,1+4/\m);

      \node at (A) [vertex,label=right:$A$] {};
      \node at (B) [vertex,label=right:$B$] {};
      \node at (E) [vertex,label=left:$E$] {};
      \node at (F) [vertex,label=left:$F$] {};
      \node at (wtF) [vertex,label=left:$F'$] {};%

      \ifthenelse{\m=2}%
      {
        \node at (C) [vertex,label=below:${C=C'}$] {};
        \node (D) at (0,0) [vertex,label=left:$D$] {};
        \node (WtF) at (0,3.5) [vertex] {};
        \node (WtG) at (0,0) [vertex] {};
        \node at (0.7,4.3) {\Cor{handles4}};

        \draw[thick,dotted]%
       (D) -- (E) -- (C) -- (D);
      }
      {
        \node at (C) [vertex,label=below:$C$] {};
        \ifthenelse{\m=3}%
        {
          \node at (wtC) [vertex,label=right:$C'$] {};
          \node (WtF) at (0.333,3.5) [vertex] {};

        }
        {
          \ifthenelse{\m=4}%
          {
            \node at (wtC) [vertex,label=right:$C'$] {};

          }
          {
            \node at (wtC) [vertex,label=right:$C'$] {};

            \draw[thick,dotted]%
            (0,3.5) -- (D) -- (0.5,0) -- (0.5,3.5);
          }
        }
      }
      \draw[thick, dashed](wtC) -- (E) -- (wtF) -- (B) -- (wtC);
    }
  \end{tikzpicture}
}

\begin{figure}[h]
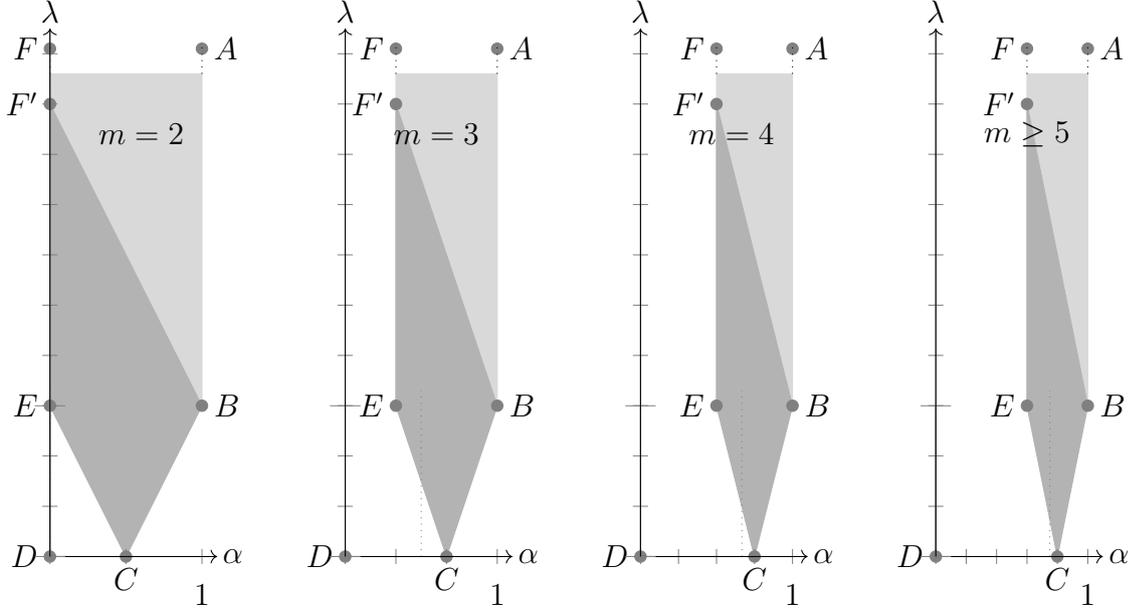

  \centering
  \DrawParRangeAdheringSimple{2}\quad
  \DrawParRangeAdheringSimple{3}\quad
  \DrawParRangeAdheringSimple{4}\quad
  \DrawParRangeAdheringSimple{5}

  \caption{The range inside $BFEC$ is the range covered by our
    simplified adhering case \Cor{handles4} (dark grey).  Here,
    ``simplified'' means that $X$ consists only of two isometric
    manifolds $\Omega^-$ and $\Omega^+$.  The parameter range covered
    by~\cite[Thm.~2.6]{khrabustovskyi:13} is again inside the polygon
    $ABCEF$ (lighter grey).}
\label{fig:par-range-simple}
\end{figure}

\begin{remark}[Comparision with results by
  Khrabustovskyi~\cite{khrabustovskyi:13}]
  \label{rem:andrii}
  In~\cite{khrabustovskyi:13} Khrabustovskyi considers a special case:
  the manifold $M_\eps$ is obtained from two isometric copies
  $\Omega^-$ and $\Omega^+$ of an open subset of $\R^m$ joint by
  handles glued at balls with \emph{periodically} placed centres
  $I_\eps \subset \eta_\eps\Z^m$.  This is a combination of the
  previous two situations (two identical copies and flat identified
  parts).

  Here is a tabular showing the correspondence of our parameters and
  the ones in~\cite{khrabustovskyi:13}:
  \begin{center}
    \renewcommand{\arraystretch}{1.2}  
    \begin{tabular}{|l|l|l|}
      \hline
      quantity & this article & \cite{khrabustovskyi:13}\\
      \hline\hline
      handle radius & $\eps$ & $d^\eps=\bm d\eps^{\wt \alpha}$\\\hline
      handle radius exponent & $1$ & $\wt \alpha=1/\alpha$\\\hline
      uniform cover distance & $\eta_\eps=r_0 \eps^\alpha$ & $\eps$ \\\hline
      uniform cover distance exponent & $\alpha=1/\wt \alpha$ & $1$ \\\hline
      handle length & $\ell_\eps=\eps^\lambda$ & $q^\eps=\bm q\eps^\beta$\\\hline
      handle length exponent  & $\lambda=\beta/\wt \alpha$
                              & $\beta=\lambda/\alpha$\\\hline
    \end{tabular}
  \end{center}
  Moreover, the
  parameters $(\alpha,\lambda)$ transform into
  $(\wt \alpha,\beta)=(1/\alpha,\lambda/\alpha)$
  in~\cite[Ex.~2.8]{khrabustovskyi:13}.  We use the same labels as
  in~\cite[Fig.~2]{khrabustovskyi:13} for the parameter points (in our
  notation) given by
  \begin{align*}
    A& =(1,\infty),
    & B &=(1,1),
    & C &=\Bigl(\frac{m-1}m,0\Bigr),\\
    D &=(0,0),
    & E &=\Bigl(\frac{m-2}m,1\Bigr),
    &F &=\Bigl(\frac{m-2}m,\infty\Bigr).\\
  \intertext{Moreover, we need}
    \wt C &=\Bigl(\frac{2(m-1)}{2m-1}, \frac {m-1}{2m-1}\Bigr),
    & \wt E &=\Bigl(\frac{m-2}{m-1},1\Bigr),
    & \wt F &=\Bigl(\frac{m-2}{m-1},\frac{2m-1}{m-1}\Bigr),\\
    F' &=\Bigl(\frac{m-2}m,3\Bigr)
  \end{align*}
  for our results (see \Figs{par-range}{par-range-simple}).
  \begin{enumerate}
  \item
    \label{andrii.a}
    \myparagraph{Fading case:} %
    In Theorem~2.6 of~\cite{khrabustovskyi:13}, the limit Hilbert
    space is $\HS=\Lsqr X= \Lsqr{\Omega^-} \oplus \Lsqr{\Omega^+}$
    covering the unbounded polygon $F(0,\infty)DCEF$ corresponding to
    our fading case only covering the light grey area.
    Our results do not cover the infinite rectangle
    $EE'(\alpha_m,\infty)((m-2)/m,\infty)$ (dotted area) in
    \Fig{alpha-lambda}).
  \item
    \label{andrii.b}
    An interesting case with a coupled operator appears on the
    polygonal line segment $CEF$ in \Fig{par-range-simple}.  In this
    case, the limit operator acts still on the two copies, but with a
    \emph{coupling} between them ($V>0$
    in~\cite[Thm.~2.6]{khrabustovskyi:13}).

  \item
    \label{andrii.c}
    \myparagraph{Adhering case:} %
    \cite[Thm.~2.5]{khrabustovskyi:13} covers the open unbounded
    polygon $ABCEFA$ corresponding to our adhering case~\Thm{handles4}
    (see also \Fig{par-range-simple}).

    The discrepancy with the optimal parameter range used in the more
    general result \Thm{handles3} lead to a worse parameter region
    right of $\wt C\wt E\wt F$ in \Fig{par-range} (instead of the
    optimal $CEF$ in \Fig{par-range-simple}) is due to our more
    general result (we allow parts which are not identified, leading
    to the additional condition
    ($\deltaAntisym \eps/\sqrt {\wt\eps} \to 0$).  If we consider the
    simplified adhering case (similar to the one treated
    in~\cite{khrabustovskyi:13}), then we do not cover the parameter
    range given by the infinite polygon $ABF'F$ in
    \Fig{par-range-simple}.
  \item
    \label{andrii.d}
    The (open) segment from $C$ to $B$ is the case $p \in (0,\infty)$
    (in the notation of~\cite{khrabustovskyi:13}) corresponding to a
    homogenised problem with non-trivial shift in the spectrum.
\end{enumerate}
  Note that Khrabustovskyi actually shows a weaker result, namely only
  the convergence of the spectra, not a generalised norm resolvent
  convergence and convergence speed estimates, as we do here.

  We do not know whether (generalised) norm resolvent convergence in
  the parameter regions not covered by our results can hold or not.
\end{remark}

%
\section{Estimates on harmonic functions on handles}
\label{sec:est.handles}
%

In this section, we collect the estimates on the handles used in the
proofs.  The situation is the one described previously in
\Subsec{many.small.handles} with the same notations.  We need estimates
of the harmonic extension onto the handles.
\begin{definition}[harmonic extension on the handle]
  \label{def:harmprol}
  For $u \in \Sob{X_\eps,g}$ we denote by
  $\Phi_\eps u \in \Sob {\Cyl,g_\can}$ the harmonic extension of $u$
  on the handles, i.e., $\Phi_\eps u = h$ minimises $\qf d_\eps(U)$
  among all functions $U=(u,h) \in \dom \qf d_\eps$ for a given
  $u \in \Sob{X_\eps,g}$.
\end{definition}
The notation $\Phi_\eps u$ is justified as the minimiser
$\Phi_\eps u=h=(h_p)_{p \in I_\eps^-}$ is unique for
$u \in \Sob{X_\eps,g}$.  Moreover, $h_p$ fulfils
\begin{equation*}
  \Delta_{\Cyl_\eps} h_p
  \coloneqq
  -\partial^2_1 h_p
  + \Bigl(\frac{\ell_\eps} \eps \Bigr)^2 \Delta_\Sphere h_p
  =0
\end{equation*}
and
\begin{equation}
  \label{eq:a}
  h_p(0,\cdot)
  = \sqrt{\eps^{m-1}\ell_\eps} \cdot u_p(\eps,\cdot)
  \qquadtext{and}
  h_p(1,\cdot)
  = \sqrt{\eps^{m-1}\ell_\eps} \cdot u_{\bar p}(\eps,\cdot)
\end{equation}
for all $p \in I_\eps^-$, where $\partial_1$ acts on the first
(longitudinal) variable and $\Delta_\Sphere$ is the Laplacian on the
sphere acting on the second variable of $h_p$.

Denote by $\Mu \coloneqq \sqrt{\spec \Sphere}$ the
values\footnote{Actually, we know the spectrum of
  $\Sphere=\Sphere^{m-1}$ explicitly: namely, we have
  $\Mu = \set{\sqrt{k(k+m-2)}}{k \in \N_0=\N^*\cup\{0\}}$
  by~\cite[Thm.~III.22.1]{shubin:01}.  Nevertheless, the precise
  values of $\Mu$ are irrelevant; only in the proof of \Lem{est.harm}
  we make use of the fact that the first non-zero eigenvalue of $\Mu$
  fulfils $\mu_1 =\sqrt{m-1} \ge 1$ in order to simplify some
  estimates.}  $\mu \ge 0$ such that $\mu^2$ is an eigenvalue of
$\Sphere$, and by $(\phi_\mu)_{\mu \in \Mu}$ the corresponding
orthonormal basis of eigenfunctions.  Strictly speaking, $\Mu$ is a
so-called \emph{multiset} where an element $\mu$ appears in $\Mu$
multiple times according to the multiplicity of $\mu^2$ in
$\spec \Sphere$.

A separation of variables ansatz gives us
\begin{equation*}
  h_p=\sum_{\mu \in \Mu} h_p^\mu \otimes \phi_\mu,
  \quadtext{or}
  h_p(s,\theta)=\sum_{\mu \in \Mu} h_p^\mu(s) \phi_\mu(\theta)
\end{equation*}
almost everywhere, the convergence is a priori only in $\Lsqr{\Cyl_1}$
(recall that $\Cyl_1\coloneqq[0,1]\times \Sphere$).  As
$h_p \in \Sob{\Cyl_1}$ the trace $h_p(s,\cdot) \in \Lsqr \Sphere$ is
well-defined for $s \in [0,1]$.  In particular, we have
\begin{equation*}
  \iprod[\Lsqr \Sphere] {h_p(s,\cdot)}{\phi_\mu} = h_p^\mu(s),
\end{equation*}
and the boundary conditions at $s=0$ and $s=1$ yield
\begin{align*}
  \sqrt{\eps^{m-1}\ell_\eps} \iprod[\Lsqr \Sphere]{f_p(\eps,\cdot)}{\phi_\mu}
  = \iprod[\Lsqr \Sphere] {h_p(0,\cdot)}{\phi_\mu} = h_p^\mu(0)
  &\eqqcolon a_p^\mu\\
  \sqrt{\eps^{m-1}\ell_\eps} \iprod[\Lsqr \Sphere]{f_{\bar p}(\eps,\cdot)}{\phi_\mu}
  = \iprod[\Lsqr \Sphere] {h_p(1,\cdot)}{\phi_\mu} = h_p^\mu(1)
  &\eqqcolon a_{\bar p}^\mu
\end{align*}
for $p \in I_\eps^-$ (note the bar on $\bar p$ in the second line).  Here,
\begin{equation*}
  h_p^\mu(s)
  =
  \begin{cases}
    a_{\bar p}^\mu s + a_p^\mu (1-s), & \mu=0\\[2ex]
    a_{\bar p}^\mu \dfrac{\sinh (\mu_\eps s)}{\sinh \mu_\eps} + a_p^\mu
    \dfrac{\sinh (\mu_\eps (1-s))}{\sinh \mu_\eps}, & \mu>0,
  \end{cases}
  \quadtext{with}
  \mu_\eps = \frac {\ell_\eps} \eps\cdot \mu,
\end{equation*}
for $p \in I_\eps^-$.  With the notation of the next lemma, we have
\begin{equation*}
  h_p^\mu = h_{a_{\bar p}^\mu,a_p^\mu}^{\mu_\eps}.
\end{equation*}
We need estimates of the longitudinal contribution.  Note that in the
next lemma we use the fact that $\Sphere=\Sphere^{m-1}$ is connected
(for $m \ge 2$), i.e., that $\mu=0$ is a \emph{simple} eigenvalue.
\begin{lemma}
  \label{lem:est.h}
  Let $a_\pm \in \C$, $\mu \ge 0$ and
  $\map {h^\mu=h^\mu_{a_+,a_-}} {[0,1]} \C$ be the function given by
  \begin{equation*}
    h^\mu(s)=
    \begin{cases}
      a_+ s + a_- (1-s), & \mu=0,\\[2ex]
      a_+ \dfrac{\sinh(s \mu)}{\sinh \mu}
      + a_- \dfrac{\sinh((1-s) \mu)}{\sinh \mu}, & \mu>0.
    \end{cases}
  \end{equation*}
  Then we have
  \begin{equation}
    \label{eq:est.h.l2}
    \normsqr[{\Lsqr{[0,1]}}] {h^\mu}
    \le
    \begin{cases}
      \dfrac 12 \bigl(\abssqr{a_+}+\abssqr{a_-}\bigr), & \mu=0,\\[2ex]
      \dfrac 2{3\mu} \bigl(\abssqr{a_+}+\abssqr{a_-}\bigr), & \mu>0
    \end{cases}
  \end{equation}
  and
  \begin{equation}
    \label{eq:est.h'.l2}
    \normsqr[{\Lsqr{[0,1]}}] {\partial_1 h^\mu}
    \le
    \begin{cases}
      2\bigl(\abssqr{a_+}+\abssqr{a_-}\bigr), & \mu=0,\\[2ex]
      (\mu +2) \bigl(\abssqr{a_+}+\abssqr{a_-}\bigr), & \mu>0
    \end{cases}
  \end{equation}
\end{lemma}
\begin{proof}
  We have
  \begin{align*}
    \normsqr
    {h^\mu}
    =
    \begin{cases}
      \dfrac 13 \bigl(\abssqr{a_+}+ \Re(a_+a_-) + \abssqr{a_-}\bigr)
      =\dfrac 16 \abssqr{a_++ a_-}
      + \dfrac 16 \bigl(\abssqr{a_+} + \abssqr{a_-}\bigr),
      &\mu=0\\[2ex]
      \dfrac1\mu\Bigl(
      \underbrace{\dfrac{\sinh(2\mu)-2\mu}{4\sinh^2\mu}}%
      _{\le 1/2}
      \bigl(\abssqr{a_+} + \abssqr{a_-}\bigr)
      + \underbrace{\dfrac{\mu \cosh \mu-\sinh \mu}{\sinh^2\mu}}%
      _{\le 1/3}  
      \Re(a_+a_-)\Bigr),
      &\mu>0
    \end{cases}
  \end{align*}
  and
  \begin{align*}
    \normsqr
    {\partial_1 h^\mu}
    =
    \begin{cases}
      \abssqr{a_+-a_-},
      &\mu=0\\[2ex]
      \mu \Bigl(
      \underbrace{\dfrac{\sinh(2\mu)+2\mu}{4\sinh^2\mu}}%
      _{\le 1+1/\mu}
      \bigl(\abssqr{a_+} + \abssqr{a_-}\bigr)
      + \underbrace{\dfrac{\mu \cosh \mu + \sinh \mu}{\sinh^2\mu}}%
      _{\le 2/\mu}
      \Re(a_+a_-)
      \Bigr),
      &\mu>0.
    \end{cases}
  \end{align*}
  The Cauchy-Young inequality gives the result in both cases.
\end{proof}

Let now $h^\bullet = (h_p^\bullet)_{p \in I_\eps^-}$ and
$h^\perp = (h_p^\perp)_{p \in I_\eps^-}$ with
\begin{align*}
  h_p^\bullet \coloneqq h_p^0 \otimes \phi_0
  \qquadtext{and}
  h_p^\perp = h_p - h_p^\bullet
  = \sum_{\mu \in \Mu \setminus \{0\}} h_p^\mu \otimes \phi_\mu.
\end{align*}
Note that the dependence of $h_p$ and $h_p^\perp$ on $\eps$ is not
explicitly mentioned in the notation.
\begin{lemma}
  \label{lem:est.harm}
  We have
  \begin{subequations}
    \begin{align}
      \label{eq:est.harm.a}
      \normsqr[\Lsqr{\Cyl,g_\can}]{h^\bullet}
      &\le \frac 12 \sum_{p \in I_\eps} \abssqr{a_p^0},\\
      \label{eq:est.harm.b}
      \normsqr[\Lsqr{\Cyl,g_\can}]{h^\perp}
      &\le \frac \eps {\ell_\eps}
        \sum_{p \in I_\eps} \sum_{\mu \in \Mu \setminus \{0\}} \abssqr{a_p^\mu},\\
      \label{eq:est.harm.c}
      \qf d_{\Cyl_\eps}(h^\bullet)
      & = \frac1{\ell_\eps^2} \sum_{p\in I^-_\eps}
        \bigabssqr{a_{\bar p}^0-a_p^0}
        \le \frac2{\ell_\eps^2} \sum_{p\in I_\eps} \abssqr{a_p^0},\\
      \label{eq:est.harm.d}
      \qf d_{\Cyl_\eps}(h^\perp)
      & \le \sum_{p\in I_\eps} \sum_{\mu \in \Mu \setminus \{0\}}
        \Bigl(\frac {\mu+1} {\eps \ell_\eps}
        + \frac2{\ell_\eps^2}
        \Bigr)
        \abssqr{a_p^\mu}.
    \end{align}
  \end{subequations}
\end{lemma}
\begin{proof}
  The first estimate follows directly from~\eqref{eq:est.h.l2}, as
  well as the second, where we have to replace $\mu$ by
  $\ell_\eps/\eps \cdot \mu \ge \ell_\eps/\eps \cdot \mu_1$ provided
  $\mu \in \Mu \setminus \{0\}$.  As the square root $\mu_1$ of the
  smallest non-zero eigenvalue of the unscaled sphere $\Sphere^{m-1}$
  is $\mu_1=\sqrt{m-1} \ge 1$ (\cite[Thm.~III.22.1]{shubin:01}), we use
  the estimate $2/(3\mu_1)\le 1$.

  The third estimate follows from~\eqref{eq:est.h'.l2} and the
  definition of $\qf d_{\Cyl_\eps}$ in~\eqref{eq:qeps.c} and the
  fourth estimate follows from
  \begin{align*}
    \qf d_{\Cyl_\eps(p)}(h^\mu_p)
    &= \frac1{\ell_\eps^2} \normsqr[{\Lsqr{[0,1]}}] {\partial_1 h_p^\mu}
    + \frac \mu{\eps^2}  \normsqr[{\Lsqr{[0,1]}}] {h_p^\mu}\\
    &\le \Bigl(\frac{\mu_\eps+2}{\ell_\eps^2} +
      \frac{2\mu}{3\mu_\eps \eps^2}\Bigr)
      \bigl(\abssqr{a_{\bar p}^\mu} + \abssqr{a_p^\mu}\bigr)
      = \Bigl(\frac{\mu+ 2\eps/\ell_\eps}{\eps \ell_\eps}
      +\frac{2}{3\eps \ell_\eps}\Bigr)
      \bigl(\abssqr{a_{\bar p}^\mu} + \abssqr{a_p^\mu}\bigr)\\
      &\le \Bigl(\frac {\mu+1} {\eps \ell_\eps}
      + \frac2{\ell_\eps^2}
      \Bigr)
      \bigl(\abssqr{a_{\bar p}^\mu} + \abssqr{a_p^\mu}\bigr)
  \end{align*}
  for $p \in I_\eps^-$ using
  again~\eqref{eq:est.h.l2}--\eqref{eq:est.h'.l2}, $2/3<1$ and the
  definition $\mu_\eps=(\ell_\eps/\eps)\cdot \mu$.
\end{proof}

For the rest of this section we assume that $I_\eps$ is
$\eps$-separated with $\eta$-cover number $N \in \N^*$.  From
\Lem{est.h}, \eqref{eq:a} and some other estimates we conclude:
\begin{proposition}
  \label{prp:harm}
  Let $\eps, \eta \in (0,r_0)$ such that $0<\eps<\eta$ then we have
    \begin{align*}
      \normsqr[\Lsqr{\Cyl,g_\can}] {\Phi_\eps^\bullet u}
      &\le \ell_\eps \OptSobTr_m(\eps,\eta)^2
        \normsqr[\Sob{B_\eta \setminus B_\eps,g}]u \quad\text{and}\\
      \normsqr[\Lsqr{\Cyl,g_\can}] {\Phi_\eps^\perp u}
      &\le \eps \OptSobTr_m(\eps,\eta)^2
        \normsqr[\Sob{B_\eta \setminus B_\eps,g}]u
    \end{align*}
    for $u \in \Sob{B_\eta \setminus B_\eps}$, where
    \begin{equation*}
      \OptSobTr_m(\eps,\eta)
      := N^{1/2}K^{(m+2)/4} \OptSobTrEucl_m(\eps,\eta),
    \end{equation*}
    and where
    $\OptSobTrEucl_m(\eps,\eta)$ is defined in~\eqref{eq:c.opt'.eucl}.
\end{proposition}
\begin{proof}
  We have
  $\normsqr[\Lsqr{\Cyl,g_\can}] {\Phi^\bullet u} =
  \normsqr[\Lsqr{\Cyl,g_\can}] {h^\bullet} \le (1/2)\sum_{p \in
    I_\eps^-}\abssqr{a_p^0}$ by~\eqref{eq:est.harm.a}.  Moreover,
  \begin{align*}
    \abssqr{a_p^0}
    &\le \normsqr[\Lsqr \Sphere] {h_p(0,\cdot)}\\
    &= \eps^{m-1}\ell_\eps \normsqr[\Lsqr {\Sphere}]{u_p(\eps,\cdot)}
      = \ell_\eps \normsqr[\Lsqr {\bd B_\eps(p),\iota_\eps^* g_\eucl}]
    {u_p(\eps,\cdot)}
    &\text{(by~\eqref{eq:a})}\\
    & \le \ell_\eps \OptSobTr(\bd B_\eps, B_\eta(p)\setminus B_\eps(p))^2
      \normsqr[\Sob {B_\eta(p) \setminus B_\eps(p), g_\eucl}] {u_p}
    &\text{(by~\eqref{eq:opt.sob.tr})}\\
    & \le  \ell_\eps  K^{(m+2)/2} \OptSobTr(\bd B_\eps, B_\eta(p)\setminus B_\eps(p))^2
      \normsqr[\Sob {B_\eta(p) \setminus B_\eps(p), g}] {u_p}
    &\text{(by~\Cor{eucl.metric2})}
  \end{align*}
  Note that
  \begin{align*}
    \OptSobTr(\bd B_\eps, B_\eta(p)\setminus B_\eps(p))
    \le \OptSobTr(\bd B_\eps, B_\eta(p))
    =\OptSobTrEucl_m(\eps,\eta)
  \end{align*}
  by \Lem{bd.ball.est3} and ~\eqref{eq:c.opt'.eucl}.  The same is true
  for $p \in I_\eps^+$ and $h_p(0,\cdot)$ replaced by
  $h_p(1,\cdot)$. Summing over $p \in I_\eps$ gives
  \begin{equation}
    \label{eq:harm.a}
    \sum_{p \in I_\eps}\abssqr{a_p^0}
    \le  \ell_\eps  \OptSobTr_m(\eps,\eta)^2
    \normsqr[\Sob {B_\eta \setminus B_\eps , g}] u.
  \end{equation}
  For the second estimate we have
  \begin{equation*}
    \normsqr[\Lsqr{\Cyl,g_\can}] {\Phi^\perp u}
    = \normsqr[\Lsqr{\Cyl,g_\can}] {h^\perp}
    \le \frac \eps{\ell_\eps} 
    \sum_{\mu \in \Mu \setminus \{0\}}\abssqr{a_p^\mu}
    \le \frac \eps{\ell_\eps} 
    \normsqr[\Lsqr \Sphere]{h_p(0,\cdot)}
  \end{equation*}
  by~\eqref{eq:est.harm.b} for the first inequality.  The remaining
  estimate follows from~\eqref{eq:harm.a} and the arguments just
  mentioned.
\end{proof}

Combining both estimates we obtain:
\begin{corollary}
  \label{cor:harm}
  Let $\eps, \eta=\eta_\eps \in (0,r_0)$ such that $0<\eps<\eta_\eps$
  then we have
  \begin{align*}
    \norm[\Lsqr{\Cyl,g_\can}] {\Phi_\eps u}
    \le \deltaHarm \eps \norm[\Sob{B_\eta \setminus B_\eps,g}]u
  \end{align*}
  for $u \in \Sob{B_\eta \setminus B_\eps}$, where
  \begin{equation*}
    \deltaHarm \eps \coloneqq
    \bigl((\ell_\eps+\eps)\bigr)^{1/2} \OptSobTr_m(\eps,\eta)
    = \Bigl(\Bigl(\frac{\ell_\eps}\eps+1\Bigr)
      \eps \Bigr)^{1/2} \OptSobTr_m(\eps,\eta_\eps)
  \end{equation*}
\end{corollary}

Similarly, we have for the quadratic form (now estimating over the
entire balls $B_\eta$):
\begin{proposition}
  \label{prp:harm2}
  Let $\eps, \eta=\eta_\eps \in (0,r_0)$ such that $0<\eps<\eta$
  then we have
  \begin{align*}
    \qf d_{\Cyl_\eps}(\Phi_\eps^\bullet f)
    & \le (\deltaHarm \eps^\bullet)^2 \normsqr[\Sob{B_\eta,g}] f
      \quadtext{and}
      \qf d_{\Cyl_\eps}(\Phi_\eps^\perp f)
      \le (\deltaHarm \eps^\perp)^2 \normsqr[{\Sob[2]{B_\eta,g}}]f
  \end{align*}
  for $f \in \Sob{B_\eta}$ resp.\ $f \in \Sob[2]{B_\eta}$, where
  \begin{align*}
    \deltaHarm \eps^\bullet
    & \coloneqq
      \bigl(2 \ell_\eps^{-1}\bigr)^{1/2} \OptSobTr_m(\eps,\eta)
      = \Bigl(\frac2 {\eps \ell_\eps} \Bigr)^{1/2} \eps^{1/2}\OptSobTr_m(\eps,\eta)
      \qquad\text{and}\\
      \deltaHarm \eps^\perp
    & \coloneqq
          \Bigl(2(m-1)^{1/2}K^{(m+2)/2}\Bigl(\frac\eps{\ell_\eps}+1\Bigr)
           \Bigr)^{1/2}\OptNonConc_m(\eps,\eta)
  \end{align*}
\end{proposition}
\begin{proof}
  From~\eqref{eq:est.harm.c} and~\eqref{eq:harm.a} we conclude
  \begin{align*}
    \qf d_{\Cyl_\eps}(\Phi_\eps^\bullet f)
    = \qf d_{\Cyl_\eps}(h^\bullet)
    \le \frac 2{\ell_\eps^2} \sum_{p \in I_\eps}\abssqr{a_p^0}
    \le 2 \ell_\eps^{-1} \OptSobTr_m(\eps,\eta)^2
      \normsqr[\Sob {B_\eta, g}] f.
  \end{align*}
  Moreover, from~\eqref{eq:est.harm.d} we obtain
  \begin{align*}
    \qf d_{\Cyl_\eps}(\Phi_\eps^\perp f)
    = \qf d_{\Cyl_\eps}(h^\perp)
    &\le \sum_{p\in I_\eps} \sum_{\mu \in \Mu \setminus \{0\}}
        \Bigl(\frac {\mu+1} {\eps \ell_\eps}
        + \frac2{\ell_\eps^2}\Bigr) \abssqr{a_p^\mu}\\
    &= \sum_{p\in I_\eps} \sum_{\mu \in \Mu \setminus \{0\}}
        \Bigl(\frac {1+1/\mu} {\eps \ell_\eps}
        + \frac2{\mu \ell_\eps^2}\Bigr) \mu \abssqr{a_p^\mu}\\
    &\le \sum_{p\in I_\eps}
        2\Bigl(1 + \frac \eps{\ell_\eps}\Bigr)\cdot
      \frac 1{\eps \ell_\eps} \sum_{\mu \in \Mu \setminus \{0\}} \mu \abssqr{a_p^\mu}\\
  \end{align*}
  as $\mu \ge \mu_1=\sqrt{m-1}\ge 1$ for $\mu \in \Mu \setminus\{0\}$
  in the second inequality.  In addition, we have
  \begin{align*}
    \frac 1{\eps \ell_\eps}
    \sum_{\mu \in \Mu} \mu \abssqr{a_p^\mu}
    &= \frac 1{\eps \ell_\eps}
    \normsqr[\Lsqr \Sphere]{\Delta_\Sphere^{1/4} h_p(0,\cdot)}
    =\eps^{m-2}
    \normsqr[\Lsqr \Sphere]{\Delta_\Sphere^{1/4} f_p(\eps,\cdot)}\\
    &= \normsqr[\Lsqr {\bd B_\eps(p),\iota_\eps^*g_\eucl}]
      {\Delta_{\eps \Sphere}^{1/4} f_p(\eps,\cdot)}\\
    &\le (m-1)^{1/2} \normsqr[\Lsqr{B_\eps(p),g_\eucl}] {df}\\
    &\le (m-1)^{1/2} K^{(m+2)/2}\normsqr[\Lsqr{B_\eps(p),g}] {df}
  \end{align*}
  for $p \in I_\eps^-$, using the spectral decomposition of
  $\Delta_\Sphere$ for the first,~\eqref{eq:a} for the second, the
  scaling behaviour for the third equality, \Lem{muf} for the second
  last and \Cor{eucl.metric2} for the last step.  A similar argument
  holds for $p \in I_\eps^+$.  Finally, we apply
  \Prps{non-concentr2}{0} and obtain
  \begin{equation*}
     \norm[\Lsqr{B_\eps,g}] {df}
     \le \OptNonConc_m(\eps,\eta)
     \norm[{\Sob[2]{B_\eta,g}}] f.\qedhere
  \end{equation*}
\end{proof}

\begin{corollary}
  \label{cor:harm2}
  Let $\eps, \eta=\eta_\eps \in (0,r_0)$ such that
  $0<\eps<\eta_\eps/2$ then we have
  \begin{align*}
    \qf d_{\Cyl_\eps}(\Phi_\eps f)
      \le (\deltaHarm \eps')^2
    \normsqr[{\Sob[2]{B_\eta,g}}]f
  \end{align*}
  for $f \in \Sob[2]{B_\eta}$, where
  \begin{equation*}
    \deltaHarm \eps' \coloneqq
    \max \{ \deltaHarm \eps ^\bullet, \deltaHarm \eps ^\perp\}
  \end{equation*}
\end{corollary}
\begin{proof}
  The claim follows from
  \begin{align*}
    \qf d_{\Cyl_\eps}(\Phi_\eps f)
    = \qf d_{\Cyl_\eps}(\Phi^\bullet_\eps f)
      + \qf d_{\Cyl_\eps}(\Phi^\perp_\eps f)
  \end{align*}
  and \Prp{harm2}.
\end{proof}

\begin{lemma}
  \label{lem:dech}
  We have
  \begin{equation*}
    \normsqr[\Lsqr{\Cyl,g_\can}]{h-\Phi_\eps u}
    \le \frac{\ell_\eps^2}{\pi^2} \qf d_{\Cyl_\eps}(h)
    \le \frac{\ell_\eps^2}{\pi^2} \qf d_\eps(U)
  \end{equation*}
  for all $U=(u,h) \in \dom \qf d_\eps$.
\end{lemma}
\begin{proof}
  We have
  $h_0\coloneqq h-\Phi_\eps u \in \Sobn C=\set{g \in \Sob C}{g \restr {\bd
      C}=0}$,
  i.e., $h_0$ fulfils Dirichlet conditions on $\bd C$.  From the
  min-max principle we obtain
  $\lambda_1^\Dir \normsqr[\Lsqr{\Cyl,g_\can}] {h_0} \le \qf d_{\Cyl_\eps}(h_0)$,
  where $\lambda_1^\Dir=\pi^2/\ell_\eps^2$ is the smallest Dirichlet
  eigenvalue on $\Cyl_\eps$.  Moreover, we have
  \begin{equation*}
    \qf d_{\Cyl_\eps}(h_0)
    \le \qf d_{\Cyl_\eps}(h_0)+\qf d_{\Cyl_\eps}(\Phi_\eps u)
    =\qf d_{\Cyl_\eps}(h)
  \end{equation*}
  and the result follows.
\end{proof}

\begin{lemma}
  \label{lem:dech2}
  We have
  \begin{equation*}
    \norm[\Lsqr{\Cyl,g_\can}] h
    \le \deltaHandle \eps
    \norm[\HS_\eps^1] U
  \end{equation*}
  for all $U=(u,h) \in \dom \qf d_\eps$,
  where
  \begin{equation*}
    \deltaHandle \eps
    \coloneqq
    \frac{\ell_\eps}\pi + \deltaHarm \eps
    =\frac{\ell_\eps}\pi + \bigl((\ell_\eps+\eps)\bigr)^{1/2} \OptSobTr_m(\eps,\eta).
  \end{equation*}
\end{lemma}
\begin{proof}
  We have
  \begin{align*}
    \norm[\Lsqr{\Cyl,g_\can}] h
    &\le \norm[\Lsqr{\Cyl,g_\can}] {h-\Phi_\eps u}
    + \norm[\Lsqr{\Cyl,g_\can}] {\Phi_\eps u}\\
    &\le \frac{\ell_\eps}\pi \sqrt{\qf d_{\Cyl_\eps}(h)}
      + \deltaHarm \eps \norm[\Sob{X_\eps,g}] u
     \le \deltaHandle \eps \norm[\HS_\eps^1] U
  \end{align*}
  using \Lem{dech} and \Cor{harm}.
\end{proof}

\begin{remark}[Relation between and order of the estimates]
  \label{rem:harm}
  For completeness, we also define
  \begin{equation}
    \label{eq:def.delta.ball}
    \deltaBall \eps
    := \OptNonConc_m(\eps,\eta_\eps)
    = N^{1/2} K^{(m+1)/2} \OptNonConcEucl_m(\eps,\eta_\eps)
    =\Err\Bigl(\Bigl(\frac{\eps^m}{\eta_\eps^m} + \eps^2 [-\log \eps]_2
    \Bigr)^{1/2}\Bigr),
  \end{equation}
  the non-concentrating constant for a ball of radius $\eps$ inside a
  ball of radius $\eps_\eps$ (see \Prp{0}).  We have
  \begin{align*}
    \deltaHarm \eps
    &= \Err\Bigl(\Bigl(\Bigl(\frac{\ell_\eps}\eps+1\Bigr)
      \Bigl(\frac{\eps^m}{\eta_\eps^m} + \eps^2 [-\log \eps]_2 \Bigr)
      \Bigr)^{1/2}\Bigr),\\
    \deltaHandle \eps
    &= \Err\Bigl(\ell_\eps + \Bigl(\Bigl(\frac{\ell_\eps}\eps+1\Bigr)
      \Bigl(\frac{\eps^m}{\eta_\eps^m} + \eps^2 [-\log \eps]_2 \Bigr)
      \Bigr)^{1/2}\Bigr),\\
    \deltaHarm \eps^\perp
    &= \Err\Bigl(\Bigl(\Bigl(\frac \eps {\ell_\eps} +1\Bigr)
      \Bigl(\frac{\eps^m}{\eta_\eps^m} + \eps^2 [-\log \eps]_2 \Bigr)
      \Bigr)^{1/2}\Bigr),\\
    \deltaHarm \eps^\bullet
    &= \Err\Bigl(\Bigl(\frac1 {\eps \ell_\eps}
      \Bigl(\frac{\eps^m}{\eta_\eps^m} + \eps^2 [-\log \eps]_2 \Bigr)
      \Bigr)^{1/2}\Bigr),\\
    \deltaHarm \eps'
    &= \Err\Bigl(\Bigl(\frac1 {\eps \ell_\eps}
      \Bigl(\frac{\eps^m}{\eta_\eps^m} + \eps^2 [-\log \eps]_2 \Bigr)
      \Bigr)^{1/2}\Bigr),\\
    \deltaAntisym \eps
    &= \Err\Bigl(\Bigl(\frac{\eta_\eps^m}{\eps^{m-2}}
      \Bigl(\frac {\ell_\eps}\eps
      + \Bigl[\log\frac {\eta_\eps}\eps\Bigr]_2
      \Bigr) \Bigr)^{1/2}\Bigr)
  \end{align*}
  (see \Rem{order.of.opt.const} for the order of
  $\eps \OptSobTr_m(\eps,\eta_\eps)$ and
  $\OptNonConc_m(\eps,\eta_\eps)$; for the definition of
  $\deltaAntisym \eps$, see also \Lem{ua}).

  Note that $\ell_\eps/\eps$ can be seen as a sort of ``resistance''
  along a handle, so handles with very slowly (longitudinally)
  shrinking length (compared to their radius) are somehow ``punished''
  in $\deltaHarm \eps$ and $\deltaHandle \eps$; the ratio radius per
  uniform cover distance then has to compensate such handles in the
  fading case.

  We have the following relations between the convergences:
  \begin{equation*}
    \begin{tikzcd}
      &
      & \deltaHarm \eps^\perp \to 0
      \ar[Rightarrow,rd, "\text{$\frac{\ell_\eps}\eps$ bdd}"]
      &&\\
      & \deltaHarm \eps^\bullet \to 0
      \ar[Rightarrow,ru, "\text{$\ell_\eps$ bdd}"]
      \ar[Rightarrow,rr, "\ell_\eps\to 0"]
      \ar[Rightarrow,ld,bend left, "\text{$\ell_\eps$ bdd}"]
      &
      & \deltaHarm \eps  \to 0
      \ar[Rightarrow,r, "\text{$\ell_\eps$ bdd}"]
      \ar[Rightarrow,rd, "\ell_\eps \to 0"]
      &\deltaBall \eps \to 0\\
      \deltaHarm \eps'\to 0 \ar[Rightarrow,ru]
      &
      &
      &
      & \deltaHandle \eps \to 0.
      \ar[Rightarrow,u]
      \ar[Rightarrow,lu,bend left]
    \end{tikzcd}
  \end{equation*}
  In some sense, the energy of the transversally constant part of the
  harmonic extension on the handle given by $\deltaHarm \eps^\bullet$
  gives the worst estimate (as well es the entire part of the harmonic
  extension given by $\deltaHarm \eps'$).

  If we have $\eta_\eps=r_0 \eps^\alpha$ for some $\alpha \in [0,1)$
  and $\ell_\eps=\eps^\lambda$ for $\lambda>0$, then the estimates on
  the convergence speed have order
  \begin{align*}
    \deltaBall \eps
    &= \Err\bigl(\eps^{\min\{m(1-\alpha),2\}/2}\bigr),\\
    \deltaHarm \eps
    &= \Err\bigl(\eps^{(\min\{\lambda-1,0\}+\min\{m(1-\alpha),2\})/2}\bigr),\\
    \deltaHandle \eps
    &= \Err\bigl(\eps^{(\min\{\lambda,\min\{\lambda-1,0\}+\min\{m(1-\alpha),2\}\})/2}\bigr),\\
    \deltaHarm \eps^\perp
    &= \Err\bigl(\eps^{(\min\{1-\lambda,0\}+\min\{m(1-\alpha),2\})/2}\bigr),\\
    \deltaHarm \eps^\bullet
    &= \Err\bigl(\eps^{(-\lambda-1+\min\{m(1-\alpha),2\})/2}\bigr),\\
    \deltaHarm \eps'
    &= \Err\bigl(\eps^{-\lambda-1+\min\{m(1-\alpha),2\})/2}\bigr),\\
    \deltaAntisym \eps
    &= \Err\bigl(\eps^{-\lambda-1+\min\{m(1-\alpha),2\})/2}\bigr).
  \end{align*}
  We illustrate the parameter range in the $(\alpha,\lambda)$-plane in
  \FigS{alpha-lambda-new}{anti-sym}, being slightly more precise than in the
  general convergence scheme above.
\end{remark}

\newcommand{\ColorErrorA}{fill=gray!10!white, draw=gray!10!} 
\newcommand{\DrawParRangeDeltaBall}[1]
{
  \begin{tikzpicture}[scale=1.5, inner sep=2pt,
    vertex/.style={circle,draw=black!50,%
      fill=black!50, 
      inner sep=0pt,%
      minimum size=1.5mm}%
    ]
    \foreach \m in {#1}
    {
      \filldraw[fill=gray!60!white, draw=gray!60!]
      (0,0)--(0,1.5)--(1-2/\m,1.5)--(1-2/\m,0)--(0,0);

      \filldraw[fill=gray!20!white, draw=gray!20!]
      (1,0)--(1-2/\m,0)--(1-2/\m,1.5)--(1,1.5)--(1,0);

      \draw[very thin,color=gray,step=1/\m] (0,-0.05) grid (1.0,0.05);
      \draw[->] (0,0) -- (1.1,0) node[right] {$\alpha$};
      \node at (1,-0.25) {$1$};
      \ifthenelse{\m=4} 
      {
        \node at (3/4,-0.25) {$\frac34$};
        \node at (1/2,-0.25) {$\frac12$};
      }
      {
        \ifthenelse{\m=5} 
        {
          \draw[dotted,color=gray] (1-1/\m,0.05) -- (1-1/\m,1.4);
          \node at (1-1/\m, 1.3) {\;$\frac{m-1}m$};
          \node at (1-2/\m,-0.25) {$\frac{m-2}m$};
        }
        {
          \foreach \x [evaluate =\x as \mM using int(\m-\x)]in {1,2}
          {
            \node at (1-\x/\m,-0.25) {$\frac\mM\m$};
          }
        }
      }
      \draw[very thin,color=gray,step=1/3] (-0.05,0) grid (0.05,1.2);
      \draw[very thin,color=gray,step=1] (-0.1,0) grid (0.1,1);
      \draw[->] (0,0) -- (0,1.7) node[left] {$\lambda$\;};
      \node at (-0.2,1) {$1$};
      \node at (0.5,1.9) {$\deltaBall \eps$};
      \node at (0,0) [vertex,label=left:$D$] {};
    }
  \end{tikzpicture}
}
\newcommand{\DrawParRangeDeltaHarm}[1]
{
  \begin{tikzpicture}[scale=1.5, inner sep=2pt,
    vertex/.style={circle,draw=black!50,%
      fill=black!50, 
      inner sep=0pt,%
      minimum size=1.5mm}%
    ]
    \foreach \m in {#1}
    {
      \filldraw[fill=gray!80!white, draw=gray!80!]
      (0,0)--(0,1)--(1-2/\m,1)--(1-2/\m,0)--(0,0);

      \filldraw[fill=gray!60!white, draw=gray!60!]
      (0,1)--(0,1.5)--(1-2/\m,1.5)--(1-2/\m,1)--(0,1);

      \filldraw[fill=gray!40!white, draw=gray!40!]
      (1,1)--(1-1/\m,0)--(1-2/\m,0)--(1-2/\m,1)--(1,1);

      \filldraw[fill=gray!20!white, draw=gray!20!]
      (1,1)--(1-2/\m,1)--(1-2/\m,1.5)--(1,1.5)--(1,1);

      \draw[very thin,color=gray,step=1/\m] (0,-0.05) grid (1.0,0.05);
      \draw[->] (0,0) -- (1.1,0) node[right] {$\alpha$};
      \node at (1,-0.25) {$1$};
      \ifthenelse{\m=4} 
      {
        \node at (3/4,-0.25) {$\frac34$};
        \node at (1/2,-0.25) {$\frac12$};
      }
      {
        \ifthenelse{\m=5} 
        {
          \draw[dotted,color=gray] (1-1/\m,0.05) -- (1-1/\m,1.1);
          \node at (1-1/\m, 1.3) {\;$\frac{m-1}m$};
          \node at (1-2/\m,-0.25) {$\frac{m-2}m$};
        }
        {
          \foreach \x [evaluate =\x as \mM using int(\m-\x)]in {1,2}
          {
            \node at (1-\x/\m,-0.25) {$\frac\mM\m$};
          }
        }
      }
      \draw[very thin,color=gray,step=1/3] (-0.05,0) grid (0.05,1.2);
      \draw[very thin,color=gray,step=1] (-0.1,0) grid (0.1,1);
      \draw[->] (0,0) -- (0,1.7) node[left] {$\lambda$\;};
            \node at (-0.2,1) {$1$};
      \node at (0.5,1.9) {$\deltaHarm \eps$};
      \node at (1-2/\m,1) [vertex,label=left:$E$] {};
      \node at (0,0) [vertex,label=left:$D$] {};
      \node at (1-1/\m,0) [vertex,label=above:$C$] {};
      \node at (1,1) [vertex,label=right:$B$] {};
    }
  \end{tikzpicture}
}
\newcommand{\DrawParRangeDeltaHandle}[1]
{
  \begin{tikzpicture}[scale=1.5, inner sep=2pt,
    vertex/.style={circle,draw=black!50,%
      fill=black!50, 
      inner sep=0pt,%
      minimum size=1.5mm}%
    ]
    \foreach \m in {#1}
    {
      \filldraw[fill=gray!80!white, draw=gray!80!]
      (0,0)--(0,1)--(1-2/\m,1)--(1-1/\m,0)--(0,0);

      \filldraw[fill=gray!60!white, draw=gray!60!]
      (0,1)--(0,1.5)--(1-2/\m,1.5)--(1-2/\m,1)--(0,1);

      \filldraw[fill=gray!40!white, draw=gray!40!]
      (1,1)--(1-1/\m,0)--(1-2/\m,1)--(1,1);

      \filldraw[fill=gray!20!white, draw=gray!20!]
      (1,1)--(1-2/\m,1)--(1-2/\m,1.5)--(1,1.5)--(1,1);

      \draw[very thin,color=gray,step=1/\m] (0,-0.05) grid (1.0,0.05);
      \draw[->] (0,0) -- (1.1,0) node[right] {$\alpha$};
      \node at (1,-0.25) {$1$};
      \ifthenelse{\m=4} 
      {
        \node at (3/4,-0.25) {$\frac34$};
        \node at (1/2,-0.25) {$\frac12$};
      }
      {
        \ifthenelse{\m=5} 
        {
          \draw[dotted,color=gray] (1-1/\m,0.05) -- (1-1/\m,1.1);
          \node at (1-1/\m, 1.3) {\;$\frac{m-1}m$};
          \node at (1-2/\m,-0.25) {$\frac{m-2}m$};
        }
        {
          \foreach \x [evaluate =\x as \mM using int(\m-\x)]in {1,2}
          {
            \node at (1-\x/\m,-0.25) {$\frac\mM\m$};
          }
        }
      }
      \draw[very thin,color=gray,step=1/3] (-0.05,0) grid (0.05,1.2);
      \draw[very thin,color=gray,step=1] (-0.1,0) grid (0.1,1);
      \draw[->] (0,0) -- (0,1.7) node[left] {$\lambda$\;};
      \node at (0.5,1.9) {$\deltaHandle \eps$};
      \node at (1-2/\m,1) [vertex,label=left:$E$] {};
      \node at (0,0) [vertex,label=left:$D$] {};
      \node at (0,1) [vertex,label=left:$D'$] {};
      \node at (1-1/\m,0) [vertex,label=above:$C$] {};
      \node at (1,1) [vertex,label=right:$B$] {};
    }
  \end{tikzpicture}
}

\newcommand{\DrawParRangeDeltaHarmOrth}[1]
{
  \begin{tikzpicture}[scale=1.5, inner sep=2pt,
    vertex/.style={circle,draw=black!50,%
      fill=black!50, 
      inner sep=0pt,%
      minimum size=1.5mm}%
    ]
    \foreach \m in {#1}
    {
      \filldraw[fill=gray!80!white, draw=gray!80!]
      (0,1)--(0,3)--(1-2/\m,3)--(1-2/\m,1)--(0,1);

      \filldraw[fill=gray!60!white, draw=gray!60!]
      (0,0)--(0,1)--(1-2/\m,1)--(1-2/\m,0)--(0,0);

      \filldraw[fill=gray!40!white, draw=gray!40!]
      (1,1)--(1-2/\m,0)--(1-2/\m,3)--(1,1);

      \filldraw[fill=gray!20!white, draw=gray!20!]
      (1,1)--(1,0)--(1-2/\m,0)--(1-2/\m,1)--(1,1);

      \draw[very thin,color=gray,step=1/\m] (0,-0.05) grid (1.0,0.05);
      \draw[->] (0,0) -- (1.1,0) node[right] {$\alpha$};
      \node at (1,-0.25) {$1$};
      \ifthenelse{\m=2}
      {
        \node at (0,1) [vertex,label=left:${D'=E}$] {};
      }
      {
        \node at (0,1) [vertex,label=left:$D'$] {};
        \node at (1-2/\m,1) [vertex,label=left:$E$] {};
        \ifthenelse{\m=4} 
        {
          \node at (3/4,-0.25) {$\frac34$};
          \node at (1/2,-0.25) {$\frac12$};
        }
        {
          \ifthenelse{\m=5} 
          {
            \draw[dotted,color=gray] (1-1/\m,0.05) -- (1-1/\m,1.1);
            \node at (1-1/\m, 1.3) {\;$\frac{m-1}m$};
            \node at (1-2/\m,-0.25) {$\frac{m-2}m$};
          }
          {
            \foreach \x [evaluate =\x as \mM using int(\m-\x)]in {1,2}
            {
              \node at (1-\x/\m,-0.25) {$\frac\mM\m$};
            }
          }
        }
      }
      \draw[very thin,color=gray,step=1/3] (-0.05,0) grid (0.05,3.2);
      \draw[very thin,color=gray,step=1] (-0.1,0) grid (0.1,3);
      \draw[->] (0,0) -- (0,3.3) node[left] {$\lambda$\;};
      \node at (0.5,3.3) {$\deltaHarm \eps^\perp$};
      \node at (1-2/\m,3) [vertex,label=right:$F'$] {};
      \node at (0,0) [vertex,label=left:$D$] {};
      \node at (0,3) [vertex,label=left:$G'$] {};
      \node at (1-1/\m,0) [vertex,label=above:$C$] {};
      \node at (1,1) [vertex,label=right:$B$] {};
    }
  \end{tikzpicture}
}

\newcommand{\DrawParRangeDeltaHarmBullet}[1]
{
  \begin{tikzpicture}[scale=1.5, inner sep=2pt,
    vertex/.style={circle,draw=black!50,%
      fill=black!50, 
      inner sep=0pt,%
      minimum size=1.5mm}%
    ]
    \foreach \m in {#1}
    {
      \filldraw[fill=gray!80!white, draw=gray!80!]
      (0,0)--(0,1)--(1-2/\m,1)--(1-2/\m,0)--(0,0);

      \filldraw[fill=gray!40!white, draw=gray!40!]
      (1-2/\m,1)--(1-1/\m,0)--(1-2/\m,0)--(1-2/\m,1);

      \draw[very thin,color=gray,step=1/\m] (0,-0.05) grid (1.0,0.05);
      \draw[->] (0,0) -- (1.1,0) node[right] {$\alpha$};
      \node at (1,-0.25) {$1$};
      \ifthenelse{\m=4} 
      {
        \node at (3/4,-0.25) {$\frac34$};
        \node at (1/2,-0.25) {$\frac12$};
      }
      {
        \ifthenelse{\m=5} 
        {
          \node at (1-2/\m,-0.25) {$\frac{m-2}m$};
        }
        {
          \foreach \x [evaluate =\x as \mM using int(\m-\x)]in {1,2}
          {
            \node at (1-\x/\m,-0.25) {$\frac\mM\m$};
          }
        }
      }
      \draw[very thin,color=gray,step=1/3] (-0.05,0) grid (0.05,1.2);
      \draw[very thin,color=gray,step=1] (-0.1,0) grid (0.1,1);
      \draw[->] (0,0) -- (0,1.7) node[left] {$\lambda$\;};
      \node at (-0.2,1) {$1$};
      \node at (0.5,1.9) {$\deltaHarm \eps^\bullet$};
      \node at (0.5,2.3) {$\deltaHarm \eps'$};
      \node at (1-2/\m,1) [vertex,label=left:$E$] {};
      \node at (0,0) [vertex,label=left:$D$] {};
      \node at (1-1/\m,0) [vertex,label=above:$C$] {};
    }
  \end{tikzpicture}
}

\newcommand{\DrawParRangeDeltaHarmOther}
{
  \begin{tikzpicture}[scale=1.5, inner sep=2pt,
    vertex/.style={circle,draw=black!50,%
      fill=black!50, 
      inner sep=0pt,%
      minimum size=1.5mm}%
    ]
    \filldraw[fill=gray!40!white, draw=gray!40!]
    (0,1) -- (1/2,0) -- (0,0) -- (0,1);

    \node at (0.5,2) {$\deltaHarm \eps^\bullet$};
    \node at (0.5,1.5) {$\deltaHarm \eps'$};
    \node at (0,0) [vertex,label=left:$D$] {};

      \draw[very thin,color=gray,step=1/\m] (0,-0.05) grid (1.0,0.05);
      \draw[->] (0,0) -- (1.1,0) node[right] {$\alpha$};
      \node at (1,-0.25) {$1$};
      \ifthenelse{\m=4} 
      {
        \node at (3/4,-0.25) {$\frac34$};
        \node at (1/2,-0.25) {$\frac12$};
      }
      {
        \foreach \x [evaluate =\x as \mM using int(\m-\x)]in {1,2}
        {
          \node at (1-\x/\m,-0.25) {$\frac\mM\m$};
        }
      }
      \draw[very thin,color=gray,step=1/3] (-0.05,0) grid (0.05,1.2);
      \draw[very thin,color=gray,step=1] (-0.1,0) grid (0.1,1);
      \draw[->] (0,0) -- (0,1.3) node[above] {$\lambda$};
      \node at (-0.25,1/3) {$1/3$};
      \node at (-0.2,1) {$1$};
  \end{tikzpicture}
}

\begin{figure}[h]
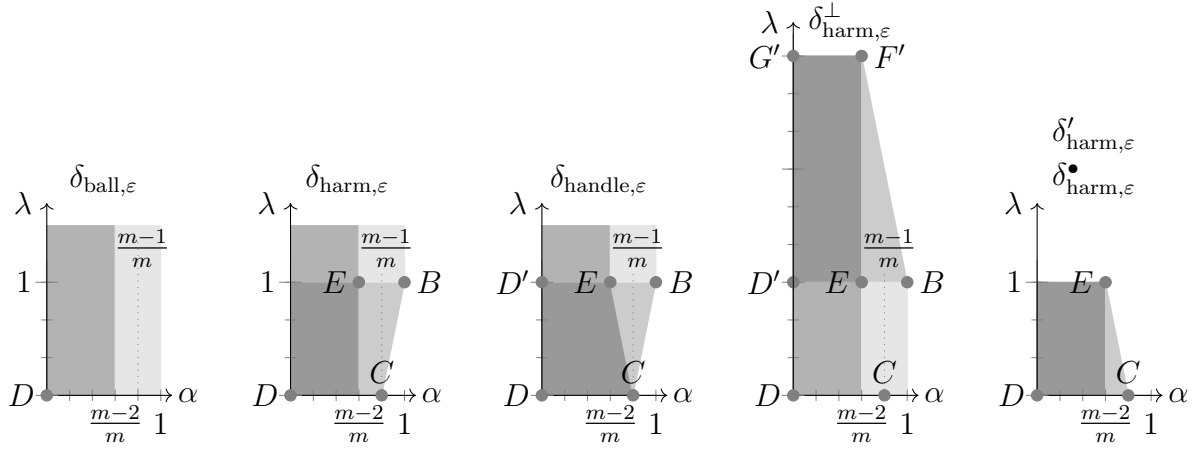

  \centering
  \DrawParRangeDeltaBall{5}\hfill
  \DrawParRangeDeltaHarm{5}\hfill
  \DrawParRangeDeltaHandle{5}\hfill
  \DrawParRangeDeltaHarmOrth{5}\hfill
  \DrawParRangeDeltaHarmBullet{5}
  \caption{The range of the parameters $\alpha$ and $\lambda$ for the
    various constants.  Darkest grey: \emph{Second left:}
    $\Err(\eps^{(\lambda+1)/2})$, \emph{middle:} $\Err(\eps^\lambda)$,
    \emph{second right:} $\Err(\eps^{(3-\lambda)/2})$,
    \emph{right:} $\Err(\eps^{(1-\lambda)/2})$. %
    Second darkest grey: \emph{Left, second left, middle and second
      right:} $\Err(\eps)$. %
    Second lightest grey: \emph{Second left and middle:}
    $\Err(\eps^{(\lambda+(m-1)-\alpha m)/2})$, \emph{second right:}
    $\Err(\eps^{1-\lambda+(1-\alpha) m)/2})$, \emph{right:}
    $\Err(\eps^{(-\lambda+(m-1)-\alpha m)/2})$. %
    Lightest grey: \emph{Left, second left, middle and second right:}
    $\Err(\eps^{m(1-\alpha)/2})$.}
  \label{fig:alpha-lambda-new}
\end{figure}

\newcommand{\DrawParRangeAntisym}[1]%
{
  \begin{tikzpicture}[scale=1.5, inner sep=2pt,
    vertex/.style={circle,draw=black!50,%
      fill=black!50, 
      inner sep=0pt,
      minimum size=1.5mm}
    ]
    \newcommand{\Myinfty}{2.0}
    \newcommand{\alphaScale}{1}
    \foreach \m in {#1}
    {
      \ifthenelse{\m=3}
      {\newcommand{\alphaM}{1/3}}
      {\ifthenelse{\m=4}
        {\newcommand{\alphaM}{1/2}}
        {\newcommand{\alphaM}{1/2}}
      }
      \coordinate(A) at (\alphaScale,\Myinfty);
      \coordinate(B) at (\alphaScale,1);
      \coordinate(C) at ({\alphaScale*(1-1/\m)},0);
      \coordinate(Cp) at ({\alphaScale*\alphaM},0);
      \coordinate(D) at (0,0); 
      \coordinate(Dm) at (0,{\alphaM/2});
      \coordinate(Dp) at (0,1);
      \coordinate(Dpl) at (0,{1-\alphaM});
      \coordinate(E) at ({\alphaScale*(1-2/\m)},1);
      \coordinate(Ep) at ({\alphaScale*\alphaM},1);
      \coordinate(F) at ({\alphaScale*(1-2/\m)},{\Myinfty});
      \coordinate(H) at ({\alphaScale*\alphaM*\alphaM/(2-\alphaM)},%
      {\alphaM/(2-\alphaM)});
      \coordinate(Hp) at ({\alphaScale*\alphaM*\alphaM/(2-\alphaM)},%
      {(2-3*\alphaM)/(2-\alphaM)});

      \filldraw[fill=gray!20!white, draw=gray!20!]
      (A) --(B) -- (E) -- (F) -- (A);

      \filldraw[fill=gray!40!white, draw=gray!40!]
      (C)--(E)--(B)--(\alphaScale,0)--(C);

      \node at (0.7,\Myinfty+0.7) {$\deltaAntisym \eps$};
      \ifthenelse{\m=2}
      {
        \node at (F) [vertex,label=right:$F$] {};
        \node at ({\alphaScale*(1-2/\m)},-0.25) {$0$};
        \node at (0.7,\Myinfty+0.4) {$m=\m$};
      }
      {
        \ifthenelse{\m=3}
        {
          \node at (F) [vertex,label=right:$F$] {};
          \node at ({\alphaScale*(1-2/\m)},-0.25) {$\frac13$};
          \node at (0.7,\Myinfty+0.4) {$m=\m$};
        }
        {
          \ifthenelse{\m=4}
          {
            \node at ({\alphaScale*(1-2/\m)},-0.25) {$\frac12$};
            \node at (0.7,\Myinfty+0.4) {$m=\m$};
          }
          {
            \node at (F) [vertex,label=left:$F$] {};
            \node at ({\alphaScale*(1-2/\m)},-0.25) {$\frac{m-2}m$};
            \node at (0.7,\Myinfty+0.4) {$m \ge \m$};
          }
        }
        }
      \draw[dotted,color=gray] ({\alphaScale*(1-1/\m)},0.05)
      -- ({\alphaScale*(1-1/\m)},\Myinfty-0.5);
      \node at (1-1/\m, 1.5) {\;$\frac{m-1}m$};
      \draw[dotted,color=gray] ({\alphaScale*(1-2/\m},0.05)
      -- ({\alphaScale*(1-2/\m)},1);

      \node at (A) [vertex,label=right:$A$] {};
      \node at (B) [vertex,label=right:$B$] {};
      \node at (C) [vertex,label=above:$C$] {};
      \node at (D) [vertex,label=left:$D$] {};
      \node at (E) [vertex,label=left:$E$] {};

      \draw[very thin,color=gray,step=\alphaScale/\m] (0,-0.05) grid (\alphaScale*1.0,0.05);
      \draw[->] (0,0) -- (\alphaScale*1.2,0) node[right] {$\alpha$};
      \draw[very thin,color=gray,step=1/3] (-0.05,0) grid (0.05,1.2);
      \draw[very thin,color=gray,step=1] (-0.1,0) grid (0.1,\Myinfty);
      \draw[->] (0,0) -- (0,\Myinfty+0.2) node[left] {$\lambda$};
    }
  \end{tikzpicture}
}
\begin{figure}[h]
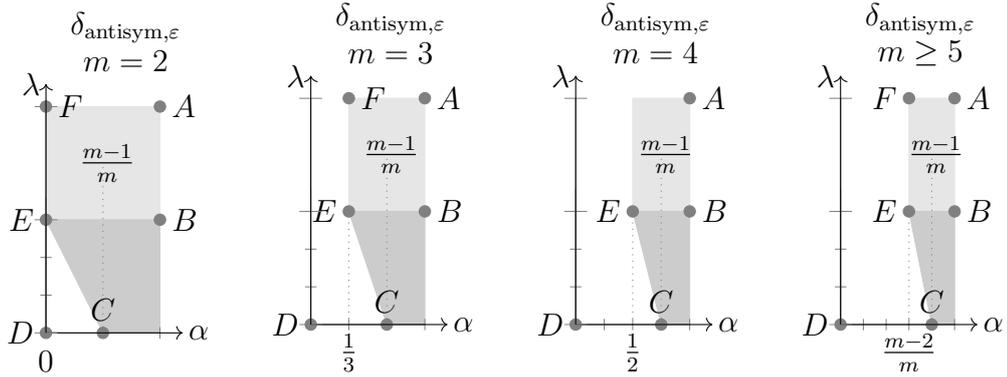

  \centering
  \DrawParRangeAntisym{2}\quad
  \DrawParRangeAntisym{3}\quad
  \DrawParRangeAntisym{4}\quad
  \DrawParRangeAntisym{5}
  \caption{The range of the parameters $\alpha$ and $\lambda$ where
    $\deltaAntisym \eps$ (appearing in \Lem{ua}) converges.  The light
    grey (infinite) rectangle $ABEF$ is of order
    $\Err(\eps^{(m\alpha-(m-2))/2})$; the dark grey polygon $C(1,0)BE$
    is of order $\Err(\eps^{(\lambda + m\alpha-(m-1))/2})$.}
  \label{fig:anti-sym}
\end{figure}

%
\section{Proofs of the main theorems}
\label{sec:proofs}
%

\subsection{Proof of \Thm{handles0} (fading handles I)}
\label{ssec:proof.main0}
The goal is to define transplantation operators which satisfy the
hypotheses of \Def{quasi-uni}.

The spaces are
\begin{align*}
  \HS&\coloneqq\HS_0=\Lsqr{X,g},
  &\wt \HS&\coloneqq\HS_\eps=\Lsqr{X_\eps,g}\oplus \Lsqr{\Cyl,g_\can}
  \cong\Lsqr{M_\eps,g_\eps},\\
  \HS^1&\coloneqq\HS_0^1=\Sob{X,g},
  &\wt \HS^1&\coloneqq\HS^1_\eps=\dom {\qf d}_\eps.
\end{align*}
An element $U\in\HS_\eps$ is written $U=(u,h)$ with
$u=U\restr {X_\eps}$ and $h \in \Lsqr {\Cyl,g_\can}$.

In this first approach use some the calculations already done
in~\cite{anne-post:21} for Dirichlet obstacles.

We define the transplantation operators as follows
\begin{align*}
    \map {J_\eps} {\HS_0} {&\HS_\eps},
    & f & \mapsto (f \restr {X_\eps},0),\\
    \map {J_\eps^1} {\HS_0^1} {&\HS^1_\eps},
    & f & \mapsto (\chi_\eps f,0),\\
   \map{J_\eps'}{\HS_\eps}{&\HS_0},
     & U=(u,h) & \mapsto \bar u,\\
   \map{J_\eps^{\prime1}}{\HS^1_\eps}{&\HS_0^1},
     &  U=(u,h) & \mapsto E_\eps u,
\end{align*}
where $\chi_\eps $ is the cut-off function already defined
in~\cite[Section 5.2]{anne-post:21}, where $\bar u$ is the extension
of $u \in \Lsqr{X_\eps,g}$ by $0$ onto $\Lsqr{B_\eps,g}$, and where
\begin{equation}
  \label{eq:ext.op}
  \map{E_\eps}{\Sob{X_\eps,g}}{\Sob{X,g}}
\end{equation}
is the harmonic extension operator with respect to the Euclidean
metric $g_{\eucl,p}$ near each ball $B_\eps(p)$
(cf.~\eqref{eq:eucl.met}), i.e., $E_\eps u$ is harmonic on $B_\eps$
(for $g_{\eucl.p}$).  We have shown in~\cite[Prop. 4.6]{anne-post:21}
that there is a constant $\Cext \ge 1$ such that $E_\eps$ as operator
in~\eqref{eq:ext.op} is bounded by $\Cext$ for all $\eps \in (0,r_0)$.

Moreover, $\map{\chi_\eps}X{[0,1]}$ is in $\Sob {X,g}$, has support in
$X_\eps$, it is $0$ on $B_\eps$ and $1$ outside
$B_{\eps^+}=\bigdcup_{p \in I_\eps^+} B_{\eps^+}(p)$, and on
$B_{\eps^+} \setminus B_\eps$, it is the harmonic extension, where
$0<\eps < \eps^+ < \eta_\eps/4$.  In particular, the support of
$\1_X-\chi_\eps$ is $B_{\eps^+}$.  This cut-off function has been
used in~\cite[Thm.~5.6]{anne-post:21}, where we have set
$\eps^+\coloneqq (\zeta_\eps \eta_\eps)^{1/2}$.  In the notation
of~\cite[Thm.~5.6]{anne-post:21} and by our assumption, we have
$\omega_\eps \coloneqq \zeta_\eps/\eta_\eps \to 0$ as $\eps \to 0$.
The complicated formula for $\zeta_\eps$ comes from an application of
a Sobolev embedding theorem, cf.~\cite[Lem.~5.4 and
Prp.~5.5]{anne-post:21}.

Let us now check the conditions of \Def{quasi-uni}:
\begin{proof}[Proof of \Thm{handles0}]
  First, we have $J_\eps'=J_\eps^*$ and
  $\norm{J_\eps}=\norm{J_\eps'}=1$.  Moreover,
  $J_\eps'J_\eps f=\1_{X_\eps} f$ and hence
  $J_\eps J_\eps' J_\eps f = (f \restr {X_\eps},0)=J_\eps f$, hence
  $J_\eps$ is a partial isometry.  Second,
  \begin{equation*}
    \norm[\HS_0] {f-J_\eps'J_\eps f}
    =\norm[\Lsqr{B_\eps,g}] f
    \le C_m(\eps,\eta_\eps) \norm[\Sob{X,g}] f
  \end{equation*}
  for $f \in \Sob{X,g}$ by \Prp{0}.  In addition,
  \begin{align*}
    \norm[\HS_\eps] {U-J_\eps J_\eps'U}
    =\norm[\Lsqr{\Cyl,g_\can}] h
    \le \deltaHandle \eps \norm[\HS_\eps^1] U
  \end{align*}
  for $U \in \HS_\eps^1$ by \Lem{dech2}.  Next, we have
  \begin{align*}
    \norm[\HS_\eps]{(J_\eps-J_\eps^1)f}
    =\norm[\Lsqr{X_\eps,g}]{(1-\chi_\eps)f}
    \le \norm[\Lsqr{B_{\eps^+},g}] f
    &\le C_m(\eps^+,\eta_\eps) \norm[\Sob{X,g}] f
  \end{align*}
  for $f \in \Sob{X,g}$ again by \Prp{0}.  Note that $\eps^+ > \eps$
  and that $C_m(\cdot,\eta)$ is monotonely increasing,
  hence $C_m(\eps^+,\eta_\eps)$ dominates $C_m(\eps,\eta_\eps)$.
  Moreover, we have
  \begin{align*}
    \norm[\HS_0]{(J_\eps'-J_\eps^{\prime1})U}
    = \norm[\Lsqr{B_\eps,g}]{E_\eps u}
    &\le C_m(\eps,\eta_\eps) \norm[\Sob{X,g}] {E_\eps u}\\
    &\le \Cext C_m(\eps,\eta_\eps) \norm[\HS_\eps^1] U
  \end{align*}
  for $U=(u,h) \in \HS^1_\eps$ again by \Prp{0} and the
  $\eps$-independent norm bound on the extension operator
  in~\eqref{eq:ext.op}.

  Finally, we have
  \begin{align}
    \nonumber
    \qf d_\eps(J_\eps^1f,U)-\qf d_0(f,J_\eps^{\prime1} U)
    &=\int_{X_\eps} \iprod{d(\chi_\eps f)}{du} \dd g
      -\int_X \iprod{df}{d (E_\eps u)} \dd g\\
    &=\int_{X_\eps} \bigiprod{d\bigl((\chi_\eps-1) f\bigr)}{du} \dd g
      -\int_{B_\eps}\iprod{df}{d (E_\eps u)} \dd g.
      \label{eq:proof.est1}
  \end{align}
  For the second term in~\eqref{eq:proof.est1}, the non-concentrating
  property \Prp{0} gives via \Prp{non-concentr2} the estimate
  \begin{equation*}
    \norm[\Lsqr{B_\eps,g}] {df}
    \le C_m(\eps,\eta_\eps) \norm[{\Sob[2]{X,g}}] f
    \le \Cellreg C_m(\eps,\eta_\eps)
    \norm[\Lsqr{X,g}]{(\laplacian{(X,g)}+1)f},
  \end{equation*}
  using also \Prp{ell.reg}.  Moreover, we have
  \begin{equation*}
    \norm[\Lsqr{B_\eps,g}]{d(E_\eps u)}
    \le \norm[\Lsqr{X_\eps,g}]{d(E_\eps u)}
    \le \Cext \norm[\HS^1_\eps] U
  \end{equation*}
  For the first term in~\eqref{eq:proof.est1}, we do the same calculus
  as in~\cite[Sec.~5.2 and Thm.~5.6]{anne-post:21} and obtain (with the same notation)
  \begin{align*}
    \norm[\Lsqr{B_{\eps^+},g}] {d((\chi_\eps-1)f)}
    &\le \norm[\Lsqr{B_{\eps^+},g}] {(\chi_\eps-1)df)}
    + \norm[\Lsqr{B_{\eps^+},g}] {f d\chi_\eps}\\
    &\le (\OptNonConc_m(\eps^+,\eta_\eps) + \delta_\eps^+)\norm[{\Sob[2] {X,g}}] f\\
    &\le \Cellreg (\OptNonConc_m(\eps^+,\eta_\eps) + \delta_\eps^+)
      \norm[\Lsqr{X,g}]{(\laplacian{(X,g)}+1)f},
  \end{align*}
  where $\delta_\eps^+=\Err(\omega_\eps^{\gamma_m})$ with
  $\gamma_m=1/(2(1-\alpha_m))$ ($m \ge 3$), i.e., $\gamma_m=1$ for
  $m \ge 5$, $\gamma_m\in (1/2,1)$ for $m=4$, $\gamma_m=3/4$ for
  $m=3$; and we set $\gamma_m=1/2$ for $m=2$.  Note that
  $\OptNonConc_m(\eps^+,\eta_\eps)
  =\Err(\omega_\eps^{m/4}+\omega_\eps^{1/2}\eta_\eps[-\log
  (\omega_\eps^{1/2} \eta_\eps)]_2^{1/2})$, and as $\gamma_m\le m/4$,
  the dominant power of $\omega_\eps$ is $\omega_\eps^{\gamma_\eps}$.
  As
  $\deltaHarm \eps = \Err(\eps^{(m-1-m \alpha_m)/2}
  \omega_\eps^{m/2})$ for $m \ge 3$
  (see~\eqref{eq:handles0.deltaharm}, note that
  $\eps/\omega_\eps=\eps^{1-\alpha_m}\eta_\eps \to 0$, so the first
  term in~\eqref{eq:handles0.deltaharm} is dominant) and
  $\deltaHarm \eps = \Err(\eps^{1/2}\abs{\log \eps}^{1/2}
  \omega_\eps)$ for $m=2$, we conclude that $\deltaHarm \eps$ is not dominant
  against $\delta_\eps^+$ and $\OptNonConc_m(\eps^+,\eta_\eps)$; only
  $\Err(\ell_\eps)$ remains from $\deltaHandle \eps$ (see also
  \Rem{harm}).  In particular, the leading term is given by
  \begin{equation}
    \label{eq:handles0.error}
    \delta_\eps
    = \Err \bigl(
    \ell_\eps
    + \omega_\eps^{\gamma_m}
    + \omega_\eps^{1/2}\eta_\eps [-\log (\omega_\eps^{1/2} \eta_\eps)]_2^{1/2}
    \bigr).
  \end{equation}
  and we have proven \Thm{handles0}.
\end{proof}
\begin{proof}[Proof of \Cor{handles0}]
  It remains to show the estimates on the convergence speed.  We have
  $\omega_\eps=\Err(\eps^{\alpha_m-\alpha})$ and
  $\eps^+=\omega_\eps^{1/2}\eta_\eps=\Err(\eps^{(\alpha_m+\alpha)/2)})$,
  so that
  \begin{align*}
    \delta_\eps
    = \Err \bigl(
    \eps^{\min\{
    \lambda,
    (\alpha_m-\alpha)\gamma_m,
    (\alpha_m+\alpha)/2
    \}}\bigr)
  \end{align*}
  for $m \ge 3$ and $\omega_\eps=\abs{\log \eps}^{-(1/2-\alpha)}$ for
  $m=2$. Note that $\ell_\eps=\eps^\lambda$ is not dominant here.
\end{proof}
\subsection{Proof of \Thm{handles1} (fading handles II)}
\label{ssec:proof.main1}

We give here an alternative proof for fading handles, i.e., we avoid
the complicated control result~\cite[Lemma 5.4]{anne-post:21} related
to the use of a cut-off function $\chi_\eps$.  Instead, we use the
harmonic extension $\Phi_\eps f$ onto the handles defined in
\Def{harmprol}.

The transplantation operators are the same as in \Subsec{proof.main0},
except $J_\eps^1$ which here is given by
\begin{align*}
    \map {J_\eps^1} {\HS^1_0} {&\HS^1_\eps},
    & f & \mapsto (f \restr {X_\eps},\Phi_\eps f)
\end{align*}
where the harmonic extension $\Phi_\eps f$ onto the handles is defined
in \Def{harmprol}.

\begin{proof}[Proof of \Thm{handles1}]
  Again, $J_\eps$ is a partial isometry, and the estimates are the
  same as in the proof of \Thm{handles0}, except for the following: We
  have
  \begin{align*}
    \norm[\HS_\eps]{(J_\eps-J_\eps^1)f}
    =\norm[\Lsqr{C,g_\can}] {\Phi_\eps f}
    &\le \deltaHarm \eps\norm[\Sob{X,g}] f
  \end{align*}
  for $f \in \Sob{X,g}$ by \Cor{harm}.  Moreover,
  \begin{equation*}
    \qf d_\eps(J_\eps^1f,U)-\qf d_0(f,J_\eps^{\prime1} U)
    =\qf d_{\Cyl_\eps}(\Phi_\eps f, h)
      -\int_{B_\eps} \iprod{df}{d (E_\eps u)} \dd g
  \end{equation*}
  for $f \in \HS_0^1=\Sob {X,g}$ and
  $U=(u,h) \in \HS_\eps^1=\dom \qf d_\eps$.  The second term can be
  controlled as before with contribution
  $\Cext \Cellreg \OptNonConc_m(\eps,\eta_\eps)$ to the final
  convergence speed.  For the first term, note that
  \begin{equation*}
    \abs{\qf d_{\Cyl_\eps}(\Phi_\eps f, h)}
    \le \sqrt {\qf d_{\Cyl_\eps}(\Phi_\eps f)} \sqrt {\qf d_{\Cyl_\eps}(h)}
    \le \deltaHarm \eps' \norm[{\Sob[2]{B_\eta \setminus B_\eps,g}}]f
        \norm[\HS_\eps^1] U
  \end{equation*}
  for $f \in \Sob[2]{X,g}$ by \Cor{harm2}.

  In particular, the entire estimate on the convergence speed is given
  by
  \begin{equation}
    \label{eq:err.handles1}
    \delta_\eps \coloneqq
    \Cellreg \max \Bigl\{
    \deltaHandle \eps,  
    \deltaHarm \eps' + \Cext C_m(\eps,\eta_\eps)
    \Bigr\}
  \end{equation}
  using also \Prp{ell.reg} and
  $\deltaHandle \eps \le \ell_\eps/\pi +\deltaHarm \eps'$ (see
  \Rem{harm}).  The leading term in~\eqref{eq:err.handles1} is the
  second one, together with $\Err(\ell_\eps)$, see \Rem{harm}.
\end{proof}

\section{Adhering handles: some more tools and proof}
\label{sec:gluing.two.parts}

\subsection{A symmetrisation operator and related estimates}
\label{ssec:more.tools}

We prove here the possibility of making a bridge (``wormhole'')
between two isometric parts $\Omega^\pm$ of $X$ by short
\emph{non-fading handles} (see also \Fig{identified-mfd}).

Let $\map{\chi_\eps} X {[0,1]}$ be the cut-off function defined
already in~\cite[Eq.~(6.6)]{anne-post:21} with
\begin{equation*}
  \chi_\eps(x)
  =\wt\chi \Bigl(\frac{d(x,\Omega)}{\wt \eps}\Bigr).
\end{equation*}
Here, $\map {\wt \chi}\R{[0,1]}$ is a smooth function with
$\norm[\infty]{\wt \chi'}\le 2$, $\wt \chi(s)=0$ for $s \le 0$ and
$\wt \chi(s)=1$ for $s\ge 1$ and $\wt \chi(s) \in (0,1)$ for
$s \in (0,1$).  As a consequence,
\begin{equation*}
  \supp(1-\chi_\eps)=\clo \Omega_{\wt \eps}.
\end{equation*}
We also need the so-called \emph{symmetrisation} of a function
$f \in \Lsqr{X,g}$: We define
\begin{equation*}
  \map {S} {\Lsqr{X,g}} {\Lsqr{X,g}},
  \qquad
  (S f)(x):=
  \begin{cases}
    f(x), & x \in X \setminus \Omega_{r_1},\\
    \frac12(f(x)+f(\phi(x)), & x \in \Omega_{r_1}^-,\\
    \frac12(f(\phi^{-1}(x))+f(x), & x \in \Omega_{r_1}^+.
  \end{cases}
\end{equation*}
We call $\map f X \C$ \emph{($S$-)symmetric} if $S f=f$ and
\emph{($S$-)anti-symmetric} if $S f=0$.  Note that $\Gamma S f=0$ if
$f \in \Lsqr{X,g}$ or $f \in \Sob{X,g}$ (cf.~\eqref{eq:def.gamma}).
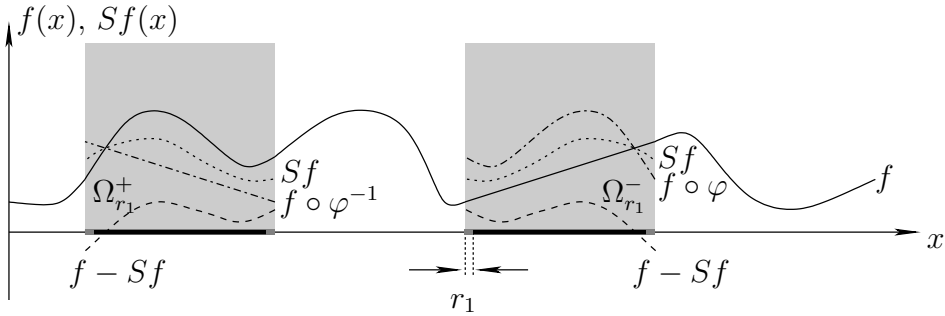
\begin{figure}[h]
  \centering
  \begin{picture}(0,0)%
    \includegraphics{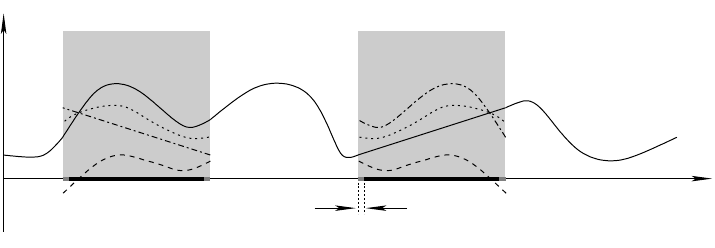}%
  \end{picture}%
  \setlength{\unitlength}{4144sp}%
  \begin{picture}(5439,1881)(424,-1039) \put(2057,-214){$Sf$}%
    \put(5887,-607){$x$}%
    \put(5581,-241){$f$}%
    \put(2053,-389){$f \circ \phi^{-1}$}%
    \put(804,-806){$f-Sf$}%
    \put(3063,-961){$r_1$}%
    \put(4141,-806){$f-S f$}%
    \put(496,699){$f(x)$, $Sf(x)$}%
    \put(946,-331){$\Omega^+_{r_1}$}%
    \put(4294,-127){$S f$}%
    \put(4294,-292){$f \circ \phi$}%
    \put(3961,-331){$\Omega^-_{r_1}$}%
  \end{picture}%
  \caption{The symmetrisation operator $S$ changing $f$ on
    $\Omega_{r_1}=\Omega^+_{r_1} \dcup \Omega^-_{r_1}$ (fat grey and
    black line, in grey: $\Omega=\Omega^+\dcup \Omega^-$); shaded in
    grey in the diagram.  Here, the diagram is plotted for $x \in X$
    along the dotted line of \Fig{identified-mfd}.  \emph{full line:}
    the original function $f$.  \emph{Dotted:} The symmetrised
    function $S f$ (only changed on $\Omega_{r_1}$).  \emph{Dashed:}
    The anti-symmetric part $f-S f$ of $f$.  \emph{Dotted-dashed:} The
    function $f$ copied from the other isometric set.}
  \label{fig:symmetrisation}
\end{figure}
We note the following simple observation (as $\phi$ is an isometry from
$\Omega^-_{r_1}$ onto $\Omega^+_{r_1}$ and
$B_\eps^\pm=\bigcup_{p \in I_\eps^\pm} B_\eps(p) \subset \Omega^\pm$
for the last assertion):
\begin{lemma}[properties of the symmetrisation]
  \label{lem:sym.harm.ext}
  \indent
  \begin{enumerate}
  \item
    \label{sym.harm.ext.a}
    $S$ is idempotent and self-adjoint (i.e., an orthogonal
    projection).  In particular, $Sf$ is symmetric and
    $f-S f=\Gamma f/2$ is anti-symmetric on $\Omega_{r_1}$; any
    $f \in \Lsqr{X,g}$ has the unique decomposition $f=S f + (f-S f)$
    into its symmetric and anti-symmetric part.
  \item
    \label{sym.harm.ext.b}
    If $f_1$ is symmetric and $f_2$ is anti-symmetric, then
    \begin{align*}
      \iprod[\Lsqr{\Omega',g}]{f_1}{f_2}=0
      \qquadtext{and}
      \iprod[\Lsqr{\Omega',g}]{df_1}{df_2}=0
    \end{align*}
    for any symmetric $\Omega' \subset \Omega_{r_1}$ (i.e.,
    $\phi(\Omega')=\Omega'$) and $f_1,f_2 \in \Lsqr{\Omega',g}$ resp.\
    $f_1,f_2 \in \Sob{\Omega',g}$.
  \item
    \label{sym.harm.ext.c}
    We have
    \begin{align*}
      \norm[\Lsqr{\Omega',g}] {S f}
      &\le \norm[\Lsqr{\Omega',g}] f, \quad
        \norm[\Lsqr{\Omega',g}] {d(S f)} \le \norm[\Lsqr{\Omega',g}] {df},\\
      \norm[\Lsqr{\Omega',g}] {f-S f}
      &\le \norm[\Lsqr{\Omega',g}] f
        \quadtext{and}
        \norm[\Lsqr{\Omega',g}] {d(f-S f)} \le \norm[\Lsqr{\Omega',g}] {df}
    \end{align*}
    for $f \in \Lsqr{\Omega',g}$ resp.\ $f \in \Sob{\Omega',g}$ and
    any open set $\Omega' \subset \Omega_{r_1}$.

  \item
    \label{sym.harm.ext.d}
    We have $S E_\eps u = E_\eps S \bar u$ for
    $u \in \Sob{X_\eps,g}$.
  \end{enumerate}
\end{lemma}
\begin{figure}[h]
  \centering
  \begin{picture}(0,0)%
    \includegraphics{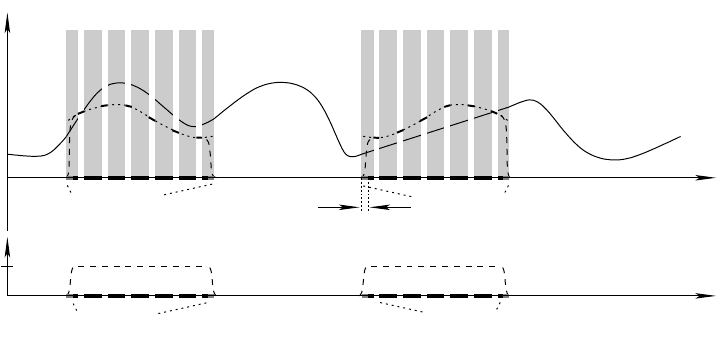}%
  \end{picture}%
  \setlength{\unitlength}{4144sp}%
  \begin{picture}(5472,2601)(394,-1759) %
    \put(3063,-961){$r_1$}%
    \put(5891,-556){$x$}%
    \put(5891,-1456){$x$}%
    \put(2057,-214){$SE_\eps u$}%
    \put(496,-1096){$1-\chi_\eps(x)$}%
    \put(496,-241){$\bar u$}%
    \put(946,-736){$\Omega^+_{r_1}\cap X_\eps$}%
    \put(3556,-736){$\Omega^-_{r_1} \cap X_\eps$}%
    \put(991,-1681){$\Omega^+ \cap X_\eps$}%
    \put(3646,-1681){$\Omega^- \cap X_\eps$}%
    \put(2018,-466){$(1-\chi_\eps)S E_\eps u$}%
    \put(496,699){$\bar u(x)$, $(SE_\eps u)(x)$}%
  \end{picture}%
  \caption{The identification operators $J_\eps'$ and
    $J_\eps^{\prime1}$ and its ingredients: $\bar u$ (full line) is
    the extension by $0$ onto $B_\eps$; $S E_\eps u$ (dotted and
    thick line) is the symmetrisation of the harmonic extension of $u$
    and $(1-\chi_\eps)S E_\eps u$ (dashed) is the smoothed
    function symmetric on $\Omega$ using the cut-off function
    $1-\chi_\eps$ with support in $\Omega_{\wt \eps}$.}
  \label{fig:symmetrisation2}
\end{figure}
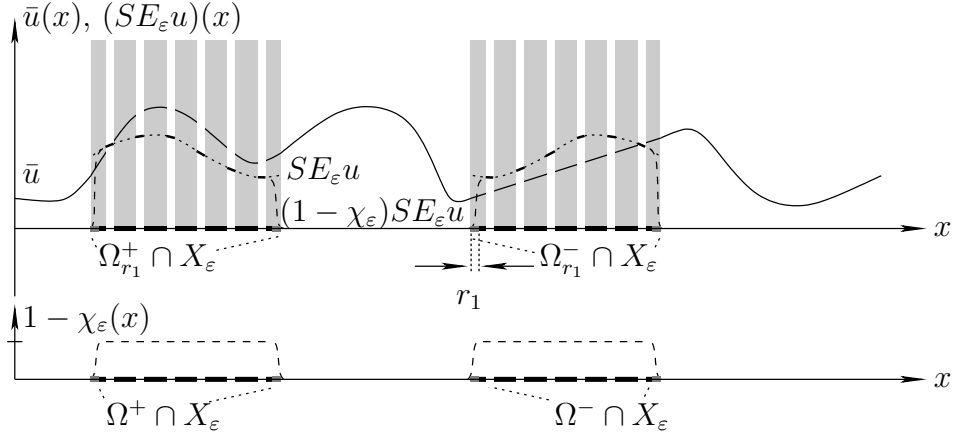

Before proving \Thm{handles3}, we need some more lemmas:

First, we need an estimate on anti-symmetric functions.  The parameter
range where $\deltaAntisym \eps \to 0$ is shown in \Fig{anti-sym}.
\begin{lemma}[estimate on anti-symmetric functions]
  \label{lem:ua}
  We have
  \begin{equation*}
    \norm[\Lsqr{B_{\eta_\eps}\setminus B_\eps}]{u - S u}
    \le \deltaAntisym \eps
    \norm[\HS_\eps^1] U
  \end{equation*}
  for $U=(u,h) \in \HS^1_\eps$, where
  \begin{equation}
    \label{eq:delta.antisym}
    \deltaAntisym \eps^2
    \coloneqq
    \frac{N K^{m/2}}{m} \cdot
    \frac{\eta_\eps^m}{\eps^{m-2}} \cdot
    \max \Bigl\{\frac {\ell_\eps}\eps,
      4 \Bigl[\log\frac {\eta_\eps}\eps\Bigr]_2
    \Bigr\}.
  \end{equation}
\end{lemma}
\begin{proof}
  Let $p \in I_\eps^-$ and $(\eps,\theta)$ be the polar coordinates of
  $x \in \bd B_\eps(p) $.
  Then we have
  \begin{align*}
    (u-S\bar u)_{p}(\eps,\theta)
    &= (u-S\bar u)(x)
    = \frac 12\bigl(u(x)-u(\phi(x)\bigr)\\
    &= \frac 12\bigl(u_{p}(\eps,\theta)-u_{\bar p}(\eps,\theta)\bigr)\\
    &= -\frac 12 (\eps^{m-1} \ell_\eps)^{-1/2} (h_p(1,\theta)-h_p(0,\theta))
    &\\
    &= -\frac 12 (\eps^{m-1} \ell_\eps)^{-1/2} \int_0^1 \partial_1 h_p(s,\theta)
      \dd s.
  \end{align*}
  due to the coupling condition~\eqref{eq:qeps.a} for the third
  equality.  In particular, we have
  \begin{align}
    \nonumber
    \int_{\Sphere^{m-1}} \abssqr{(u-S\bar u)_{p} (\eps,\theta)} \dd \theta
    &\le \frac 1{4\eps^{m-1}\ell_\eps}
    \int_0^1 \int_{\Sphere^{m-1}} \abssqr{\partial_1 h_p(s,\theta)}
      \dd \theta \dd s\\
    \label{eq:u-us}
    &\le \frac {\ell_\eps}{4\eps^{m-1}}
    \qf d_{\Cyl_\eps(p)}(h_p)
  \end{align}
  due to the Cauchy-Schwarz inequality and~\eqref{eq:qeps.c}.

  As $(u-S \bar
  u)(\cdot,\theta)$ is weakly differentiable on
  $\Omega_{\eta_\eps}\subset\Omega_{r_1}$, we have for $r \in
  [\eps,\eta_\eps)$
  \begin{align*}
    (u-S\bar u)_{p}(r,\theta)
    = (u-S\bar u)_{p}(\eps,\theta)
    + \int_\eps^r \partial_1(u-S\bar u)_{p}(s,\theta) \dd s
  \end{align*}
  Using the Cauchy-Young inequality and Fubini we conclude
  \begin{align*}
    \int_{B_{\eta_\eps}(p) \setminus B_\eps(p)}
    \abssqr{u-S\bar u} \dd g_\eucl
    &= \int_\eps^{\eta_\eps} \int_{\Sphere^{m-1}}
    \abssqr{(u-S\bar u)_{p}(r,\theta)}
      \dd \theta\, r^{m-1} \dd r\\
    & \le
      2 \int_\eps^{\eta_\eps} \int_{\Sphere^{m-1}}
       \abssqr{(u-S\bar u)_{p}(\eps,\theta)}
      \dd \theta\, r^{m-1} \dd r\\
    &\hspace*{0.1\textwidth}
      + 2 \int_\eps^{\eta_\eps} \int_{\Sphere^{m-1}}
      \Bigabssqr{\int_\eps^r \partial_1(u-S\bar u)_{p}(s,\theta) \dd s}
      \dd \theta\, r^{m-1} \dd r.
  \end{align*}
  For the first term in the last estimate, we use~\eqref{eq:u-us} and
  obtain
  \begin{align*}
      2 \int_\eps^{\eta_\eps} \int_{\Sphere^{m-1}}
       \abssqr{(u-S\bar u)_{p}(\eps,\theta)}
      \dd \theta\, r^{m-1} \dd r
    &=\frac2m\bigl(\eta_\eps^m-\eps^m\bigr)
    \int_{\Sphere^{m-1}} \abssqr{(u-S\bar u)_{p}(\eps,\theta)} \dd \theta\\
    &\le \frac{\ell_\eps}{2m \eps^{m-1}}\bigl(\eta_\eps^m-\eps^m\bigr)
    \qf d_{\Cyl_\eps(p)}(h_p)\\
    &\le \frac{\ell_\eps}{2m} \cdot \frac{\eta_\eps^m}{\eps^{m-1}}
    \qf d_{\Cyl_\eps(p)}(h_p).
  \end{align*}
  Using \Cor{eucl.metric2} on the annulus
  $B_{\eta_\eps}(p)\setminus B_\eps (p)$ (hence the factor
  $K^{m/2}$), summing over $p \in I_\eps=I_\eps^- \dcup I_\eps^+$
  (hence the factor $2$) and using that $I_\eps$ has $\eta_\eps$-cover
  number $N$ (hence the factor $N$), we obtain the first term in the
  maximum in~\eqref{eq:delta.antisym}.

  For the second term, we use again Cauchy-Schwarz and obtain
  \begin{multline*}
    \int_\eps^{\eta_\eps}\int_{\Sphere^{m-1}}
    \Bigabssqr{\int_\eps^r \partial_1(u-S\bar u)_{p}(s,\theta) \dd s}
      \dd\theta\, r^{m-1} \dd r\\
    \le \int_\eps^{\eta_\eps} \int_\eps^r s^{1-m}\dd s \, r^{m-1}\,\dd r \cdot
      \int_\eps^{\eta_\eps}\int_{\Sphere^{m-1}}
      \abssqr{\partial_1(u-S\bar u)_{p}(s,\theta)} s^{m-1}
      \dd s \dd\theta
  \end{multline*}
  (this calculation gives also the techniques for a proof ``\`a la
  main'' of~\cite[Proposition 6.7]{anne-post:21}).  The last integral
  can be estimated by
  \begin{align*}
    K^{m/2}\normsqr[\Lsqr{B_{\eta_\eps}(p)\setminus B_\eps(p)}]
    {\partial_1(u-S\bar u)}
    \le 2K^{m/2}\bigl(\normsqr[\Lsqr{B_{\eta_\eps}(p)\setminus B_\eps(p)}] {du}
    + \normsqr[\Lsqr{B_{\eta_\eps}(\bar p)\setminus B_\eps(\bar p)}] {du}\bigr)
  \end{align*}
  using \Cor{eucl.metric2} on an annulus (hence the factor $K^{m/2}$).
  Moreover, for the first integral, we have
  \begin{align*}
    \int_\eps^{\eta_\eps} \int_\eps^r s^{1-m}\dd s r^{m-1}\dd r
    &
      = \int_\eps^{\eta_\eps} (h_m(r)-h_m(\eps)) r^{m-1}\dd r
    \\
    &=
      \begin{cases}
        \dfrac1{m-2} \Bigl(
        \dfrac 1m \cdot \Bigl(\dfrac{\eta_\eps^m}{\eps^{m-2}}-\eps^2\Bigr)
        - \dfrac12\cdot(\eta_\eps^2-\eps^2)
        \Bigr), & m \ge 3\\[1ex]
        \dfrac{\eta_\eps^2}2 \Bigl(\log\Bigl(\dfrac {\eta_\eps}\eps\Bigr)
        -\dfrac12\Bigr)
        +\dfrac14 \eps^2, & m=2
      \end{cases}\\
      &
    \le \frac 1 m \Bigl[\log\frac {\eta_\eps}\eps\Bigr]_2
       \frac{\eta_\eps^m}{\eps^{m-2}},
  \end{align*}
  where $h_m(s)=s^{2-m}/(2-m)$ for $m \ge 3$ and $h_2(s)=\log s$,
  provided $\eps \le \eta_\eps/2$.  We conclude using that $I_\eps$
  has $\eta_\eps$-cover number $N$.
\end{proof}

As in~\cite[Lem.~A.3]{anne-post:21} we conclude from
\Ass{unif.tub.nbhd} the following estimate:
\begin{lemma}[general estimate on tubular neighbourhood]
  \label{lem:ev.tub.nbhd'}
  There is a constant $\Cnbhd$ depending only on $r_0$ and
  $\bd \Omega$ such that
  \begin{equation*}
    \norm[\Lsqr{\Omega_{\wt \eps} \setminus \Omega}] f
    \le \Cnbhd \sqrt{\wt \eps}
    \norm[\Sob{\Omega_{r_1} \setminus \Omega}] f
  \end{equation*}
  for $f \in \Sob {\Omega_{r_1} \setminus \Omega}$.
\end{lemma}
\begin{proof}
  A proof is indicated in~\cite[Lemma~A.3]{anne-post:21} using the
  almost product structure of the tubular neighbourhood
  $\Omega_{r_1} \setminus \Omega$ as in the proof of the previous
  lemma.
\end{proof}
In the sense of \Def{non-concentr},
$(\Omega_{\wt \eps}\setminus \Omega,\Omega_{r_1} \setminus \Omega)$ is
$\Cnbhd\sqrt{\wt \eps}$-non-concentrating, and hence we can also
apply \Prp{non-concentr2}.  The estimate in \Lem{ev.tub.nbhd'} cannot
be improved: if $\bd \Omega$ is compact and $f$ has constant value
$1$, then
$\norm[\Lsqr{\Omega_{\wt \eps} \setminus \Omega}]
f=\sqrt{\vol(\Omega_{\wt \eps} \setminus \Omega)}$ and this term is of
order $\sqrt{\wt \eps}$.

\subsection{Proof of \Thm{handles3} (adhering handles)}
\label{ssec:proof.main3}

We now define the following transplantation operators
\begin{align*}
    \map {J_\eps} {\HS_0} {&\HS_\eps},
    & f & \mapsto (f \restr {X_\eps},0),\\
    \map {J_\eps^1} {\HS_0^1} {&\HS^1_\eps},
    & f & \mapsto (f \restr {X_\eps}, \Phi_\eps f)\\
   \map{J_\eps'}{\HS_\eps}{&\HS_0},
     & U=(u,h) & \mapsto \1_{X \setminus \Omega} u + \1_\Omega S\bar u,\\
   \map{J_\eps^{\prime1}}{\HS^1_\eps}{&\HS^1_0},
    &  U=(u,h) & \mapsto \chi_\eps u + (1-\chi_\eps) S E_\eps u
\end{align*}
(see~\eqref{eq:limit.solid} for the definition of $\HS_0$ and
$\HS_0^1$), where\footnote{We write a bit pedantic $S \bar u$ as $S$
  needs a function defined on $X$, and not only on $X_\eps$.} %
(as before) $\bar u$ is the extension of $u \in \Lsqr{X_\eps,g}$ onto
$\Lsqr{X,g}$ by $0$, and $E_\eps u$ is the harmonic extension,
cf.~\eqref{eq:ext.op}.  Note that $J_\eps' U \in \HS_0$ and
$J_\eps^{\prime1} U \in \HS_0^1$ as
$J_\eps' U \restr \Omega =S \bar u \restr \Omega$ and similarly
$J_\eps^{\prime1} U \restr \Omega = S E_\eps u \restr \Omega$ are both
symmetric on $\Omega$ (see also \Fig{symmetrisation2}).
Moreover,
$(1-\chi_\eps)S E_\eps u \in \Sob {X,g}$ as $1-\chi_\eps$ is
smooth and $0$, where $S \bar u$ has a jump, namely at
$\bd \Omega_{r_1}$.

Let us now prove \Thm{handles3} by checking the conditions of
\Def{quasi-uni}.

\begin{proof}[Proof of \Thm{handles3}]
  Let $f\in\HS_0^1$ and $U=(u,h) \in\HS_\eps^1$.  Then we have
  \begin{subequations}
  \begin{align*}
    \iprod[\HS_\eps] {J_\eps f} U
    &= \iprod[\Lsqr{X \setminus \Omega,g}] f u
      + \iprod[\Lsqr{\Omega \setminus B_\eps,g}] {Sf} u\\
    &= \iprod[\Lsqr{X \setminus \Omega,g}] f u
      + \iprod[\Lsqr{\Omega \setminus B_\eps,g}] f {Su}\\
    &= \iprod[\Lsqr{X \setminus \Omega,g}] f u
      + \iprod[\Lsqr{\Omega,g}] f {S\bar u}\\
    &= \iprod[\HS_0] f {J_\eps'U}
  \end{align*}
  as $Sf=f$ on $\Omega$ for the first equation, the selfadjointness of
  $S$ for the second equation and that $\bar u=0$ on $B_\eps$ for the
  third.  In particular, $J_\eps'=J_\eps^*$ and one easily sees
  $\norm{J_\eps}=\norm{J_\eps'}=1$.  Moreover, we have
  \begin{align*}
    f-J_\eps'J_\eps f
    = f-\1_{X \setminus \Omega} f - \1_{\Omega \setminus B_\eps} S f
    = f-\1_{X_\eps} f
    = \1_{B_\eps} f
  \end{align*}
  as $Sf=f$ and therefore $J_\eps f-J_\eps J'_\eps J_\eps f=0$, i.e.,
  $J_\eps=J_\eps J^*_\eps J_\eps$ and $J_\eps$ is a partial isometry;
  and
  \begin{align}
    \label{eq:handles3.est1}
    \norm[\HS_0]{f-J_\eps'J_\eps f}
    = \norm[\Lsqr{B_\eps,g}] f
    \le \OptNonConc_m(\eps,\eta_\eps)
      \norm[\Sob{B_{\eta_\eps},g}] f.
  \end{align}
  We also have
  $U-J_\eps J_\eps'U = (\1_{\Omega \setminus B_\eps}(u-Su),h)$ for
  $U=(u,h) \in \HS^1_\eps$, and therefore
  \begin{align}
    \label{eq:handles3.est2}
    \norm[\HS_\eps]{U-J_\eps J_\eps'U}
    \le \bigl(\deltaAntisym \eps^2 + \deltaHandle \eps^2\bigr)^{1/2}
      \norm[\HS_\eps^1] U
    \le \bigl(\deltaAntisym  \eps + \deltaHandle \eps \bigr)
    \norm[\HS_\eps^1] U
  \end{align}
  applying \Lem{ua} for the first component (note that
  $\Omega \setminus B_\eps \subset B_{\eta_\eps} \setminus B_\eps$)
  and \Lem{dech2} for the second component.  Moreover,
  \begin{align}
    \label{eq:handles3.est3}
    \norm[\HS_\eps]{(J_\eps^1-J_\eps)f}
    =\norm[\Lsqr{\Cyl,g_\can}]{\Phi_\eps f}
    \le \deltaHarm \eps \norm[\Sob{B_{\eta_\eps} \setminus B_\eps,g}] f
  \end{align}
  by \Cor{harm}.  Next,
  \begin{align}
    \nonumber
    \normsqr[\HS_0]{(J_\eps^{\prime1}-J_\eps')U}
    &= \normsqr[\Lsqr{\Omega_{\wt \eps} \setminus \Omega,g}]{(1-\chi_\eps)(u-Su)}
      + \normsqr[\Lsqr{B_\eps,g} ]{E_\eps u}\\
    \nonumber
    &\le \normsqr[\Lsqr{B_{\eta_\eps} \setminus B_\eps,g}]{u-Su}
      + \OptNonConc_m(\eps,\eta_\eps)^2 \normsqr[\Sob{B_{\eta_\eps},g}]{E_\eps u}\\
    \nonumber
    &\le \bigl(\deltaAntisym \eps^2
      + \OptNonConc_m(\eps,\eta_\eps)^2 \Cext^2 \bigr)
     \normsqr[\HS^1_\eps] U\\
    \label{eq:handles3.est4}
    &\le \bigl(\deltaAntisym \eps
      + \OptNonConc_m(\eps,\eta_\eps) \Cext\bigr)^2
     \normsqr[\HS^1_\eps] U
  \end{align}
  using again
  $\Omega_{\wt \eps} \setminus \Omega \subset B_{\eta_\eps} \setminus
  B_\eps$ and \Prp{0} for the first estimate, and \Lem{ua}
  and~\eqref{eq:ext.op} for the second.

  Finally, for the quadratic forms, we have
  \begin{align}
    \nonumber
    \qf d_\eps(J_\eps^1f,U) - \qf d_0(f,J_\eps^{\prime1} U)
    &= \int_{X_\eps}
      \bigiprod{df}{d\bigl((1-\chi_\eps)(u-S  \bar u)\bigr)} \dd g
      +  \qf d_{\Cyl_\eps}(\Phi_\eps f,h)\\
    \nonumber
    & \hspace*{0.1\textwidth}
      - \int_{B_\eps} \iprod{df} {d (E_\eps S  \bar u)} \dd g\\
    \nonumber
    &=
       \int_{\Omega_{\wt \eps} \setminus \Omega}
      \bigiprod{df}
      {d\bigl((1-\chi_\eps)(u-S  \bar u)\bigr)} \dd g
      + \qf d_{\Cyl_\eps}(\Phi_\eps^\perp f,h)\\
    & \hspace*{0.1\textwidth}
      - \int_{B_\eps} \iprod{df} {d (E_\eps S  \bar u)}  \dd g
      \label{eq:handles3.est}
  \end{align}
  using $E_\eps u \restr {X_\eps}=u$ and
  \Lemenum{sym.harm.ext}{sym.harm.ext.d} for the first equality. For
  the second equality, note that
  $\supp(1-\chi_\eps)=\Omega_{\wt \eps}$.  Moreover, $\chi_\eps$ is
  symmetric and $f$ symmetric (but only) on $\Omega$, and hence
  $(1-\chi_\eps)(u-S  \bar u)$ is anti-symmetric, and the integral
  over $\Omega'=\Omega \setminus B_\eps$ vanishes according to
  \Lemenum{sym.harm.ext}{sym.harm.ext.b}.  Finally, for the second
  term note that
  $\Phi_\eps f = \Phi_\eps^\bullet f + \Phi_\eps^\perp f$ and
  $\Phi_\eps^\bullet f$ is constant on the handles as $f$ is
  symmetric, hence $\qf d_{\Cyl_\eps}(\Phi_\eps^\bullet f,h)=0$.

  For the first summand in~\eqref{eq:handles3.est} we have
  \begin{align}
    \nonumber
    \Bigabssqr{
    \int_{\Omega_{\wt \eps} \setminus \Omega}
    \bigiprod{df}
    {d\bigl((1-\chi_\eps)(u-S  \bar u)\bigr)} \dd g
    }
    \\& \hspace*{-0.3\textwidth}\leq
    \nonumber
    2\int_{\Omega_{\wt \eps} \setminus \Omega}\abssqr{df}\dd g \cdot
    \Bigl(\frac2{\wt \eps^2}
    \normsqr[\Lsqr{\Omega_{\wt \eps} \setminus \Omega,g}]{u-S  \bar u}
    +\normsqr[\Lsqr{\Omega_{\wt \eps} \setminus \Omega,g}]{d(u-S \bar u)}
    \Bigr)
    \\& \hspace*{-0.3\textwidth}\le
    \label{eq:handles3.est5}
    2 (\Cnbhd)^2
    \biggl(\frac{2\deltaAntisym \eps^2}{\wt \eps}  + \wt \eps \biggr)
    \normsqr[{\Sob[2]{\Omega_{r_0} \setminus \Omega,g}}] f \normsqr[\HS ^1] U
  \end{align}
  using \Lem{ev.tub.nbhd'} together with \Prp{non-concentr2} for the
  first factor, and \Lem{ua} resp.\
  $\norm{d(u-S  \bar u)} \le \norm{du} \le \norm[\HS^1_\eps] U$
  from \Lemenum{sym.harm.ext}{sym.harm.ext.c} for the second.

  To control the second summand in~\eqref{eq:handles3.est} we
  apply \Prp{harm2}, i.e., we have
  \begin{align}
    \label{eq:handles3.est6}
    \bigabs{\qf d_{\Cyl_\eps}(\Phi_\eps^\perp f,h)}
    \le \deltaHarm \eps^\perp \norm[{\Sob[2]{X,g}}] f \norm[\HS^1_\eps] U.
  \end{align}

  For the third summand in~\eqref{eq:handles3.est}, we have
  \begin{align}
    \nonumber
    \Bigabs{\int_{B_\eps} \iprod{df} {d (E_\eps S  \bar u)}  \dd g}
    &\le \Cext \norm[\Lsqr{B_\eps,g}]{df} \norm[\Sob{X_\eps,g}] u\\
    \label{eq:handles3.est7}
    &\le \Cext \OptNonConc_m(\eps,\eta_\eps)
      \norm[{\Sob[2]{B_{\eta_\eps},g}}]f \norm[\Sob{X_\eps,g}] u
  \end{align}
  using the uniform bound on the extension operator~\eqref{eq:ext.op}
  and \Prps{0}{non-concentr2}.  Finally, we pass to the graph norm
  \begin{align}
    \label{eq:handles3.est8}
    \norm[{\Sob[2]{B_{\eta_\eps},g}}]f
    \le \Cellreg \norm[\Lsqr{X,g}]{(\Delta_0+1)f}
  \end{align}
  using \Prp{ell.reg}, where $\Delta_0$ is the operator associated with
  $\qf d_0$.  Note that $\dom \Delta_0$ is a subset of the domain of the
  entire Laplacian on $X$, hence \Prp{ell.reg} can also be applied in
  this situation.
\end{subequations}

In summary, from~\eqref{eq:handles3.est1}--\eqref{eq:handles3.est4}
and \eqref{eq:handles3.est5}--\eqref{eq:handles3.est8} we conclude
that \Def{quasi-uni} can be satisfied with
\begin{align}
  \nonumber
    \delta_\eps
    &=\max \Bigl\{
      \OptNonConc_m(\eps,\eta_\eps),
      \deltaAntisym  \eps + \deltaHandle \eps,
      \deltaHarm \eps,
      \deltaAntisym  \eps
      + \Cext \OptNonConc_m(\eps,\eta_\eps),\\
  \nonumber
    &\hspace*{0.1\textwidth}
      \Cellreg \Bigl(
    \Bigl(\sqrt 2 \Cnbhd
    \biggl(\frac{2\deltaAntisym \eps^2}{\wt \eps} + \wt \eps \biggr)^{1/2}
      + \deltaHarm \eps^\perp
      + \Cext \OptNonConc_m(\eps,\eta_\eps)
    \Bigr)
      \Bigr\}\\
    \label{eq:err.handles3}
    &= \Err\Bigl(\frac{\deltaAntisym \eps}{\sqrt {\wt \eps}}+\sqrt{\wt \eps}\Bigr)
    + \Err \Bigl(\OptNonConc_m(\eps,\eta_\eps)\Bigr)
    + \Err(\deltaHandle \eps) + \Err(\deltaHarm \eps^\perp)
  \end{align}
  using $\deltaHarm \eps \le \deltaHandle \eps$.
\end{proof}

\appendix

%
\section{Some estimates on Euclidean balls, annuli and spheres}
\label{app:eucl.balls}
%
We prove here the non-concentrating property needed in \Prp{0} for
pairs of Euclidean balls in dependence of their radii.  We actually
calculate the optimal constant.  Some results follow from the abstract
theory of Dirichlet-to-Neumann maps and boundary pairs (see
e.g.~\cite{post:16}).  For the convenience of the reader, we give
proofs without referring to the abstract theory.  Note that a simple
scaling argument as in \Lem{muf} does only lead to the worse
estimate~\eqref{eq:bd.ball.est'}, see \Rem{opt.const}.

For $r>0$, denote by $B_r$ the open ball around $0$ in $\R^m$,
together with the Euclidean metric $g_\eucl$.  For short, we also
write $B_r$ for the Riemannian manifold $(B_r,g_\eucl)$, and similarly
for other subsets of $\R^m$.  Moreover, denote by
$\embmap {\iota_r}{\bd B_r} {B_r}$ the natural embedding, and by
$\iota_r^* g_\eucl$ the induced metric on $\bd B_r$: the standard
metric on a sphere of radius $r>0$.  We denote by $\bd B_r$ also the
Riemannian manifold $(\bd B_r, \iota_r^* g_\eucl)$, i.e., the standard
$(m-1)$-dimensional sphere of radius $r>0$.

For the non-concentrating property
\Def{non-concentr}, we actually need the estimate
\begin{equation}
    \label{eq:ball.est}
  \norm[\Lsqr{B_\eps}] f
  \le \OptNonConc(B_\eps,B_\eta) \norm[\Sob{B_\eta}] f
\end{equation}
for all $f \in \Sob{B_\eta}$ and $0<\eps<\eta$.
Here, the optimal constant is
\begin{equation}
  \label{eq:c.opt'}
  \OptNonConc(A,B)
  := \sup_{f\in \Sob{B}}
  \frac{\norm[\Lsqr{A}] f}%
  {\norm[\Sob{B}] f}
\end{equation}
for reasonable subsets $A,B$ of a Riemannian manifold with
$A \subset B$.

The aim of this section is to give an asymptotic expansion of the
optimal constant of the non-concentrating property for $A=B_\eps$ and
$B=B_\eta$.  For a comment on the previous result
of~\cite[Lem.~3.10]{anne-post:21}, we refer to \Rem{opt.const}.

We also need the Sobolev trace estimate
\begin{equation}
  \label{eq:bd.bd.est}
  \norm[\Lsqr Y] f
  \le \OptSobTr(Y,X) \norm[\Sob X] f
\end{equation}
for all $f \in \Sob X$, where $Y \subset X$ is a hypersurface (with
its natural measure) in a Riemannian manifold $(X,g)$.  The optimal
constant is here
\begin{equation}
  \label{eq:opt.sob.tr}
  \OptSobTr(Y,X)
  := \sup_{f \in \Sob X}Q'(u;Y,X)
  \quadtext{with}
  Q'(u;Y,X)
  := \frac{\norm[\Lsqr Y] {u\restr Y}}%
  {\norm[\Sob X] u}.
\end{equation}
From~\cite[Thm.~2.7~(ii)]{post:16} we know that
$\OptSobTr(Y,X)^2=1/\inf \spec{\Lambda_{Y,X}}$, where $\Lambda_{Y,X}$
is the Dirichlet-to-Neumann operator (at the spectral value $-1$) on
$X$ with boundary space $Y$ and with Neumann boundary conditions on
$\bd X \setminus Y$ (if not empty), i.e.,
$\Lambda_{Y,X}\phi=\normder \wt u$ for a suitable function
$\map \phi Y \C$, where $\normder \wt u$ is the normal (outwards)
derivative\footnote{If $Y$ lies on both sides of $X$, then we take the
  sum of the two normal outwards derivatives of each side.} %
of $\wt u$ on $Y$ and $\wt u \in \Sob X$ is the weak solution of
$(\Delta+1)\wt u$ with $\wt u \restr Y=\phi$ and $\normder \wt u=0$ on
$\bd X \setminus Y$ (also called the \emph{harmonic extension (at the
  spectral value $-1$)}), i.e.,
\begin{equation}
  \label{eq:weak.sol}
  \forall v \in \Sob X, v \restr Y=0 \colon \;
  \iprod[\Sob X]{\wt u}v
  =\int_X \bigr(\iprod{d \wt u}{d v} + \wt u  \conj v \bigl)\dd g
    =0.
\end{equation}

The following relation between the optimal constants for balls holds:
\begin{lemma}
  \label{lem:opt.const.rel}
  We have
  \begin{align*}
    \OptNonConc(B_\eps,B_\eta)^2
    = \int_0^\eps \OptSobTr(\bd B_r,B_\eta)^2 \dd r.
  \end{align*}
\end{lemma}
\begin{proof}
  Using
  $\normsqr[\Lsqr {B_\eps}] f =\int_0^\eps \normsqr[\Lsqr{\bd B_r}]{f
    \restr {\bd B_r}} \dd r$ we obtain
  \begin{align*}
    \OptNonConc(B_\eps,B_\eta)^2
    = \sup_{f \in \Sob {B_\eta}}
     \int_0^\eps\frac{\normsqr[\Lsqr{\bd B_r}]{f \restr {\bd B_r}}}%
    {\normsqr[\Sob{B_\eta}] f} \dd r
    = \int_0^\eps \Bigl(\sup_{f \in \Sob {B_\eta}}
    \frac{\norm[\Lsqr{\bd B_r}]{f \restr {\bd B_r}}}%
    {\norm[\Sob{B_\eta}] f} \Bigr)^2 \dd r,
  \end{align*}
  where the interchange of the supremum and the integral is allowed to
  the monotonicity of the integral, the square function on positive
  numbers and the monotone convergence theorem.
\end{proof}

\begin{lemma}[estimate on spherical shells against an annulus]
  \label{lem:bd.ball.est}
  We have
  \begin{equation}
    \label{eq:bd.ball.est1}
    \OptSobTr(\bd B_r,B_\eta \setminus B_r)^2
    = -\frac 1 {f_{0,r,\eta}'(r)},
  \end{equation}
  where $f=f_{0,r,\eta}$ is the solution of
  \begin{equation}
    \label{eq:f.ode}
    -\frac1{s^{m-1}} \bigl(s^{m-1}f'(s)\bigr)' + f(s)=0
  \end{equation}
  on $r \le s \le \eta$ with boundary conditions $f(r)=1$ and
  $f'(\eta)=0$.
\end{lemma}
\begin{proof}
  \begin{subequations}
    We have seen above that the optimal constant
    $\OptSobTr(\bd B_r,B_\eta \setminus B_r)^2$ is
    given by the inverse of the infimum of
    $\spec{\Lambda_{\bd B_r,B_\eta\setminus B_r}}$.  Moreover, this
    infimum is given by the radially symmetric part, i.e., by
    $-f'(r)$.  For convenience of the reader, we give a proof
    of~\eqref{eq:bd.ball.est1} here:

    The inequality ``$\ge$'' is easily seen by calculating
    $Q_m'(\wt f;r,\eta):=Q'(\wt f;\bd B_r,B_\eta\setminus B_r)$ for
    $\wt f(s,\theta)=f(s)$ in polar coordinates
    $(s,\theta) \in (r,\eta)\times \Sphere$.  We have
    \begin{align*}
      \normsqr[\Sob{B_\eta\setminus B_r}] {\wt f}
      &= \omega_m' \int_r^\eta
        \Bigl(
        (f'(s))^2 + (f(s))^2 s^{m-1}\dd s\\
      &= \omega_m' \int_r^\eta
        \Bigl(
        - \frac1{s^{m-1}}\bigl(s^{m-1}f'(s)\bigr)'
        +  f(s)
        \Bigr) f(s) s^{m-1} \dd s
        + \omega_m'\Bigl[f'(s) f(s) s^{m-1}\Bigr ]_r^\eta\\
      &= -\omega_m' r^{m-1}f'(r),
    \end{align*}
    where $\omega_m'=\vol_{m-1} \Sphere$, as $f$ fulfils the differential
    equation~\eqref{eq:f.ode}, $f'(\eta)=0$ and $f(r)=1$.  Moreover,
    $\normsqr[\Lsqr{\bd B_r}] {\wt f}=\omega_m'r^{m-1}$, hence we have the
    desired inequality by definition of the supremum.

    For ``$\le$'' we argue as follows: First, we can calculate the
    supremum in two steps as
    \begin{align*}
      \OptSobTr(\bd B_r,B_\eta \setminus B_r)
      = \sup_{\phi \in \HSaux^{1/2},}
      \sup_{\substack{u \in \Sob{B_\eta \setminus B_r}\\ u \restr {\bd B_r}=\phi}}
      Q_m'(u;r,\eta),
    \end{align*}
    where
    $\HSaux^{1/2}=\set{u \restr {\bd B_r}} {u \in \Sob{B_\eta\setminus
        B_r}}$ and $Q_m'(u;r,\eta)=Q(u;\bd B_r,B_\eta\setminus B_r)$.
    Now, the second supremum is maximised for the (unique) harmonic
    extension $\wt u$ (at the spectral value $-1$) with boundary data
    $\phi$ (see \eqref{eq:weak.sol} with $X=B_\eta \setminus B_r$ and
    $Y=\bd B_r$).  Then $v=u-\wt u$ fulfils $v \restr {\bd B_r}=0$,
    hence $v$ and $\wt u$ are orthogonal in
    $\Sob{B_\eta \setminus B_r}$ and we have
    \begin{align*}
      \normsqr[\Sob{B_\eta\setminus B_r}] u
      = \normsqr[\Sob{B_\eta\setminus B_r}] {\wt u + v}
      &= \normsqr[\Sob{B_\eta\setminus B_r}] {\wt u}
        + \normsqr[\Sob{B_\eta\setminus B_r}] v\\
      &\ge \normsqr[\Sob{B_\eta\setminus B_r}] {\wt u}.
    \end{align*}
    As $\wt u \restr {\bd B_r}= u \restr{\bd B_r}$, we then have
    $Q_m'(u,r,\eta)\le Q_m'(\wt u ,r,\eta)$, hence
    \begin{align*}
      \OptSobTr(\bd B_r,B_\eta \setminus B_r)
      \le \sup_{\phi \in \HSaux^{1/2}}
      Q_m'(\wt u,r,\eta)
      = \sup_{\phi \in \HSaux^{1/2}}
      \frac{\norm[\Lsqr{\bd B_r}] \phi}
      {\norm[\Sob{B_\eta \setminus B_r}] {S\phi}},
    \end{align*}
    where $S\phi:=\wt u$.

    Now, $\normsqr[\Sob{B_\eta \setminus B_r}] {S\phi}$ is the
    (non-negative) quadratic form of the Dirichlet-to-Neumann operator
    $\Lambda$ (at $-1$), hence we look at the infimum of the spectrum
    of this operator.  A separation of variables ansatz (decomposing
    $\Lsqr\Sphere=\bigoplus_{\mu \in \Mu} \C \phi_\mu$, where
    $\phi_\mu$ are the eigenvectors of $\Sphere$ and $\Mu$ is a
    multiset
    , i.e., it counts multiplicities) gives
    \begin{align}
      \label{eq:def.dtn}
      \Lambda = \bigoplus_{\mu \in \Mu} \Lambda_\mu,
    \end{align}
    and $\Lambda_\mu$ is the multiplication operator with a
    non-negative real number.  By partial integration as before, it
    can be seen that its value is
    \begin{equation}
      \label{eq:dtn.m}
      -f'(r)
      = \int_r^\eta
      \Bigl(
      \abssqr{f'(s)} + \frac \mu{s^2} \abssqr{f(s)} +  \abssqr{f(s)}
      \Bigr) s^{m-1}\dd s,
    \end{equation}
    where $f=f_{\mu,r,\eta}$ is the unique
    solution of
    \begin{equation}
      \label{eq:f.ode'}
      -\frac1{s^{m-1}} \bigl(s^{m-1}f'(s)\bigr)'
      + \Bigl(\frac \mu{s^2} +1\Bigr) f(s)=0,
      \quad f(r)=1,
      \quad f'(\eta)=0.
    \end{equation}
    As $-f'_{0,r,\eta}(r)\le -f'_{\mu,r,\eta}$ for any $\mu \in \Mu$
    by~\eqref{eq:dtn.m}, we see that the radially symmetric part
    $f_{0,r,\eta}=f_{r,\eta}$ ($\mu=0$) actually gives the infimum of
    the spectrum of $\Lambda$.
  \end{subequations}
\end{proof}

Similarly, we see:
\begin{lemma}[estimate on spherical shells against a ball]
  \label{lem:bd.ball.est2}
  We have
  \begin{equation}
    \label{eq:bd.ball.est2}
    \OptSobTr(\bd B_r,B_r)^2
    = \frac 1 {g_r'(r)},
  \end{equation}
  where $g=g_r$ is the solution of the ODE~\eqref{eq:f.ode} on
  $0 < s \le r$ with boundary conditions
  \begin{align}
    \label{eq:int.cond}
    g(r)=1
    \qquadtext{and}
    \int_0^r(\abssqr{g(s)}+\abssqr{g'(s)}) s^{m-1} \dd s<\infty
  \end{align}
  (the latter is a condition at $0$).
\end{lemma}

Finally, we calculate the optimal constant needed in \Lem{opt.const.rel}
\begin{lemma}[estimate on spherical shells in a larger ball]
  \label{lem:bd.ball.est3}
  We have
  \begin{equation}
    \label{eq:bd.ball.est3}
    \OptSobTr(\bd B_r,B_\eta)^2
    = \bigl(-f_{0,r,\eta}(r) + g_r'(r)\bigr)^{-1}
    = \Bigl(\frac 1{\OptSobTr(\bd B_r,B_\eta \setminus B_r)^2}
    +\frac 1 {\OptSobTr(\bd B_r,B_r)^2}\Bigr)^{-1},
  \end{equation}
  where $f_{0,r,\eta}$ and $g_r$ are the solutions of
  \Lem{bd.ball.est} and \Lem{bd.ball.est2}.
\end{lemma}

\begin{remark}
  \label{rem:abstract.dtn}
  Here, we consider $\bd B_r$ as ``boundary'' of $B_\eta$.  Using
  again the theory of boundary pairs (as in~\cite{post:16}) one can
  see that the optimal constant is again given by
  \begin{align*}
    \OptSobTr(\bd B_r,B_\eta)^2
    = \frac 1{\inf \spec{\Lambda_{\bd B_r,B_\eta}}},
  \end{align*}
  where $\Lambda_{\bd_r B,B_\eta}$ is the Dirichlet-to-Neumann
  operator (always at the spectral value $-1$) of the coupled problem
  $B_\eta = B_r \cup (B_\eta \setminus B_r)$.  It can be seen
  (see~e.g.~\cite[Prp.~5.3]{post:16}) that
  \begin{align*}
    \Lambda_{\bd B_r ,B_\eta}
    =\Lambda_{\bd B_r, B_r} + \Lambda_{\bd B_r,B_\eta \setminus B_r},
  \end{align*}
  where $\Lambda_{\bd B_r,B_r}$ is the Dirichlet-to-Neumann operator
  of the ball of radius $r$ and boundary $\bd B_r$ and
  $\Lambda_{\bd B_r,B_\eta \setminus B_r}$ is the operator $\Lambda$
  in~\eqref{eq:def.dtn} in the proof of \Lem{bd.ball.est}.  We have
  \begin{align*}
    \inf \spec{\Lambda_{\bd B_r,B_\eta}}
    =\inf \spec{\Lambda_{\bd B_r, B_r}}
    + \inf \spec{\Lambda_{\bd B_r,B_\eta \setminus B_r}}
    = g_r'(r)-f'_{0,r,\eta}(r)
  \end{align*}
  given both by the radially symmetric part (argument as in the proof
  of \Lem{bd.ball.est}).
\end{remark}

We now calculate the asymptotic expansion of the optimal constant:
\begin{proposition}
  \label{prp:copt.asymp}
  We have
  \begin{align}
    \label{eq:c.opt'.eucl}
    \OptSobTrEucl_m(r,\eta)^2
    := \OptSobTr(\bd B_r,B_\eta)^2
    = \frac{mr^{m-1}}{\eta^m} + h_m(r) + R'_m(r,\eta),
  \end{align}
  where $\gamma$ is Euler's constant and $0.115 < \log 2-\gamma < 0.116$,
  \begin{align}
    \label{eq:def.h}
    h_m(r) :=
    \begin{cases}
      (\log 2 - \gamma - \log r) r, &m=2,\\[1ex]
      \dfrac r{m-2}, &m \ge 3.
    \end{cases}
  \end{align}
  Moreover, there are constants $c_m'>0$ and $\eta_m \in (0,1]$
  depending only on $m$ such that
  \begin{equation*}
    \abs{R'_m(r,\eta)}
    \le c_m' \Bigl( \frac{r^{m-1}}{\eta^m} \cdot [-\log \eta]_2 \eta^2 + r^2\Bigr)
  \end{equation*}
  holds for all $0 < r < \eta < \eta_m$.  Finally, the function
  $\OptSobTrEucl_m(\cdot,\eta)$ is monotonely increasing for all such
  $\eta$.
\end{proposition}
We prove the proposition in several lemmas.
\begin{lemma}
  \label{lem:copt.concrete1}
  We have
  \begin{align*}
    \OptSobTr(\bd B_r, B_\eta \setminus B_r)^2
    = \frac {I_\nu(r) K_{\nu+1}(\eta)+K_\nu(r)I_{\nu+1}(\eta)}
    {K_{\nu+1}(r)I_{\nu+1}(\eta)-I_{\nu+1}(r)K_{\nu+1}(\eta)},
    \quadtext{where}
    \nu=\frac{m-2}2
  \end{align*}
  and where $I_\nu$ resp.\ $K_\nu$ are the modified Bessel functions
  of first resp.\ second kind.
\end{lemma}
\begin{proof}
  In order to calculate the solution $f=f_{0,r,\eta}$ of the
  ODE~\eqref{eq:f.ode}, note that $f$ is a linear combination of
  \begin{align}
    \label{eq:solutions.ode}
    \phi_\nu(s)=s^{-\nu}I_\nu(s) \qquadtext{and}
    \psi_\nu(s)=s^{-\nu}K_\nu(s);
  \end{align}
  this can be seen using
  \begin{subequations}
  \begin{align}
    \label{eq:abramowitz}
    \phi_\nu'(s)&=s^{-\nu}I_{\nu+1}(s),&
    \psi_\nu'(s)&=-s^{-\nu}K_{\nu+1}(s),\\
    \label{eq:abramowitz'}
    (s^\nu I_\nu(s))'&=s^\nu I_{\nu-1}(s),&
    (s^\nu K_\nu(s))'&=-s^\nu K_{\nu-1}(s)
  \end{align}
\end{subequations}
  (see e.g.~\cite[9.6.27--28]{abramowitz-stegun:64}).  A
  straightforward calculation using the boundary conditions give
  \begin{align*}
    f_{r,\eta}'(r)
    = \frac{\phi_\nu'(r)\psi'_\nu(\eta)-\psi_\nu'(r)\phi'_\nu(\eta)}%
    {\phi_\nu(r)\psi'_\nu(\eta)-\psi_\nu(r)\phi'_\nu(\eta)}
    = \frac {-I_{\nu+1}(r)K_{\nu+1}(\eta)+K_{\nu+1}(r)I_{\nu+1}(\eta)}%
    {-I_\nu(r)K_{\nu+1}(\eta)-K_\nu(r)I_{\nu+1}(\eta)}.
  \end{align*}
  The expression for $\OptSobTr(\bd B_r, B_\eta \setminus B_r)$ in
  terms of modified Bessel functions follows again
  from~\eqref{eq:abramowitz}.
\end{proof}

\begin{lemma}
  \label{lem:copt.concrete2}
  We have
  \begin{align*}
    \OptSobTr(\bd B_r, B_r)^2
    = \frac {I_\nu(r)}{I_{\nu+1}(r)}.
  \end{align*}
\end{lemma}
\begin{proof}
  The solution $g=g_r$ is again a linear combination of $\phi_\nu$ and
  $\psi_\nu$ (see~\eqref{eq:solutions.ode}).  But only $\phi_\nu$
  fulfils the integrability condition in~\eqref{eq:int.cond}, so
  together with the first condition $g_r(r)=1$ we obtain
  \begin{align*}
    g_r(s)=\frac{\phi_\nu(s)}{\phi_\nu(r)},
    \qquadtext{hence}
    g_r'(r)=\frac{\phi_\nu'(r)}{\phi_\nu(r)}
    =\frac{I_{\nu+1}(r)}{I_\nu(r)},
  \end{align*}
  where the latter equality follows as before
  by~\eqref{eq:abramowitz}).
\end{proof}

\begin{lemma}
  \label{lem:copt.concrete3}
  We have
  \begin{align}
    \label{eq:copt.concrete3}
    \OptSobTr(\bd B_r,B_\eta)^2
    = \frac{\dfrac{IK^+_\nu(r)}{IK_{\nu+1}(\eta)} + KK^+_\nu(r)}
    {1+\dfrac{IK_{\nu+1}(r)}{IK_\nu(r)}},
  \end{align}
  where
  \begin{align*}
    IK_\nu(r)
    :=\frac{I_\nu(r)}{K_\nu(r)}, \quad
      IK^+_\nu(r)
    :=\frac{I_\nu(r)}{K_{\nu+1}(r)}
    \quadtext{and}
    KK^+_\nu(r)=
    \frac{K_\nu(r)}{K_{\nu+1}(r)}
  \end{align*}
\end{lemma}
\begin{proof}
  We first rewrite the expression for
  $\OptSobTr(\bd B_r,B_\eta \setminus B_r)$ of \Lem{copt.concrete1} in
  terms of the above defined functions as
  \begin{align*}
    \OptSobTr(\bd B_r,B_\eta \setminus B_r)^2
    = \frac{\dfrac{IK^+_\nu(r)}{IK_{\nu+1}(\eta)} + KK^+_\nu(r)}
    {1-\dfrac{IK_{\nu+1}(r)}{IK_{\nu+1}(\eta)}}.
  \end{align*}
  Now
  \begin{align*}
    \OptSobTr(\bd B_r,B_\eta)^{-2}
    &= \OptSobTr(\bd B_r,B_\eta \setminus B_r)^{-2}
    +\OptSobTr(\bd B_r,B_r)^{-2}\\
    &=\frac{1-\dfrac{IK_{\nu+1}(r)}{IK_{\nu+1}(\eta)}
      +\dfrac{I_{\nu+1}(r)}{I_\nu(r)}%
      \Bigl(\dfrac{IK^+_\nu(r)}{IK_{\nu+1}(\eta)} + KK^+_\nu(r)\Bigr)}%
      {\dfrac{IK^+_\nu(r)}{IK_{\nu+1}(\eta)} + KK^+_\nu(r)}\\
    &=\frac{1+\dfrac{IK_{\nu+1}(r)}{IK_\nu(r)}}%
      {\dfrac{IK^+_\nu(r)}{IK_{\nu+1}(\eta)} + KK^+_\nu(r)}
  \end{align*}
  as
  \begin{equation*}
    \frac{I_{\nu+1}(r)}{I_\nu(r)} \cdot IK^+_\nu(r)
    = IK_{\nu+1}(r)
    \quadtext{and}
    \frac{I_{\nu+1}(r)}{I_\nu(r)} \cdot KK^+_\nu(r)
    = \frac{IK_{\nu+1}(r)}{IK_\nu(r)}.\qedhere
  \end{equation*}
\end{proof}

\begin{lemma}
  \label{lem:copt.concrete4}
  There are constants $c_{k,m}$ ($k=1,2,3,4$) and $a_{m-1},b_m>0$ such
  that for all $0<r < \eta < 1$, we have:
  \begin{align*}
    \abs{IK^+_\nu(r)-a_{m-1}r^{m-1}}
    \le c_{1,m} r^{m-1}\cdot r^2[-\log r]_2
  \end{align*}
  and
  \begin{align*}
    \abs{IK_{\nu+1}(r)-b_m r^m}
    \le c_{2,m}r^m \cdot  r^2 [-\log r]_2,
  \end{align*}
  where the term $[\dots]_2$ appears only if $m=2$; in this case we
  restrict to $r \le \e^{-1/2}$ in order to keep $r^2(-\log r)$
  monotonely increasing.  We also have $a_{m-1}/b_m=m$.  Moreover,
  \begin{align*}
    \abs{KK^+_\nu(r)- h_m(r)} \le c_{3,m} r^2
  \end{align*}
  and
  \begin{align}
    \label{eq:copt.concrete4}
    0
    \le \frac{IK_{\nu+1}(r)}{IK_\nu(r)}
    \quadtext{and}
    \Bigabs{\frac{IK_{\nu+1}(r)}{IK_\nu(r)}- \frac{rh_m(r)}2}
    \le c_{4,m} r^3.
  \end{align}
  Moreover, $IK^+_\nu$, $IK_\nu$ and $KK^+_\nu$ are monotonely
  increasing.
\end{lemma}
\begin{proof}
  The claims can be seen using a symbolic computation system.  For the
  monotonicity, we only give a formal proof for $KK^+_\nu$: We have
  \begin{align*}
    KK^+_\nu(r)=
    \frac{K_\nu(r)}{K_{\nu+1}(r)}
    =\frac{\int_0^\infty \e^{-r \cosh t} \cosh(\nu t) \dd t}
    {\int_0^\infty \e^{-r \cosh t} \cosh((\nu+1) t) \dd t}.
  \end{align*}
  We use~\cite[Lem.~9]{qi:22}: if $t \mapsto \partial_r W(t,r)/W(t,r)$
  and $t \mapsto U(t)/V(t)$ are monotonely decreasing on $[0,\infty)$
  then
  $t \mapsto \int_0^\infty U(t) W(t,r) \dd t/\int_0^\infty V(t) W(t,r)
  \dd t$ is monotonely increasing; we apply it with
  $U(t)=\cosh(\nu t)$, $V(t)=\cosh ((\nu+1)t)$ and
  $W(t,r)=\e^{-r\cosh t}$; note that
  $\partial_r W(t,r)/W(t,r)=-\cosh t$ in decreasing and
  $(U/V)'(t)=-t \sinh t/\cosh ((\nu+1)t)^2 \le 0$ for
  $t \in [0,\infty)$.
\end{proof}

\begin{corollary}
  \label{lem:copt.concrete5}
  There is $c_{5,m}>0$ and $\eta_m \in (0,1)$ depending only on $m$
  such that
  \begin{align*}
    \Bigabs{\frac{IK^+_\nu(r)}{IK_{\nu+1}(\eta)}-\frac{mr^{m-1}}{\eta^m}}
    \le \frac{r^{m-1}}{\eta^m} \cdot c_{5,m}  \eta^2[-\log \eta]_2
  \end{align*}
  for $0<\eta<\eta_m$.
\end{corollary}
\begin{proof}
  We have
  \begin{align*}
    \Bigabs{\frac{a+u}{b+v}-\frac ab}
    \le \frac ab \cdot 2\Bigl(\frac{\abs u}a+\frac{\abs v}b \Bigr)
  \end{align*}
  provided $a,b>0$ and $\abs v < b/2$.  Applying this estimate with
  $a=a_{m-1}r^{m-1}$ and $b=b_m\eta^m$, we conclude from
  \Lem{copt.concrete4} that
  \begin{align*}
    \Bigabs{\frac{IK^+_\nu(r)}{IK_{\nu+1}(\eta)}-\frac{mr^{m-1}}{\eta^m}}
    \le \frac{m r^{m-1}}{\eta^m} \cdot 2
    \bigl(c_{1,m}r^2 [-\log r]_2 + c_{2,m}  \eta^2[-\log \eta]_2\bigr)
  \end{align*}
  provided $\abs {c_{2,m}\eta^2[-\log \eta]_2 } < b_m/2$; the latter
  condition is fulfilled for some $\eta<\eta_m$ and $\eta_m$ small
  enough.  If $m=2$, we also have to ensure that
  $\eta_m \le \e^{-1/2}$ in order that $r \mapsto r^2(-\log r)$ is
  monotonely increasing.  The estimate is therefore fulfilled with
  $c_{5,m}=2 m (c_{1,m}+ c_{2,m})$.
\end{proof}

\begin{proof}[Proof of  \Prp{copt.asymp}]
  The proof follows now from \Lem{bd.ball.est3}, and
  \LemS{copt.concrete3}{copt.concrete5} with
  $c_m'=\max \{c_{5,m} c_{3,m}\}$. Note that we estimate the
  denominator of~\eqref{eq:copt.concrete3} by $1$ from above
  using~\eqref{eq:copt.concrete4}.  We do not formally prove the
  monotonicity, as we only need it for the leading terms.
  Nevertheless, function plots of the expression in
  \Lem{copt.concrete3} suggest that $\OptSobTr_m(\cdot,\eta)$ is
  monotonely increasing.
\end{proof}


As a consequence, we have:
\begin{corollary}[estimate on small balls]
  \label{cor:ball.est}
  Let $\eps$ and $\eta$ be positive real numbers such that
  $0 < \eps < \eta$.  Then the optimal constant in~\eqref{eq:bd.ball.est'}
  can be estimated by
  \begin{align}
        \label{eq:c.opt.eucl}
    \OptNonConcEucl_m(\eps,\eta)^2
    =\OptNonConc(B_\eps,B_\eta)^2
    &=\frac{\eps^m}{\eta^m} + H_m(\eps)
    + R_m(r,\eps),
  \end{align}
  where
  \begin{align*}
    H_m(\eps)
    :=\int_0^\eps h_m(r) \dd r
    = \begin{cases}
     \dfrac{\eps^2}2(2+\log 2-\gamma-\log \eps) ,& m=2,\\[1ex]
    =\dfrac {\eps^2}{2(m-2)}, &m \ge 3.
    \end{cases}
  \end{align*}
  Moreover, we have
  depending only on $m$ such that
  \begin{equation*}
    \abs{R_m(\eps,\eta)}
    \le c'_m \Bigl( \frac{\eps^m}{\eta^m} \cdot  \eta^2 [-\log \eta]_2
    +\frac {\eps^3}3\Bigr)
  \end{equation*}
  for all $0 < r < \eta < \eta_m$, where $c_m',\eta_m'>0$ are the
  constants from \Prp{copt.asymp}.  Finally, the function
  $\OptNonConc_m(\cdot,\eta)$ is monotonely increasing for all such
  $\eta$.
\end{corollary}
\begin{proof}
  The claims follow from \Prp{copt.asymp} and \Lem{opt.const.rel}.
  The monotonicity of $\OptNonConc_m(\cdot,\eta)$ is clear from the
  integral formula and the fact that the integrand
  $\OptSobTr_m(\cdot,\eta)$ is non-negative.
\end{proof}

\begin{remark}
  \label{rem:order.of.opt.const}
  Note that
  \begin{align*}
    \OptNonConcEucl_m(\eps,\eta)
    =\Err\Bigl(\frac{\eps^m}{\eta^m}+[-\log \eps]_2
    \eps^2\Bigr)^{1/2}\Bigr)
    \spacetext{1ex}{and}
    \OptSobTrEucl_m(\eps,\eta)
    =\Err\Bigl(\Bigl(\frac{\eps^{m-1}}{\eta^m}+[-\log \eps]_2
    \eps \Bigr)^{1/2}\Bigr),
  \end{align*}
  hence $\OptNonConcEucl_m(\eps,\eta)$ and
  $\eps^{1/2}\OptSobTrEucl_m(\eps,\eta)$ are of same order.  The
  manifold constants $\OptNonConc_m(\eps,\eta)$ and
  $\OptSobTr_m(\eps,\eta)$ are also of the same order as they
  additionally contain only a constant $K \ge 1$ describing the
  deviation of the ball on the manifold from the Euclidean one, see
  \Prp{0} and \Prp{eucl.metric}, and the cover number $N$,
  see~\eqref{eq:uni.loc.bdd}.
\end{remark}

We need the following estimate on $\bd B_r=r\Sphere$:
\begin{lemma}[estimate on fourth root of the Laplacian on the sphere]
  \label{lem:muf}
  We have
  \begin{equation*}
    \norm[\Lsqr{\bd B_r}] {\laplacian {r\Sphere}^{1/4} f}
    \le (m-1)^{1/4} \norm[\Lsqr{B_r}] {df}
  \end{equation*}
  for all $f\in\Sob{B_r}$ and any $r>0$.
\end{lemma}
\begin{proof}
  We argue in four steps:
  \begin{myenumerate}{\arabic}
  \item
    \label{muf.step1}
    By a scaling argument we see that it is sufficient to prove the
    inequality for $r=1$: Indeed,
    \begin{equation*}
      \normsqr[\Lsqr{\bd B_r}] {\laplacian{r\Sphere}^{1/4}f}
      = r^{m-2} \normsqr[\Lsqr{\bd B_1}] {\laplacian \Sphere^{1/4} f}
      \quadtext{and}
      \normsqr[\Lsqr {rB_1}] {df}
      =r^{m-2} \normsqr[\Lsqr {B_1}] {df}.
    \end{equation*}

  \item
    \label{muf.step2}
    Let $\Lambda_{\Sphere}$ be the Dirichlet-to-Neumann operator (at
    the spectral value $0$) on $B_1$ with boundary $\Sphere=\bd B_1$.

    From the definition of $\Lambda_{\Sphere}$ (see also the theory of
    boundary pairs e.g.\ in~\cite[Sec.~2.4]{post:16}), one has
    \begin{equation*}
      \normsqr[\Lsqr \Sphere]{\Lambda_\Sphere^{1/2}\phi}
      = \normsqr[\Lsqr {B_1}]{d h},
    \end{equation*}
    where $h \in \Sob {B_1}$ is the harmonic extension of $\phi$, i.e.,
    $h$ reaches the minimum of
    \begin{equation*}
      \inf \bigset{\normsqr[\Lsqr {B_1}]{d f}}
      {f \in \Sob{B_1}, f \restr {\bd B_1}=\phi},
    \end{equation*}
    hence $\normsqr[\Lsqr {B_1}]{d h} \le \normsqr[\Lsqr {B_1}]{d f}$
    for all $f \in \Sob{B_1}$ with $f \restr {\bd B_1}=\phi$:

    Indeed if $h$ (resp.\ $\phi$) is smooth enough, we have
    $\Lambda_\Sphere \phi=\partial_{\mathrm n} h$, where
    $\partial_{\mathrm n} h$ is the normal outwards derivative of $h$
    on $\bd B_1$.  Green's formula gives (last equality)
    \begin{align*}
      \normsqr{\Lambda_\Sphere^{1/2}\phi}
      = \iprod[\Lsqr \Sphere]{\Lambda_\Sphere\phi} \phi
      = \int_{\Sphere} \partial_{\mathrm n} h \conj h
      =\normsqr[\Lsqr {B_1}] {dh}.
    \end{align*}
    Thus, for any $f\in\Sob{B_1}$, we has
    \begin{equation}
      \label{eq:ineq1}
      \normsqr[\Lsqr \Sphere]{\Lambda_\Sphere^{1/2}(f \restr \Sphere)}
      \leq\normsqr[\Lsqr {B_1}]{d f}.
    \end{equation}

   \item
     \label{muf.step3}
     The Dirichlet-to-Neumann operator $\Lambda_\Sphere$ (at the
     spectral value $0$) and the Laplacian $\laplacian \Sphere$ on the
     sphere are related by the formula
     \begin{align*}
       \Lambda_\Sphere
       =\sqrt{\laplacian \Sphere + \nu^2} - \nu,
       \quadtext{where}
       \qquad \nu = \frac{m-2}2:
     \end{align*}
     Indeed, the expression of $\laplacian {B_1}$ in polar coordinates is
     \begin{equation*}
       \laplacian {B_1}
       = -\partial_r^2-\frac{m-1}r \partial_r + \frac1{r^2}\laplacian \Sphere.
     \end{equation*}
     If $\Delta_\Sphere\phi=\mu^2\phi$ then the harmonic extension of
     $\phi$ on the ball is given by
     $h(r,\theta)=r^\lambda\phi(\theta)$, where $\lambda$ is a solution
     of
     \begin{align*}
       -\lambda(\lambda-1)-(m-1)\lambda+\mu^2
       =-\lambda^2-(m-2)\lambda+\mu^2
       =0
     \end{align*}
     this gives $\lambda=\sqrt{\mu^2+\nu^2}-\nu>0$ (the other solution
     $-\sqrt{\mu^2+ \nu^2}-\nu<0$ leads to a solution with singularity at
     $r=0$, hence is not smooth on $B_1$).  Moreover,
     $\Lambda_\Sphere \phi= \partial_{\mathrm n} h =\partial_r h
     =\lambda\phi$.  Decomposing $\phi$ with respect to the eigenfunctions
     of $\laplacian \Sphere$ gives the formula.

   \item
     \label{muf.step4}
     The lowest eigenvalue of $\Lambda_\Sphere$ is $0$ (constant
     eigenfunction); moreover, the first non-zero eigenvalue
     $\lambda_1$ fulfils $\lambda_1=\sqrt{\mu_1^2+\nu^2}-\nu$ (by the
     previous operator equality; note that $\Lambda_\Sphere$ and
     $\laplacian \Sphere$ are self-adjoint).  As the first non-zero
     eigenvalue of $\Sphere$ is $\mu_1^2=m-1$, we obtain
     $\lambda_1=1$.  In particular, we obtain the operator inequality
     $\Lambda_\Sphere=\lambda_1 \Lambda_\Sphere \le
     \Lambda_\Sphere^2$.

     Moreover, the operator equation from step~\itemref{muf.step3}
     gives
     $\laplacian \Sphere =
     (\Lambda_\Sphere+\nu)^2-\nu^2=(\Lambda_\Sphere+2\nu)\Lambda_\Sphere$,
     hence we obtain the operator inequality
     \begin{align*}
       \laplacian \Sphere
       = \Lambda_\Sphere(\Lambda_\Sphere+2\nu)
       \le (1+2\nu) \Lambda_\Sphere^2
       = (m-1) \Lambda_\Sphere^2.
     \end{align*}
     As both operators commute, we can also use the spectral calculus
     and take the square root of the above inequality.  Hence, for any
     $f\in\Sob{B_1}$ we have
     \begin{align*}
       \normsqr[\Lsqr \Sphere]{\laplacian \Sphere^{1/4} (f \restr{\bd B_1})}
       &=\iprod[\Lsqr \Sphere]{\laplacian \Sphere^{1/2} (f\restr{\bd B_1})}
         {f\restr{\bd B_1}}\\
       &\leq
         \sqrt{m-1}\;
         \iprod[\Lsqr \Sphere]{\Lambda_\Sphere (f \restr{\bd B_1})}
         {f \restr{\bd B_1}}\\
       &= \sqrt{m-1}\;
         \normsqr[\Lsqr \Sphere]{\Lambda_\Sphere^{1/2} f \restr{\bd B_1}}\\
       &\leq\sqrt{m-1} \;\normsqr[\Lsqr {B_1}]{d f},
     \end{align*}
     where we have used~\eqref{eq:ineq1} from step~\itemref{muf.step2}
     for the last inequality.\qedhere
   \end{myenumerate}
\end{proof}

%
\section{Sobolev spaces on manifolds of bounded geometry and harmonic
  charts}
\label{app:mfds.bdd.geo}
%

The Levi-Civita connection $\nabla$ extends to tensors on the manifold
and permits to define the \emph{$k$-th Sobolev space}
$\Sobx[k]{p} {X,g}$: we say that the function $u$ has a \emph{$k$-th
  weak derivative} if there exists a measurable section
$v\in\Lsymb_{1,\loc}(X,(T^*X)^{\otimes k})$ such that for all
$\phi \in \Cci {X,(T^*X)^{\otimes k}}$ (smooth section with compact
support)
\begin{equation*}
  \int_X u \cdot (\nabla^k)^\ast \phi \dd g
  = \int_X (v \cdot \phi) \dd g.
\end{equation*}
Note that $v$ is uniquely determined by $u$ and denoted by
$\nabla^k u$.  We set
\begin{gather*}
  \Sobx[k]{p} {X,g} \coloneqq
  \bigset{u \in \Lp[p] {X,g}}
  {\forall j \in \{1,\dots,k\} \; \exists \nabla^j u \in \Lp[p] {X,g}}
  \intertext{with norm given by}
  \norm[{\Sobx[k]{p} {X,g}}] u ^p
  \coloneqq \sum_{j=0}^k \norm[{\Lp[p] {T^*X^{\otimes j},g}}] {\nabla^j u}^p
\end{gather*}
for $p\geq 1$, and $\Sob[k]{X,g} \coloneqq\Sobx[k]{2} {X,g}$.
Obviously, the space $\Sob{X,g}$ agrees with the domain of the energy
form defined in \Subsec{form.laplacian} as the corresponding norms
agree.  For higher orders, the equivalence of Sobolev spaces and
Laplacian-Sobolev spaces on non-compact manifolds is not guaranteed
without further assumptions:
\begin{definition}[manifold of bounded geometry]
  \label{def:bdd.geo}
  We say that a complete Riemannian manifold $(X,g)$ has \emph{bounded
    geometry} if the injectivity radius is uniformly bounded from
  below by some constant $\iota_0>0$ and if the Ricci tensor $\Ric$ is
  uniformly bounded from below by some constant $\kappa_0 \in \R$, i.e.,
  \begin{equation*}
    \Ric_x\geq \kappa_0 g_x \qquad\text{for all $x \in X$}
  \end{equation*}
  as symmetric $2$-tensors.
\end{definition}
We will not need assumptions on \emph{derivatives} of the curvature
tensor (i.e., bounded geometry of higher order) in this article; from
the next result, we only need the explicit constant $\Cellreg$ in the
upper bound; for a proof we refer to~\cite[Prp.~2.10]{hebey:96}
or~\cite[Prp.~3.3]{anne-post:21}:
\begin{proposition}[equivalence of second order Sobolev norms]
  \label{prp:ell.reg}
  Suppose that $(X,g)$ is a complete manifold with bounded geometry,
  then the set of smooth functions with compact support
  is dense in the Sobolev space $\Sob[2]{X,g}$ and the norms of
  $\Sob[2]{X,g}$ and $\Sob[2]{\laplacian{(X,g)}}$ are equivalent,
  i.e., there are constants $\Cellreg \geq \cellreg>0$ such that
  \begin{equation*}
    \cellreg \norm[\Lsqr{X,g}] {(\laplacian{(X,g)}+1)f}
    \le \norm[{\Sob[2]{X,g}}] f
    \le \Cellreg \norm[\Lsqr{X,g}] {(\laplacian{(X,g)}+1)f}
  \end{equation*}
  for all $f \in \Sob[2]{X,g}$, where
  $\Cellreg=2+\max\{0,-\kappa_0\}$
  depends only on a lower bound $\kappa_0$ of the Ricci curvature.

  In particular, if the manifold has bounded geometry, we conclude the
  equality of the spaces $\Sob[2]{X,g}$ and
  $\Sob[2]{\laplacian{(X,g)}}$ and equivalence of their norms.
\end{proposition}

In the following, we need the (open) geodesic ball of radius $r>0$
around $p \in X$ denoted by
\begin{equation*}
  B_r(p) \coloneqq \set{x \in X}{d_g(x,p)<r},
\end{equation*}
where $d_g$ is the metric (distance function) induced by the
Riemannian metric $g$ on $X$.  Recall that a chart
$\map{\phi=(y^1,\dots,y^m)}U{\R^m}$ for an open subset $U \subset X$
is called \emph{harmonic} if each coordinate function $y^j$ is
harmonic, i.e., if $\laplacian{(U,g)} y^j=0$.
\begin{proposition}[{\cite[Th.~1.3]{hebey:96}}]
  \label{prp:eucl.metric}
  Assume that $(X,g)$ is complete and has bounded geometry (with
  constants $\kappa_0\in\R$ and $\iota_0>0$).  Then for all
  $a \in (0,1)$ there exist $r_0>0$, $K\ge 1$ and $k>0$ depending only
  on $\kappa_0$, $\iota_0$ and $a$, such that around any point
  $p \in X$ there exist harmonic charts $\phi_p=(y^1,\dots,y^m)$
  defined on $\clo{B_{r_0}(p)}$, and in these charts we have
  $g_{ij}(p)=\delta_{ij}$ and\footnote{When stressing the dependence
    of $g_x$ on $x \in X$, we also write $g(x)$ or $g_{ij}(x)$ and
    similarly for related
    objects.}
  \begin{subequations}
    \begin{gather}
      \label{eq:eucl.metric.b}
      \abs{g_{ij}(x)-g_{ij}(x')}\leq k\; d_g(x,x')^a,\\
      \label{eq:eucl.metric.a}
      K^{-1}\delta_{ij}\leq g_{ij}\leq K\;\delta_{ij}
    \end{gather}
  \end{subequations}
  for all $x, x' \in B_{r_0}(p)$.
\end{proposition}

\begin{remarks}
  \indent
  \begin{enumerate}
  \item The supremum of all radii $r_0$ such that \Prp{eucl.metric}
    holds will be called \emph{harmonic radius} in the following.  We
    refer to~\cite{hebey:96,hpw:14} and references therein for more
    details.  We assume $r_0 \le 1$ here, as it simplifies some
    estimates later on.
  \item Estimate~\eqref{eq:eucl.metric.a} (and its
    \Cors{eucl.metric}{eucl.metric2} below) means that the manifold
    has bounded geometry in the sense of Grigor'yan
    \cite[Ex.~11.12]{grigoryan:09}: there exist $\eta>0$ and $C>0$
    such that at each point $p\in X$ there exists a diffeomorphism
    between the geodesic ball of radius $\eta$ around $p$ and the
    Euclidean ball of radius $\eta$ in $\R^m$ and this diffeomorphism
    changes the metric and the measure at most by a factor $C$.  In
    particular, this property implies a lower bound of the injectivity
    radius.
  \end{enumerate}
\end{remarks}

Denote by $g_{\eucl,p}$ the Euclidean metric in the harmonic chart
$\phi_p$ defined in the geodesic ball $B\coloneqq B_{r_0}(p)$, i.e.,
\begin{equation}
  \label{eq:eucl.met}
  g_{\eucl,p}(x)(\partial_{y_i},\partial_{y_j})=\delta_{ij}
\end{equation}
for $x \in B$.  Moreover, let
$x \mapsto \map{A_{\eucl,p}(x)}{T_x B}{T_xB}$ ($x \in B$) be defined
via
\begin{equation*}
  g_{\eucl,p}(x)(\xi,\xi)=g(x)\bigl(A_{\eucl,p}(x)\xi,\xi\bigr).
\end{equation*}
i.e., the endomorphism $A_{\eucl,p}$ measures how far the Euclidean
metric on $B$ is from the original one.  If the dependence of
$g_{\eucl,p}$ and $A_{\eucl,p}$ on $p \in X$ is clear, we also write
$g_\eucl$ and $A_\eucl$.  Estimate~\eqref{eq:eucl.metric.b} now reads
as:
\begin{corollary}
  \label{cor:eucl.metric}
  For $\eps \in (0,r_0)$ we have
  \begin{align*}
   \norm[\Cont{B_\eps(p)}]{A_\eucl - \id_{TB}}
    &\coloneqq \sup_{x \in B_\eps(p)}
      \norm[\BdOp{T_x B,g_x}]{A_\eucl(x)-\id_{T_xB}}\\
    &=  \sup_{x \in B_\eps(p)} \norm[\BdOp{\R^m}]{E_m-G(x)}
   \le m k \eps^a,
  \end{align*}
  where $k$ and $a$ are as in \Prp{eucl.metric}, $\norm[\BdOp{\R^m}]B$
  is the operator norm of $\map B{\R^m}{\R^m}$ with its standard
  Euclidean norm on $\R^m$, $E_m$ is the $(m\times m)$-unit matrix and
  $G(x)=(g_{ij}(x))_{i,j=1,\dots,m}$.
\end{corollary}
\begin{proof}
  From~\eqref{eq:eucl.metric.b} we conclude
  \begin{equation*}
    -k \eps^a
    \le g_{ij}(p)-g_{ij}(x)
    \le k \eps^a
  \end{equation*}
  for $x \in B_\eps(p)$.  Moreover, $g_{ij}(p)=\delta_{ij}$, hence
  $\norm[\infty]{E_m-G(x)} \le k \eps^a$ where $E_m$ is the
  $(m \times m)$-unit matrix and $G(x)=(g_{ij}(x))_{i,j}$ the matrix
  representation of $g(x)$ and $\norm[\infty]\cdot$ is the
  component-wise maximum norm.  Moreover,
  $\norm[\R^m \to \R^m]{E_m-G(x)}\le m \norm[\infty]{E_m-G(x)}$.
  Passing to an orthonormal basis of $(T_xB,g_x)$ we see that
  \begin{equation*}
    \norm[\BdOp{\R^m}]{E_n-G(x)}=\norm[\BdOp{T_xB,g_x}]{A_\eucl(x)-\id_{T_xB}}.
    \qedhere
  \end{equation*}
\end{proof}

Using \Prp{eucl.metric} we can compare the metric on a ball with the
Euclidean one:
\begin{corollary}
  \label{cor:eucl.metric2}
  Let $p \in X$, $r \in (0,r_0)$ and $B_r \coloneqq B_r(p)$, then we have the
  following norm estimates
  \begin{align*}
    K^{-m/4} \norm[\Lsqr{\bd B_r, \iota_r^*g_\eucl}] \phi
    &\le \norm[\Lsqr{\bd B_r,\iota_r^*g}] \phi
    \le K^{m/4} \norm[\Lsqr{\bd B_r, \iota_r^*g_\eucl}] \phi\\
    K^{-m/4} \norm[\Lsqr{B_r,g_\eucl}] u
    &\le \norm[\Lsqr{B_r,g}] u
      \le K^{m/4} \norm[\Lsqr{B_r,g_\eucl}] u,\\
    K^{-(m+2)/4} \norm[\Lsqr{T^*B_r,g_\eucl}] {du}
    &\le \norm[\Lsqr{T^*B_r,g}] {du}
      \le K^{(m+2)/4} \norm[\Lsqr{T^*B_r,g_\eucl}] {du},\\
    K^{-(m+2)/4} \norm[\Sob{B_r,g_\eucl}] u
    &\le \norm[\Sob{B_r,g}] u
      \le  K^{(m+2)/4} \norm[\Sob{B_r,g_\eucl}] u
  \end{align*}
  for all $\phi \in \Lsqr{\bd B_r, \iota_r^*g}$, $u \in \Lsqr{B_r,g}$
  resp.\ $u \in \Sob{B_r,g}$.  Here $\embmap {\iota_r}{\bd B_r} {B_r}$
  denotes the natural embedding, and $\iota_r^* g$ resp.\
  $\iota_r^* g_\eucl$ the induced metric on $\bd B_r$.
\end{corollary}
\begin{proof}
  Let $(r,\theta) \in (0,r_0) \times \Sphere$ be polar coordinates
  around $p \in X$, then we can write $g=\dd r^2 + \iota_r^* g$.
  Moreover, we have $g_\eucl=\dd r^2 + \iota_r^* g_\eucl$ and
  $\iota_r^* g_\eucl=r^2 g_\Sphere$ is the (scaled) standard metric on
  $\Sphere$.  From~\eqref{eq:eucl.metric.a} we conclude
  \begin{align*}
    K^{-m/2} r^{m-1} \dd g_\Sphere
    = K^{-m/2} \dd (\iota_r^* g_\eucl)
    &\le \dd (\iota_r^* g)
    \le K^{m/2} \dd (\iota_r^* g_\eucl)
    = K^{m/2} r^{m-1} \dd g_\Sphere \quad\text{and}\\
    K^{-m/2} \dd g_\eucl
    = K^{-m/2} r^{m-1} \dd r \dd g_\Sphere
    &\le \dd g
    \le K^{m/2} r^{m-1} \dd r \dd g_\Sphere
    = K^{m/2} \dd g_\eucl
  \end{align*}
  for the $(m-1)$- resp.\ $m$-dimensional volume measure and
  \begin{equation*}
    K^{-1}\abssqr[g_\eucl]{\xi}
    \le \abssqr[g_x]{\xi}
    \le K  \abssqr[g_\eucl]{\xi}
  \end{equation*}
  for the squared norm on the cotangent bundle.  The claimed estimates
  follow now from the definition of the norms.
\end{proof}

\section{Norm convergence of operators on varying Hilbert spaces}
\label{app:main.tool}
We briefly define here the concept of \emph{quasi-unitary
  equivalence}, which gives a notion of ``distance'' between two
closed non-negative quadratic forms.  For more details, we refer
to~\cite{khrabustovskyi-post:18, %
  anne-post:21} or the monograph~\cite{post:12}.  Moreover, we define
a general framework which assures a \emph{generalised norm resolvent
  convergence} for operators $\Delta_\eps$ converging to $\Delta_0$ as
$\eps \to 0$.  For a comparison of this concept and a similar one
introduced by Weidmann~\cite[Sec.~9.3]{weidmann:00}, we refer
to~\cite{post-zimmer:22} and~\cite{post-zimmer:pre24}.  Here, each
operator $\Delta_\eps$ acts in a Hilbert space $\HS_\eps$ for
$\eps \ge 0$; and the Hilbert spaces are allowed to depend on $\eps$.
In typical applications, the Hilbert spaces $\HS_\eps$ are of the form
$\Lsqr{X_\eps}$ for some metric measure space $X_\eps$ which is
considered as a perturbation of a ``limit'' metric measure space
$X_0$; and typically, there is a topological transition between
$\eps>0$ and $\eps=0$.

In order to define the convergence, we define a sort of ``distance''
$\delta=\delta_\eps$ between $\wt \Delta\coloneqq\Delta_\eps$ and
$\Delta\coloneqq\Delta_0$, in the sense that if $\delta_\eps \to 0$ then
$\Delta_\eps$ converges to $\Delta_0$ in the above-mentioned
generalised norm resolvent sense.

Let $\HS$ and $\wt \HS$ be two separable Hilbert spaces.  We say that
$(\qf d,\HS^1)$ is an \emph{energy form in $\HS$} if $\qf d$ is a
closed, non-negative quadratic form in $\HS$ with domain $\HS^1$, i.e., if
$\qf d(f)\coloneqq\qf d(f,f)\ge 0$ for some sesquilinear form
$\map {\qf d}{\HS^1 \times \HS^1}\C$, denoted by the same symbol, with
$\HS^1\eqqcolon\dom \qf d$ endowed with the norm defined by
\begin{equation}
  \label{eq:qf.norm}
  \normsqr[1] f
  \coloneqq \normsqr[\HS^1] f
  \coloneqq \normsqr[\HS] f + \qf d(f),
\end{equation}
so $\HS^1$ is itself a Hilbert space and a dense set in $\HS$.  We call the
corresponding non-negative, self-adjoint operator by $\Delta$ (see
e.g.~\cite[Sec.~VI.2]{kato:66}) the \emph{energy operator} associated
with $(\qf d,\HS^1)$.  Similarly, let $(\wt {\qf d},\wt \HS^1)$ be an
energy form in $\wt \HS$ with energy operator $\wt \Delta$.

Associated with an energy operator $\Delta$, we can define a natural
\emph{scale of Hilbert spaces} $\HS^k$ defined via the \emph{abstract
  Sobolev norms}
\begin{equation*}
  \norm[\HS^k] f
  \coloneqq \norm[k] f
  \coloneqq \norm{(\Delta+1)^{k/2}f}.
\end{equation*}
Here, we only need the cases $k \in \{0,1,2\}$.  In particular,
$\HS^0=\HS$, $\HS^1=\dom \qf d$ and the norms
$\norm[\HS^1] f=\norm[1] f$ agree (the latter in the sense
of~\eqref{eq:qf.norm}).  Finally, $\HS^2=\dom \Delta$ and
$\norm[\HS^2] f$ is equivalent with the graph norm of $\Delta$
(see~\cite[Sec.~3.2]{post:12} for details).  Similarly, we denote by
$\wt \HS^k$ the scale of Hilbert spaces associated with $\wt \Delta$.

We denote by $\spec \Delta$ the spectrum of the energy operator and by
$R=(\Delta+1)^{-1}$ its resolvent in $-1$; we use similar notations
for $\wt \Delta$.

We now need a pair of so-called \emph{identification} or
\emph{transplantation operators} acting on the Hilbert spaces and also
a pair of identification operators acting on the form domains.  Note
that our definition is slightly more general than the ones
in~\cite[Secs. 4.1, 4.2 and 4.4]{post:12}.  The new point here is that
we allow the (somehow ``smoothing'') resolvent power of order $k/2$ on
the right hand side in~\eqref{eq:quasi-uni.d} also for $k>0$.
\begin{definition}[quasi-unitary equivalence of forms]
  \label{def:quasi-uni}
  \begin{subequations}
    Let $\delta \ge 0$, and let $\map J \HS {\wt \HS}$ and
    $\map {J'}{\wt \HS}\HS$ be linear bounded operators.  Moreover,
    let $\map {J^1} {\HS^1} {\wt \HS^1}$ and $\map
    {J^{\prime1}} {\wt \HS^1}{\HS^1}$ be linear bounded operators
    on the energy form domains.
    \begin{enumerate}
    \item We say that $J$ is \emph{$\delta$-quasi-unitary} with
      \emph{$\delta$-quasi-adjoint} $J'$ if
      \begin{gather}
        \label{eq:quasi-uni.a}
        \norm{Jf}\le (1+\delta) \norm f, \quad
        \bigabs{\iprod {J f} u - \iprod f {J' u}}
        \le \delta \norm f \norm u
        \qquad (f \in \HS, u \in \wt \HS),\\
        \label{eq:quasi-uni.b}
        \norm{f - J'Jf}
        \le \delta \norm[1] f, \quad
        \norm{u - JJ'u}
        \le \delta \norm[1] u \qquad (f \in \HS^1, u \in \wt \HS^1).
      \end{gather}

    \item We say that $J^1$ and $J^{\prime1}$ are \emph{$\delta$-compatible}
      with the identification operators $J$ and $J'$ if
      \begin{equation}
        \label{eq:quasi-uni.c}
        \norm{J^1f - Jf}\le \delta \norm[1]f, \quad
        \norm{J^{\prime1}u - J'u} \le \delta \norm[1] u
        \qquad (f \in \HS^1, u \in \wt \HS^1).
      \end{equation}

    \item We say that the energy forms $\qf d$ and $\wt {\qf d}$ are
      \emph{$\delta$-close (of order $2$)} if
      \begin{equation}
        \label{eq:quasi-uni.d}
        \bigabs{\wt{\qf d}(J^1f, u) - \qf d(f, J^{\prime1}u)}
        \le \delta \norm[2] f \norm[1] u
        \qquad (f \in \HS^2, u \in \wt \HS^1).
      \end{equation}

    \item We say that $\qf d$ and $\wt {\qf d}$ (or sometimes that
      $\Delta$ and $\wt \Delta$) are \emph{$\delta$-quasi unitarily
      equivalent (of order $2$)},
      if~\eqref{eq:quasi-uni.a}--\eqref{eq:quasi-uni.d} are fulfilled,
      i.e.,
      \begin{itemize}
      \item if there exists identification operators $J$ and $J'$ such
        that $J$ is $\delta$-quasi unitary with $\delta$-adjoint $J'$
        (i.e., \eqref{eq:quasi-uni.a}--\eqref{eq:quasi-uni.b} hold);
      \item if there exists identification operators $J^1$ and
        $J^{\prime1}$ which are $\delta$-compatible with $J$ and $J'$
        (i.e., \eqref{eq:quasi-uni.c} holds);
      \item and if $\qf d$ and $\wt {\qf d}$ are $\delta$-close (of
        order $k$) (i.e., \eqref{eq:quasi-uni.d} holds).
      \end{itemize}
    \end{enumerate}
  \end{subequations}
\end{definition}

\begin{remark}
  \label{rem:why.order.2}
  There is an asymmetry in the order in~\eqref{eq:quasi-uni.d}, namely
  $f$ appears in the \emph{graph norm} $\norm[2]f$ on the right hand
  side, while $u$ appears only in the \emph{quadratic form norm}
  $\norm[1] u$; for a discussion of this point we refer
  to~\cite[Rem.~2.7]{anne-post:21}; moreover we believe that ---
  similarly as in~\cite[Rem.~4.4]{anne-post:21} --- an estimate as
  in~\eqref{eq:quasi-uni.d} with $\norm[2]f$ replaced by the quadratic
  form norm $\norm[1]f$ would not be enough.  Nevertheless, higher
  orders $\norm[k] f$ for $k\ge 3$ are possible, but then
  estimate~\eqref{eq:prp.quasi-uni.a} in the next proposition is only
  true with $JR -\wt R J$ replaced by $(JR - \wt R J)R^{k-2}$.
\end{remark}
The property of quasi-unitary equivalence for quadratic forms give convergence
of the related operators:
\begin{proposition}[{resolvent estimate, cf.~\cite[Prps.~2.2 and~2.5]{anne-post:21}}]
  \label{prp:quasi-uni}
  Let $\qf d$ and $\wt {\qf d}$ be $\delta$-qua\-si unitarily equivalent
  (of order $2$), then the following holds true:
  \begin{subequations}
    \begin{gather}
      \label{eq:prp.quasi-uni.a}
      \bignorm{J R - \wt R J}
      \le 7\delta \qquad\text{and}\\
      \label{eq:prp.quasi-uni.b}
      \bignorm{(\wt \Delta+1)^{1/2}\bigl(J^1 R - \wt R J\bigr)}
      \le 6\delta.
    \end{gather}
  \end{subequations}
\end{proposition}

One can define a slightly different version of quasi-unitary
equivalence (see~\cite[Sec.~5]{post-zimmer:pre24}).  If the
identification operator $J$ is a partial isometry (an equivalent
characterisation is $JJ^*J=J$, see~\cite[Lem.~2.1]{post-zimmer:22}),
then our definition \Def{quasi-uni} implies the new one.  We now use
the unitary equivalence concept to define convergence and conclude:
\begin{theorem}[spectral convergence]
  \label{thm:spectrum}
  Assume that $(\qf d_n)_n$ is a sequence of energy forms such that
  each $\qf d_n$ is $\delta_n$-quasi unitarily equivalent with partial
  isometric identification operators, then
  \begin{equation}
    d_{\mathrm{Hausd}}(\spec[\bullet] {(\Delta_n+1)^{-1}},
    \spec[\bullet] {(\Delta+1)^{-1}}
    \le 7 \sqrt 3 \cdot \delta_n
  \end{equation}
  Here, $d_{\mathrm{Hausd}}$ is the Hausdorff distance of two sets,
  $\spec[\bullet] \Delta$ stands for the entire spectrum or the
  essential spectrum of $\Delta$ and similarly for $\Delta_n$.

  If $\lambda \in \spec[disc] \Delta$ is an eigenvalue of multiplicity
  $\mu>0$, then there exist $\mu$ eigenvalues (not necessarily all
  distinct) $\lambda_{n,j}$, $j=1 \dots \mu$, such that
  $\lambda_{n,j} \to \lambda$ as $n \to \infty$. (with convergence
  speed $\delta_n$ as above).  In particular, if $\mu=1$ and if
  $\psi \in \HS$ is the corresponding normalised eigenvector, then
  there exists a sequence of normalised eigenvectors $\psi_n$ of
  $\Delta_n$ such that
  \begin{equation}
    \label{eq:ef.est}
    \norm{J \psi - \psi_n} \le c \delta_n, \quad
    \norm{J' \psi_n - \psi} \le c'\delta_n, \quad
    \norm[1]{J^1 \psi - \psi_n} \le c_1' \delta_n.
  \end{equation}
  where $c$,$c'$ and $c_1'>0$ are universal constants depending only
  on $\lambda$ and the distance of $\lambda$ from the remaining
  spectrum of $\Delta$.
\end{theorem}
\begin{proof}
  Is is shown in~\cite[Thm.~2.12, Thm.~3.9]{post-zimmer:22} that
  quasi-unitary equivalence as in \Def{quasi-uni} implies a
  generalised resolvent convergence defined by
  Weidmann~\cite[Sec.~9.3]{weidmann:00}.  If the identification
  operator $J$ is a partial isometry, then the convergence speed is
  the same, and one has a slightly stronger version of quasi-unitary
  equivalence (see~\cite[Sec.~5]{post-zimmer:pre24}).  Using
  now~\cite[Thm.~F]{post-zimmer:pre24} we conclude that the Hausdorff
  distance of the resolvent spectra can be estimated by $\sqrt 3$
  times the resolvent difference~\eqref{eq:prp.quasi-uni.a}
  (see~\cite[Thm.~C]{post-zimmer:pre24}).  The energy norm estimate on
  eigenfunctions in~\eqref{eq:ef.est} follows from~\cite[Thm.~2.6]{post-simmer:18}
\end{proof}

%
%

\providecommand{\bysame}{\leavevmode\hbox to3em{\hrulefill}\thinspace}
\providecommand{\MR}{\relax\ifhmode\unskip\space\fi MR }
\providecommand{\MRhref}[2]{%
  \href{http://www.ams.org/mathscinet-getitem?mr=#1}{#2}
}
\providecommand{\href}[2]{#2}

\end{document}